\documentclass[11pt,oneside,english,reqno]{amsart}
\usepackage[T1]{fontenc}
\usepackage[utf8]{inputenc}
\usepackage[a4paper]{geometry}
\usepackage{amstext}
\usepackage{amsthm}
\usepackage{amssymb}
\usepackage{stmaryrd}   %

\usepackage{graphicx}
\usepackage{caption}
\usepackage{subcaption}
\usepackage{algorithm}
\usepackage{algorithmic}

\usepackage[dvipsnames]{xcolor}
\usepackage{enumitem}
\usepackage{verbatim}   %
\usepackage{shuffle}    %

\usepackage[colorlinks,linkcolor=blue,urlcolor=blue,citecolor=blue,
            pagebackref,hypertexnames=true,breaklinks]{hyperref}
\usepackage{aliascnt}
\usepackage{cleveref}

\numberwithin{equation}{section}

\theoremstyle{plain}
\newtheorem{thm}{Theorem}[section]
\newaliascnt{lem}{thm}  \newtheorem{lem}[lem]{Lemma}        \aliascntresetthe{lem}
\newaliascnt{prop}{thm} \newtheorem{prop}[prop]{Proposition} \aliascntresetthe{prop}
\newaliascnt{cor}{thm}  \newtheorem{cor}[cor]{Corollary}     \aliascntresetthe{cor}
\newaliascnt{example}{thm} \newtheorem{example}[example]{Example} \aliascntresetthe{example}
\newaliascnt{hyp}{thm}  \newtheorem{hyp}[hyp]{Hypothesis}   \aliascntresetthe{hyp}

\theoremstyle{definition}
\newaliascnt{defn}{thm} \newtheorem{defn}[defn]{Definition}  \aliascntresetthe{defn}
\newaliascnt{nota}{thm} \newtheorem{nota}[nota]{Notation}    \aliascntresetthe{nota}

\theoremstyle{remark}
\newaliascnt{rem}{thm}  \newtheorem{rem}[rem]{Remark}        \aliascntresetthe{rem}

\crefname{thm}{theorem}{theorems}         \Crefname{thm}{Theorem}{Theorems}
\crefname{lem}{lemma}{lemmas}             \Crefname{lem}{Lemma}{Lemmas}
\crefname{prop}{proposition}{propositions}\Crefname{prop}{Proposition}{Propositions}
\crefname{cor}{corollary}{corollaries}    \Crefname{cor}{Corollary}{Corollaries}
\crefname{example}{example}{examples}     \Crefname{example}{Example}{Examples}
\crefname{hyp}{hypothesis}{hypotheses}    \Crefname{hyp}{Hypothesis}{Hypotheses}
\crefname{defn}{definition}{definitions}  \Crefname{defn}{Definition}{Definitions}
\crefname{nota}{notation}{notations}      \Crefname{nota}{Notation}{Notations}
\crefname{rem}{remark}{remarks}           \Crefname{rem}{Remark}{Remarks}

\makeatletter
\renewcommand{\l@section}{\@tocline{1}{0pt}{0em}{1.5em}{}}
\renewcommand{\l@subsection}{\@tocline{2}{0pt}{2em}{3em}{}}
\renewcommand{\l@subsubsection}{\@tocline{3}{0pt}{4em}{5em}{}}
\makeatother

\newcommand{\cA}{\mathcal{A}}  \newcommand{\cC}{\mathcal{C}}
\newcommand{\cD}{\mathcal{D}}  \newcommand{\cE}{\mathcal{E}}
\newcommand{\cF}{\mathcal{F}}  \newcommand{\cG}{\mathcal{G}}
\newcommand{\cH}{\mathcal{H}}  \newcommand{\cI}{\mathcal{I}}
\newcommand{\cJ}{\mathcal{J}}  \newcommand{\cK}{\mathcal{K}}
\newcommand{\cL}{\mathcal{L}}  \newcommand{\cM}{\mathcal{M}}
\newcommand{\cP}{\mathcal{P}}  \newcommand{\cW}{\mathcal{W}}

\newcommand{\CC}{\mathbb{C}}   \newcommand{\NN}{\mathbb{N}}   \newcommand{\RR}{\mathbb{R}}

   \newcommand{\bC}{\mathbf{C}}   \newcommand{\bF}{\mathbf{F}}
\newcommand{\bK}{\mathbf{K}}   \newcommand{\bM}{\mathbf{M}}   \newcommand{\bS}{\mathbf{S}}
\newcommand{\bZ}{\mathbf{Z}}   \newcommand{\bv}{\mathbf{v}}   \newcommand{\bx}{\mathbf{x}}
\newcommand{\by}{\mathbf{y}}   \newcommand{\bz}{\mathbf{z}}

\newcommand{\dd}{\mathop{}\!\mathrm{d}}
\newcommand{\la}{\langle}
\newcommand{\ra}{\rangle}

\newcommand{\Sig}[1]{\mathrm{Sig}(#1)}
\newcommand{\VSig}[2]{\mathrm{VSig}(#1;#2)}
\newcommand{\Lkernel}{L^{\infty,1}(\Delta^2; \mathcal{L}(\RR^d;\RR^m))}
\newcommand{\decoSet}{\mathcal{B}}
\newcommand{\decoNum}{|\decoSet|-1}
\newcommand{\decoNumPlus}{|\decoSet|}

\begin{document}

\title[Computational aspects of the Volterra Signature]{Computational aspects of the Volterra Signature}

\author{Paul P. Hager}
\address{Paul P. Hager, Department of Statistics and Operations Research, University of Vienna,
Kolingasse 14--16, 1090 Vienna, Austria}
\email{paul.peter.hager@univie.ac.at}

\author{Fabian N. Harang}
\address{Fabian N. Harang, Department of Economics, BI Norwegian Business School, Nydalsveien 37, 0484 Oslo, Norway}
\email{fabian.a.harang@bi.no}

\author{Luca Pelizzari}
\address{Luca Pelizzari, Department of Statistics and Operations Research, University of Vienna,
Kolingasse 14--16, 1090 Vienna, Austria}
\email{luca.pelizzari@univie.ac.at}

\author{Samy Tindel}
\address{Samy Tindel, Department of Mathematics, Purdue University,
150 N. University Street, West Lafayette, IN 47907--2067, USA}
\email{stindel@purdue.edu}
\date{\today}

\begin{abstract}
The Volterra signature extends the classical path signature by incorporating general
matrix-valued kernel into its iterated integral structure, yielding a flexible
notion of memory for time series. Its components can be viewed as successive Picard iterates of linear controlled
Volterra equations, making their exact computation of additional mathematical
interest.
However, the kernel
introduces substantial algorithmic challenges. We provide a resolution by first
decomposing the Chen-type convolution relation established in~\cite{i_part}
into analytic and arithmetic parts, and then introducing several efficient
algorithms:
a general approximative scheme with quadratic complexity $O(J^2)$ in the number of time steps $J$, an FFT-based acceleration with complexity $O(J\log J)$ for convolution kernels on uniform grids, and an exact recursion with complexity $O(JR^2)$ for kernels admitting a state-space representation of dimension $R$; retaining standard signature complexity in the path dimension and truncation level $N$.
We further show that the number of factors in matrix-valued kernels of the form $K(t,s)=\sum_p k_p(t-s)A_p$ do not increase the asymptotic complexity in $J$ and $N$.
Finally, we derive a finite-difference predictor--corrector scheme for the associated Volterra signature kernel.
All algorithms are implemented in the publicly available JAX-based package \texttt{tensordev}.
\end{abstract}

\keywords{Signatures, machine learning, memory effects, Volterra equations, kernel learning, rough analysis, numerical algorithms}
\subjclass[2020]{Primary: 60L10, 45D05; Secondary: 60L70, 65R20, 65T50.}
\thanks{FH was supported by the SURE-AI Centre grant 357482, Research Council
of Norway.}

\setcounter{tocdepth}{2}
\maketitle
\tableofcontents

\section{Introduction}
Iterated path integrals, first rigorously studied by Chen \cite{Chen1957} and later developed by Lyons \cite{Lyons1998} into a cornerstone of rough path analysis before being branded as the path {\em signature}, have been established as a highly efficient feature map for analysis of sequential data. 
As such, signatures yield a universal graded representation of path features with widespread use in machine learning across many data structures, from geometric trajectory data (e.g.\ pen strokes) through irregular event records (e.g.\ health care records) to high-frequency financial streams such as prices and volatility; see e.g.\ \cite{BayerDosReisHorvathOberhauser2025SignatureMethodsInFinance,McLeodLyons2025SignatureMethodsML,chevyrev2016primer} and references therein.

The signature is defined on an analytical level for a path $x:[0,T]\to \RR^d$ with continuous derivative $\dot x$ by the fundamental linear equation
\begin{equation}\label{eq:intro_sig_fundamental}
    \Sig{x}_{s,t} \;=\; 1+ \int_{s}^t \Sig{x}_{s,u} \otimes \dot{x}_u \dd{u},
\qquad 0\le s\le t\le T,
\end{equation}
where $\otimes$ denotes the tensor product in $T((\RR^d)) := \prod_{n=0}^\infty(\RR^d)^{\otimes k}$, hence understanding $1=(1,0,\dots)\in T((\RR^d))$.
This equation is only formally implicit; solving it tensor level-wise the signature directly expands into iterated integrals of $x$ with itself (to be presented further below).
The property that makes the computation of signatures particularly convenient is Chen's relation for joining signatures on adjacent path segments:
\begin{equation}\label{eq:sig_chen_intro}
    \Sig{x}_{s,t} = \Sig{x}_{s,u}\otimes \Sig{x}_{u,t}, \qquad 0 \le s \le u\le t \le T.
\end{equation}

While the signature keeps track of the entire history of the path, i.e.\ %
it separates paths up to so-called tree-like equivalence and time reparameterization \cite{HamLy10}, it does so without discriminating between recent and distant changes in $x$.
In applications one is often interested in introducing explicit control over memory effects.
A very elementary way of introducing such structure, at least going back to Volterra \cite{volterra1909sulle,volterra1913leccons}, is via a kernel $$K:\Delta^2 = \{ (t,s) \;\vert\; 0 \le s \le t \le T\} \to \mathcal{L}(\RR^d;\RR^m),$$
where memory effects act on different scales: the influence of distant history is governed by the decay of $K$ away from the diagonal (e.g.\ exponential versus power-law decay), whereas the influence of recent history is governed by the behavior of $K$ near the diagonal (e.g.\ fractional-type), and different effects such as periodic/oscillatory weighting can be reproduced as well.
The pressing question is therefore \emph{not why} but rather \emph{how} to incorporate a kernel into the signature \eqref{eq:intro_sig_fundamental}.

Recently, in \cite{i_part} a general approach was introduced based on earlier work on Volterra rough paths \cite{HT21} via the following fundamental Volterra equation
\begin{equation}\label{eq:foundamental_volterra_intro}
    \VSig{x}{K}^{\tau}_{s,t} = 1 + \int_s ^t \VSig{x}{K}^{u}_{s,u} \otimes K(\tau, u) \dot{x}_u \dd u,\qquad 0\le s\le t \le \tau \le T.
\end{equation}
While this may not be the only way to incorporate a kernel (cf.~\cite{jaber2025exponentially, bloch2026exponentially} for the exponential case), it is a natural approach for at least four reasons presented in \cite{i_part}. 
First, it allows for generic kernels under minimal integrability assumptions, including fractional kernels up to a certain singularity. 
Secondly, the path separation property is retained under suitable assumptions on the kernel, entailing a universal approximation theorem on path space. 
Thirdly, it satisfies a Chen-type relation
\begin{align}\label{eq:vsig_chen_comp}
    \VSig{x}{K}^{\tau}_{s,t}
    =
    [{\VSig{x}{K}_{s,u}}\oast{\VSig{x}{K}}]^{\tau}_{u,t}
    \qquad
    0\leq s\leq u\leq t\leq \tau\leq T,
\end{align}
where $\oast$ denotes a combined tensor product and integration operation, to be made explicit below.
Lastly, the corresponding \emph{signature kernel},
\[
\kappa^{K}(x,y)_{s,t}
=
\left\langle
\VSig{x}{K}^{s}_{0,s},
\VSig{y}{K}^{t}_{0,t}
\right\rangle,
\]
satisfies a closed two-parameter integral equation in $\mathbb{R}^m$, thereby allowing for the application of the kernel trick in typical learning tasks on path space.

In this paper, we treat the task of computing these \emph{Volterra signatures} and signature kernels on time series data.
For analytical purposes, the Chen relation \eqref{eq:vsig_chen_comp} serves a role similar to its classical counterpart \eqref{eq:sig_chen_intro} (see \cite{HT21} for its use in lifting Volterra rough paths). However, its computational consequences are far from evident.
Indeed, since the convolution $\oast$ operation is not purely arithmetic, even when $x$ is discretized a tractable computation without resorting to quadratures on the kernel requires additional structure and understanding.
We provide such structure for a large palette of kernels, including families of fractional and exponential type kernels.
Furthermore, even when the analytic component of $\oast$ is resolved, the defining equation \eqref{eq:foundamental_volterra_intro} suggests that computation is tied to quadratic complexity in the number of time steps.
We will show that this hurdle can be overcome for convolution kernels on uniform time grids using the fast Fourier transform (FFT) and for a large class of matrix exponential kernels by lifting the state.

To further illustrate the challenges that arise and the mathematical tools needed for the computation of Volterra signatures, we first compare the situation to the classical signature:

\textbf{Computation of classical signatures.}
Since the signature is defined for paths of a continuous time-parameter, when provided with a data stream $(x_0, \dots, x_J) \in (\RR^{d})^{J+1}$ we first need to interpolate to a continuous piecewise smooth path $x: [0,T]\to\RR^d$ with $x_{t_j} = x_j$ $(j=0, \dots, J)$.
Chen's identity then allows to use
\begin{equation}\label{eq:sig_chen_comp}
    \Sig{x}_{t_0, t_j} = \Sig{x}_{t_0, t_{j-1}} \otimes \Sig{x}_{t_{j-1}, t_j},
\end{equation}
as a dynamic computation step.
For piecewise linear interpolation, the signature on the interval $[t_{j-1}, t_j]$ then simply reduces to a tensor exponential
$$ \Sig{x}_{t_{j-1}, t_j} ~=~ \exp_{\otimes}(x_{j} - x_{{j-1}}) ~=~ 1+ \sum_{k=1}^N \frac{1}{k!}(x_{j} - x_{{j-1}})^{\otimes k},$$
allowing an algorithmic evaluation of \eqref{eq:sig_chen_comp} by iteratively applying tensor multiplications.
More precisely, the most efficient way of evaluating \eqref{eq:sig_chen_comp} is via a Horner-type scheme, which evaluates the tensor product with the previous signature and the exponential in a single run (introduced in this context in \cite{kidger2021signatory}).
The total costs of computing the signature are then of the order $J d^N$ asymptotically for $J$ and $N$ large, noting that $d^N$ matches the output dimension asymptotically.

While linear interpolation may be seen as a modeling choice with apparent computational advantages, piecewise smooth interpolation itself is crucial.
If we were to directly lift the data $(x_1, ..., x_J)$ to a feature vector $\bS$ in the tensor algebra by using an Euler discretization of the signature fundamental equation \eqref{eq:intro_sig_fundamental}
$$\bS_{j} = \bS_{{j-1}} + \bS_{{j-1}} \otimes (x_{j} - x_{{j-1}}), \qquad j=1, \dots, J,$$
then this feature does not share the same algebraic properties as the signature (the shuffle identity) and hence does not have the same expressiveness; that is, it does not enjoy the same universal approximation property.
Note that a minimal extension of $\bS$ to incorporate such algebraic structure is put forward in \cite{K2020a} as the \emph{iterated sums signature}.

The need for numerical precision to preserve expressivity also appears in the computation of signature kernels \cite{KiralyOberhauser2019,Salvi2021}. In practice, accuracy is often improved by replacing the observed paths with their linear interpolations on refined grids. Since this leads to significant computational bottlenecks, the accuracy of the numerical scheme plays an important role and remain an active topic in the literature \cite{lemercier2025log, cass2025numerical,tamayo2025scalable}.

\textbf{Computation of Volterra signatures.} 
Having advocated for exact computations and approximations for signatures, let us highlight the challenges we are facing in the Volterra case and the remedies proposed in this paper. Let us start by mentioning that for a given suitably integrable kernel $K:\Delta^2 \to \mathcal{L}(\RR^d;\RR^m)$, the Chen identity can be stated as
\begin{align}\label{eq:vsig_chen_comp_int}
    \VSig{x}{K}^{\tau}_{t_0, t_j} = \VSig{x}{K}_{t_0, t_{j-1}}^{\tau} - \int_{t_{j-1}}^{t_j} \VSig{x}{K}_{t_0, t_{j-1}}^{s} \otimes \frac{\dd}{\dd{s}}\VSig{x}{K}_{s, t_{j}}^\tau \dd{s}.
\end{align}
In contrast to \eqref{eq:sig_chen_comp} the evaluation of \eqref{eq:vsig_chen_comp_int} still requires the resolution of analytic operations. 
Turning this identity into a computationally tractable scheme is the topic of this paper.
The main results and observations of this paper are summarized as follows:
\begin{enumerate}
    \item \label{itm:compu_one} An exact resolution of the single Chen step for Volterra signatures \eqref{eq:vsig_chen_comp_int} in terms of iterated tensor products is generally not possible. 
    However, for piecewise linear interpolation we can 
    separate the computation of $\VSig{x}{K}_{s, t}^\tau$ with $(s, t)\in \Delta^2_{t_{j-1}, t_j}$ into an analytic (pre-computation) part depending only on the kernel $K$, and an arithmetic part depending only on the increment $(x_{t_j} - x_{t_{j-1}})$ (cf. \Cref{prop:sig term over piece segemtn}). 
    For a variety of common kernels, such as exponential and fractional-type, the analytic part of the computation results in closed-form and efficiently computable weight factors (see \Cref{sec:kernel_computations,sec:fssk_weights}).

    \item\label{item:p2} Breaking down further to the increment level with Chen's identity, a representation of the Volterra signature in terms of weighted sums of monomials of increments is possible (cf. \Cref{prop:chen_full_breakdown}), and once again the weight terms can be efficiently computed for a wide range of kernels.
    A direct implementation of this fully combinatorial expansion of the Volterra signature is generally not feasible, but it readily connects to two approaches for which an efficient dynamic scheme is possible.
    
    \item\label{item:approx_algo}
    The first method alluded to in point \eqref{item:p2} is based on the approximation
    \[
    \VSig{x}{K}_{t_0,t_{j-1}}^{s}
    \approx
    \VSig{x}{K}_{t_0,t_{j-1}}^{t_{j-1}}
    +
    \sum_{\alpha\in \decoSet}
    {\bf C}^{\alpha}_{t_{j-1},t_j}
    (s-t_{j-1})^{\alpha},
    \]
    inserted into \eqref{eq:vsig_chen_comp_int}.
    Provided that the Volterra signature on single intervals and its integrals against $(s-\cdot)^\alpha$ can be computed exactly, see point \eqref{itm:compu_one} above, this yields an algorithmic procedure with computational cost proportional to $J^2$, where $J$ denotes the length of the data stream.
    The choice of exponents $\decoSet$ depends on the regularity of the kernel and on the desired order of convergence. For smooth kernels, integer powers suffice, for instance $\decoSet=\{1,2\}$. For kernels with a fractional singularity of order $\beta-1$, one may instead choose $\decoSet=\{\beta,1\}$.
    This algorithm and its numerical accuracy are substantiated for general kernels in \Cref{thm:quadratic_aglo}, and specialized to fractional kernels in
\Cref{thm:frac_scheme}.

    \item Specializing to convolution kernels $K(t,s)=K(t-s)$, the inner iteration in the
    general Volterra scheme takes the form of a discrete
    convolution. This allows us to identify an algorithm that performs FFTs recursively over tensor levels\footnote{We note the  similarity to the type of FFT acceleration used for higher-order schemes for Volterra convolution equations; see, e.g., \cite{hairer1985fast}.}, thereby reducing the computational cost in the sequence length from quadratic to $O(J\log J)$ (cf. Section~\ref{sec:fft}).
    
    \item For kernels of exponential-polynomial and periodic type (and linear combinations thereof) we can make use of a \emph{state space lift} to provide an \emph{exact} scheme that likewise overcomes the quadratic costs.
    This algorithm is substantiated in \Cref{thm:algo_multiplicative}, and its costs are proportional to $J\times R^2$, where $R$ is the dimension of the state space.
    In this case, to compute the weight terms associated with the kernel, we make use of their relation to higher-order Fréchet derivatives of the matrix exponential map, which allows us to resort to efficient algorithms from  numerical linear algebra (cf. \Cref{sec:fssk_weights}).

    \item Additional complexity in all of the above mentioned results stems from the fact that we treat matrix valued kernels.
    We show that this becomes particularly tractable for general convolution kernels and piecewise constant kernels, where the number of matrix kernel components in $K(t,s) = k_1(t,s) A_1 + \dots + k_q(t,s)A_q$ with scalar kernels $k_p: \Delta^2\to\RR$ and matrices $A_p\in \mathcal{L}(\RR^d;\RR^m)$ does not influence the asymptotic order of computational costs in $J$ and $N$.

    \item Finally, for finite-state-space kernels we also treat the associated Volterra signature kernel. Using the kernel trick developed in \cite{i_part}, its computation reduces to a closed system of two-parameter augmented Goursat-type equations with values in $\RR^{R\times R}$. We propose an explicit predictor--corrector finite-difference scheme for this system, which can be evaluated cell-wise along anti-diagonals and also allows for static kernel lifts of the input paths (cf. \Cref{sec:sig-kernel}).
\end{enumerate}

All algorithms proposed in this paper have been thoroughly implemented and are publicly available at
\url{https://github.com/hagerpa/tensordev}.
The implementation is based on the JAX backend \cite{jax2018github}, which provides competitive computational efficiency and enables automatic differentiation with respect to all parameters.
A numerical validation of the implementation is presented in \Cref{sec:numerical_validation}.

\textbf{Outline.}
In \Cref{sec:pscw_linear} we derive the basic decomposition of Volterra signatures for piecewise linear paths.
\Cref{sec:algo_quadratic} develops the general approximative scheme, including the FFT acceleration for convolution kernels and explicit kernel computations.
Finite-state-space kernels and the corresponding exact recursion are treated in \Cref{sec:algo_multiplicative}.
In \Cref{sec:sig-kernel} we discuss the associated Volterra signature kernel and its finite-difference approximation.
The computational cost analysis is given in \Cref{sec:cost_analysis}.

\section{Volterra signatures of piecewise linear paths}\label{sec:pscw_linear}

In this section, we provide representations of the Volterra signature for piecewise linear paths that decompose to the increment level of the underlying path, while keeping the matrix kernels in a general analytic form.
This expansion reveals the computational intricacies of the Volterra signature and will be used to derive algorithms in \Cref{sec:algo_quadratic,sec:algo_multiplicative}.
We begin with the necessary algebraic preliminaries, and then briefly recall the definition and some elementary properties of the Volterra signature; we refer to~\cite[Section 2]{i_part} for a full exposition.

\subsection{Preliminaries}
We introduce some standard notation used in the study of signatures. Throughout this article, we focus on Euclidean spaces \(V=\RR^d\) for some \(d\in\NN\).  We begin with the notion of words over a given alphabet.

\begin{defn}\label{def:words} For any $d \in \mathbb{N}$, we call $\cA_d= \{1,\dots,d\}$ the alphabet of $d$ letters. Moreover, we denote by $\cW_d^n$ the set of words of length $n\in \NN$ over the alphabet $\cA_d$, that is \begin{equation}\label{eq:words}
    \cW_d^n:=\left \{w=i_1\cdots i_n: i_1,\dots,i_n \in \cA_d\right \}.
\end{equation} For \(n=0\) we write \(\cW_d^0=\{\varnothing\}\), where \(\varnothing\) denotes the empty word. Finally, we denote by \(\cW_d\) the set of all words of any length, and whenever \(d\) is clear from the context we write \(\cW=\cW_d\).
\end{defn}
Next, we introduce the extended tensor algebra --the ambient space for signatures --equipped with addition and a product, which turns it into an algebra.

\begin{defn}\label{def:extended_TA}
    Let $\{e_1, \dots, e_d\} \subset V$ be a basis.
    Set \(V^{\otimes 0}=\RR\) and for $n\in \NN_{\ge1}$ denote by \(V^{\otimes n}\) the \(n\)-fold tensor power of the vector space \(V\).
    For any word $w = i_1 \cdots i_n\in \cW_d$ we define $e_{w} = e_{i_1} \otimes \cdots \otimes e_{i_n} \in V^{\otimes n}$ and note that $\{ e_w \,\vert\, w\in \cW^n_d\}$ forms a basis of $V^{\otimes n}$.
    The \emph{extended tensor algebra} over $V$ is given by the direct product
    $$T((V)) := \prod_{n\ge0} V^{\otimes n},$$
    which forms a vector space by componentwise summation and scalar multiplication.
    For an element $\bx$ in $T((V))$ we will use the following equivalent notation as a sequence and as a formal series for the decomposition into its tensor \emph{levels}  $$\bx = (\bx^{(0)}, \bx^{(1)}, \dots) = \sum_{n=0}^\infty \bx^{(n)} \in T((V)),$$  so that $\bx^{(n)} \in V^{\otimes n}$ for all $n\in \NN$.
    We define the \emph{tensor product} $\otimes$ on $T((V))$ in terms of the concatenation product of basis elements, i.e.,
    $e_{w} \otimes e_{v} = e_{wv}$ for all $w,v\in\cW,$
    which extends to all of $T((V))$ by bilinearity.
    In particular, for ${\bf a}, \bx,\by\in T((V))$ we may write
$$\bz={\bf a} + \bx\otimes\by,\qquad \bz^{(n)}={\bf a}^{(n)} + \sum_{k=0}^{n}\bx^{(k)}\otimes \by^{(n-k)}\in V^{\otimes n},\qquad n\in\NN.$$
    \end{defn}
In order to work with finite truncations of signatures, we next introduce the canonical projection on \(T((\RR^d))\).
\begin{defn}\label{def:proj_and_trunc}
        We write \(\pi_n:T((V))\to V^{\otimes n}\) for the canonical projection, so that $\pi_n(\bx) = \bx^{(n)}$ for all $\bx\in T((V))$ and $n\in \NN$.
        Furthermore, for $N\in\NN$ we define the \emph{tensor truncation} $$\pi_{\le N} : T((V)) \to T((V)), \qquad \bx \mapsto (\bx^{(0)}, \bx^{(1)}, \dots, \bx^{(N)}, 0, \dots).$$
        The image of the truncation map $T^N(V) := \pi_{\le N} T((V))$ is called the \emph{truncated tensor algebra}, which indeed forms an algebra under the truncated tensor product
        $$\bx \otimes_N \by := \pi_{\le N}(\bx \otimes \by), \qquad \bx,\by\in T^N(V).$$
        We also define the tensor algebra by $T(V) = \cup_{N\ge 0}T^{N}(V)$.
    \end{defn}

   Finally, we introduce the \emph{shuffle product} on $T(V)$, a commutative product that will be of importance to us for the formulation of certain  algorithms in Section~\ref{sec:algo_quadratic}.

\begin{defn}\label{def:shuffle}
We define the shuffle product $\shuffle$ for words by setting
$$w\shuffle \varnothing=\varnothing\shuffle w=w,\qquad w\in\cW_d,$$
and then recursively
$$
(iw)\shuffle (jv):= i\big(w\shuffle (jv)\big)+j\big((iw)\shuffle v\big),\qquad i,j\in\cA_d,\; w,v\in\cW_d,
$$
where $iw$ denotes concatenation of the letter $i$ with the word $w$.
Note that the shuffle of a word $i_1\cdots i_n\in\cW$ and a single letter $j\in\cA_d$ simply reduces to
\begin{equation}\label{eq:shuffle_single_letter}
(i_1\cdots i_n)\shuffle j=\sum_{k=0}^{n} i_1\cdots i_k\, j\, i_{k+1}\cdots i_n.
\end{equation}
We extend $\shuffle$ to a bilinear and here symmetric map on $T(V)$ by the identification $e_i \leftrightarrow i \in \mathcal{A}_d$.
\end{defn}

\subsection{The  Volterra Signature}
 As motivated in the introduction, we focus here on Lipschitz continuous paths $x:[0,T] \rightarrow \RR^d$, writing $x \in \cC^{0,1}([0,T];\RR^d)$, and on the following class of matrix-valued kernels.
\begin{defn}\label{def:kernel_class}
Consider a matrix-valued kernel $K: \Delta^2 \to \mathcal{L}(\RR^d, \RR^m)$. In the sequel we assume $K\in \Lkernel$, where the latter space stands for the set of measurable functions on $\Delta^2$ which satisfy $$\sup_{t\in[0,T]}\int_0^t |K(t,s)| \dd{s} < \infty,$$
    with $|\cdot|$  denoting here any matrix norm on $\mathcal{L}(\RR^d,\RR^m)$.
\end{defn}
In the following definition, we introduce the Volterra signature as the collection of iterated integrals of $x$ weighted by the kernel $K$. As discussed in~\cite[Section~2.3]{i_part}, the resulting object is indeed the unique solution to the fundamental equation~\eqref{eq:foundamental_volterra_intro} motivated in the introduction.
\begin{defn}\label{def:vsig} Consider a signal $x\in\mathcal{C}^{0,1}([0,T]; \RR^d)$, as well as a kernel $K$ with values in $\cL(\RR^{d};\RR^{m})$ that satisfies \Cref{def:kernel_class}.
We define the \emph{Volterra signature} $\mathrm{VSig}$ component-wise for all $n\in\NN$, $i_1 \cdots i_n \in \mathcal{W}^n_m$ and $(s,t,\tau)\in\Delta^3$ by
\begin{equation}\label{eq:def_vsig_comp}\VSig{x}{K}^{i_1\cdots i_n,\tau}_{s,t} = \int_{\Delta^n_{s,t}} \prod_{l=1}^n K^{i_l}(r_{l+1},r_l)\dd{x_{r_l}},\end{equation}
with the convention $r_{n+1} = \tau$, and where we denote by $ K^{i_l}\dd x$ the component $i_l$ in the  $\RR^m$ valued vector $K(\tau,\cdot)\dd{x_\cdot}$. The full Volterra signature $\VSig{x}{K}$ is then defined by the formal tensor series
\begin{align*}
    \VSig{x}{K}^{\tau}_{s,t} := \sum_{n=0}^\infty \VSig{x}{K}^{(n),\tau}_{s,t}  := \sum_{n=0}^\infty \sum_{\substack{w\in\mathcal{W}\\|w| =n}}\VSig{x}{K}^{w,\tau}_{s,t}  e_w  \in T((\RR^m)),
\end{align*} recalling that $T((\RR^m))$ is introduced in \Cref{def:extended_TA}.
\end{defn}
Next, we introduce the aforementioned convolutional tensor product~$\oast$, which yields the Chen relation central to this work. We emphasize that the definition used here differs from the original one introduced in~\cite{HT21}, but it is shown to be equivalent in~\cite[Lemma~2.26]{i_part}.
\begin{defn}\label{def:convolution_product}
    Consider a pair $(x,K)$ as in \Cref{def:vsig} with Volterra signature denoted by $\bz = \VSig{x}{K}$. %
    For any bounded and measurable $\by: [0,T] \to T((\RR^m))$, we define
    \begin{equation}\label{eq:tensor_convol_prod}
        [\by \oast \bz]_{s,t}^\tau = \by_\tau-\int_s^t \by_r \otimes \Big ( \frac{\dd }{\dd r} \bz_{r,t}^{\tau} \Big ) \dd r \in T((\RR^m)), \qquad (s,t,\tau) \in \Delta^3, 
    \end{equation} recalling that $\otimes$ is introduced in \Cref{def:extended_TA}.
\end{defn}

We conclude this section with the fundamental Chen relation for Volterra signatures, we refer to \cite[Lemma 2.26 and Corollary 2.36]{i_part} for the proof.
\begin{prop}\label{prop:Chen_part_ii} For any pair $(x,K)$ as in \Cref{def:vsig}, we have 
\begin{equation}\label{eq:chen_tensor}
    \VSig{x}{K}^{\tau}_{s,t} = [{\VSig{x}{K}^{\cdot}_{s,u}}\oast{\VSig{x}{K}}]^{\tau}_{u,t}, \qquad (s,u,t,\tau) \in \Delta^4.
\end{equation}

\end{prop}
Let us conclude with 
a remark on the analogy to the case of classical signatures.
\begin{rem}
     Suppose $m=d$ and let $K(t,s)= \mathrm{I}_d$. Then it is easily observed from \eqref{eq:def_vsig_comp} that $\VSig{x}{K}_{s,t}^\tau = \mathrm{Sig}(x)_{s,t}$ for all $(s,t,\tau) \in \Delta^3.$ Moreover, by definition \eqref{eq:tensor_convol_prod}, in this case the Chen relation \eqref{eq:chen_tensor} reads \begin{align*}
        \mathrm{Sig}(x)_{s,t} & = \mathrm{Sig}(x)_{s,u}-\mathrm{Sig}(x)_{s,u} \otimes \int_u^t \frac{\dd}{\dd r}\mathrm{Sig}(x)_{r,t} \dd r \\ & = \mathrm{Sig}(x)_{s,u}-\mathrm{Sig}(x)_{s,u} \otimes (1-\mathrm{Sig}(x)_{u,t}) \\ & = \mathrm{Sig}(x)_{s,u} \otimes \mathrm{Sig}(x)_{u,t},
    \end{align*} which corresponds to the familiar Chen relation stated in \eqref{eq:sig_chen_intro}.

\end{rem}

\subsection{General discretization procedure}\label{sec:general-discrete}
In this section we derive explicit expressions for the Volterra signature of a piecewise linear path.
The key observation is that, under the kernel decomposition \eqref{eq:decomposed_kernel}, the Volterra signature on a single interval separates into analytic weight factors depending only on the kernel and algebraic terms depending only on the path increments (see \Cref{prop:sig term over piece segemtn}).
Iterating this via the Chen relation \eqref{eq:chen_tensor} then yields a full breakdown of the Volterra signature over the entire partition in terms of weighted monomials of increments (see \Cref{prop:chen_full_breakdown}).
We first set the stage for our computations concerning piecewise linear paths with a couple of notation. \begin{nota}\label{not:comp_sec}
    We are given a kernel $K$ as in  \Cref{def:kernel_class} (that is $K$ lies in the space $\Lkernel$), as well as a piecewise linear path $x$. More specifically, $K$ and $x$ are of the following form: \begin{itemize}
        \item[(i)] For a partition $0 \le t_0 < t_1 < \cdots < t_J \le T$, we define $x$ as \begin{equation}\label{eq:pcsws_linear_x}
    x_{t} = x_{t_j} + v_j(t - t_j), \qquad v_j\in\RR^d,\quad t\in [t_j, t_{j+1}].
\end{equation}
Otherwise stated, we have $\dd x_t = v_j \dd{t}$ on $[t_j, t_{j+1})$.
\item[(ii)] As we are mainly interested in separating analytic from algebraic computations here, we assume that one can decompose the kernel $K$ as 
\begin{equation}\label{eq:decomposed_kernel}
    K(t,s) = \sum_{p=1}^q k_{p}(t,s) A_p, \qquad k_p\in L^{\infty,1}(\Delta^2; \RR), \quad A_p \in \mathcal{L}(\RR^d;\RR^m), \quad p=1, \dots, q.
\end{equation} 
This will make notation more tractable and reveal simplifications -- for instance in the scalar kernel case $K = k \mathrm{I}_d$ with $q=1$ -- more directly.
    \end{itemize}
\end{nota}

\noindent
In the setting of \Cref{not:comp_sec}, we will see how to get computationally tractable expressions for the Volterra signature. We will start with a general procedure and then move to concrete examples allowing for explicit computations.%

The following lemma shows that the constant velocity $v_{j}$ on each interval $[t_{j-1}, t_{j}]$ allows one to factor the path increments out of the iterated integrals, leaving scalar-valued Volterra signature weights $\cK$ that depend only on the kernel.

\begin{lem}\label{prop:sig term over piece segemtn}
Let $x$ be a piecewise linear path and $K$ be a kernel as in Notation \ref{not:comp_sec}, with respective decompositions \eqref{eq:pcsws_linear_x}-\eqref{eq:decomposed_kernel}. Then for the Volterra signature on a single interval of the partition it holds
\begin{equation}\label{eq:smpl_identity}
    \VSig{x}{K}^{(n),\tau}_{t_{j-1}, t_j} =  \, \sum_{p_1, \dots, p_n = 1}^q \mathcal{K}_{t_{j-1},t_j}^{p_1\cdots p_n,\tau}A_{p_1}v_j 
\otimes\cdots\otimes A_{p_n}v_j ,  \, \qquad \tau \in [t_j, T],
\end{equation}
for all $j=1, \dots, J$ and $n\in\NN_{\ge 1}$. Note that in~\eqref{eq:smpl_identity}, $\cK$ denotes the Volterra signature (as defined in \eqref{eq:def_vsig_comp}) above the path
\begin{equation}\label{eq:def_calK}
y^\tau_{s, t} = \bigg(\int_s^t k_1(\tau, u) \dd u, \dots, \int_s^t k_q(\tau, u)\dd u\bigg), \quad (s,t,\tau) \in \Delta^3.
\end{equation} 
\end{lem}
\begin{proof}
We start from the Definition \ref{def:vsig} of the Volterra signature. Since equation \eqref{eq:pcsws_linear_x} specifies that $\dot{x}_t=v_j$ on $[t_j,t_{j+1})$, formula \eqref{eq:def_vsig_comp} reads \[
\VSig{x}{K}^{(n),\tau}_{t_{j-1}, t_j} = \int_{\Delta^n_{t_{j-1}, t_j}} \bigotimes_{l=1}^n K(r_{l+1},r_l)v_j \dd r_l.
\] 
Taking into account the decomposition \eqref{eq:decomposed_kernel} for $K$, the above can be recast as 
\begin{equation}\label{eq:recast_yVsig}
    \VSig{x}{K}^{(n),\tau}_{t_{j-1}, t_j}=\sum_{p_1 \ldots, p_n = 1}^q \cK_{t_{j-1},t_j}^{p_1\cdots p_n,\tau} A_{p_1}v_j \otimes\cdots\otimes A_{p_n}v_j,
\end{equation}
where we have set \begin{equation}
    \label{eq:recast_Y_2}
    \cK_{t_{j-1},t_j}^{p_1\cdots p_n,\tau} = \int_{\Delta^n_{t_{j-1}, t_j}} \prod_{l=1}^{n} k_{p_l}(r_{l+1},r_l) \dd r_l.
\end{equation} Furthermore, it is readily checked from \eqref{eq:recast_Y_2} and Definition \ref{def:vsig} that \begin{equation}\label{eq:recast_3}
    \cK_{t_{j-1},t_j}^{p_1\cdots p_n,\tau} = \VSig{t}{K}_{t_{j-1},t_j}^{p_1\cdots p_n,\tau} = \by_{t_{j-1},t_j}^{p_1\cdots p_n,\tau},
\end{equation} where $y$ is defined by \eqref{eq:def_calK}. Putting together \eqref{eq:recast_yVsig} and \eqref{eq:recast_3}, this proves our claim~\eqref{eq:smpl_identity}. Our proof is complete.
\end{proof}

\begin{rem}\label{rem:d}
    In \eqref{eq:smpl_identity} we have expanded our expression in terms of the coordinates $p_1,\dots,p_n$. A more compact expression, which will be used below, is 
    \begin{equation}\label{f1}
    \VSig{x}{K}_{t_{j-1},t_j}^{(n),\tau} = \sum_{w \in \cW_q^n }\cK_{t_{j-1},t_j}^{w,\tau} (Av_j)^w,
    \end{equation} 
    where we have set 
    \begin{equation} \label{eq:nota_matrix_tensor}
    (Av)^w =(Av)^{p_1\cdots p_n} = A_{p_1}v \otimes \cdots \otimes A_{p_n}v.
    \end{equation}
\end{rem}
The representation given in \eqref{eq:smpl_identity} or \eqref{f1} separates the convolution integrals in the Volterra signature fully on the side of the kernel; reducing the path contribution to monomials of the increment vector $v_i$.
The corollary below further simplifies this expression in case of a scalar kernel.

\begin{rem}\label{rem:wlog_ID}
    Note that in the case $q=1$, it follows from the definition of the Volterra signature (see Definition~\ref{def:vsig}), that $\VSig{x}{kA}=\VSig{Ax}{kI_m}$. Consequently, we may assume without loss of generality that $A=I_m$, at the price of replacing the driving path $x$ by $Ax$.
\end{rem}
\begin{cor}\label{cor:scalar_kernel}
    In the same context as for Lemma \ref{prop:sig term over piece segemtn}, assume that the kernel $K$ is such that $q=1$ in decomposition \eqref{eq:decomposed_kernel}, and in addition we have $K=kI_d$ for a scalar kernel $k \in L^{\infty,1}(\Delta^2;\RR)$. Then for all $j\geq 1$ and $\tau \in [t_j,T]$ we have 
    \begin{equation*}
        \VSig{x}{k}^{(n), \tau}_{t_{j-1}, t_j} = \kappa^{n, \tau}_{t_{j-1}, t_j} \, v_j^{\otimes n}, 
        \quad\text{with}\quad
        \kappa^{n, \tau}_{s,t}= \int_{\Delta^n_{s,t}} \prod_{i=1}^{n} k(r_{i+1},r_i)\dd r_i\in \RR,
    \end{equation*}
where we have used the convention $r_{n+1} = \tau$.
\end{cor}

\begin{rem}
    Let us particularize Corollary \ref{cor:scalar_kernel} to the classical signature case, i.e. when $k\equiv 1$. In this situation we have $\kappa^{n}_{t_{j-1},t_j} \equiv \frac1{n!}{(t_j-t_{j-1})^n}$, thus retrieving the tensor exponential $$\VSig{x}{1}_{t_{j-1}, t_j} \equiv \Sig{x}_{t_{j-1}, t_j} = \exp_{\otimes}(x_{t_j}-x_{t_{j-1}}).$$  
\end{rem}
Lemma \ref{prop:sig term over piece segemtn} gave an expression for the Volterra signature of a piecewise linear path on one interval $[t_j,t_{j+1}]$ of the time partition. 
The following proposition now demonstrates how Chen's relation allows to fully resolve the Volterra signature on an arbitrary interval $[t_0,t_j]$, thanks to weighted summations of monomials.
\begin{prop}\label{prop:chen_full_breakdown}
We work under the same conditions as in Lemma \ref{prop:sig term over piece segemtn}. 
That is $x$ is piecewise linear, $K$ verifies Hypothesis \ref{def:kernel_class}, and both can be decomposed as \eqref{eq:pcsws_linear_x}-\eqref{eq:decomposed_kernel}. Then for any $t_j$ in the partition $t_0,\dots, t_J$, any $l \geq j$ and any $n\in \mathbb{N}_{\geq 1}$, we have 
\begin{equation}\label{eq:weighted_sum_thm}
    \VSig{x}{K}^{(n),t_l}_{t_0, t_j} 
    = \sum_{k=1}^{j \wedge n}
    \;\sum_{0<i_1<\dots< i_k < j}
    \;\sum_{\substack{w_1,\dots, w_k \in \mathcal{W}_q\setminus\{\varnothing\} \\ |w_1| + \cdots + |w_k| = n}} \mathcal{K}^{(w_1, \dots, w_k)}_{i_1, \dots, i_k,l}(Av_{i_1})^{w_1} \otimes \cdots \otimes (Av_{i_k})^{w_k}
\end{equation} 
In equation \eqref{eq:weighted_sum_thm}, we used the following conventions: first for a word $w=p_1\cdots p_n \in \cW_q$ and $v \in \RR^d$, recall that $(Av)^w$ is defined by \eqref{eq:nota_matrix_tensor} and that the parameter $q$ comes from the decomposition \eqref{eq:decomposed_kernel}. 
Next for $w_1,\dots,w_k\in \cW_q$ and $0<i_1<\cdots < i_{k+1} \leq j$ we define 
\begin{equation}\label{eq:def_cK_words}\mathcal{K}^{w_1 \cdots w_k}_{i_1,\dots, i_{k+1}} := \int_{t_{i_1-1}}^{t_{i_1}} \cdots \int_{t_{i_{k}-1}}^{t_{i_{k}}} \prod_{l=1}^k \dot{\cK}^{w_l,r_{l+1}}_{r_l, t_{i_l}} \dd r_l\;\;\in \RR,\quad r_{k+1} = t_{i_{k+1}},\end{equation}
where the expression of $\dot{\cK}^{w,\tau}$ for a word $w=p_1\cdots p_n$ is 
\begin{equation}\label{eq:def_dot_cK}
    \dot{\cK}_{s,t}^{p_1\cdots p_n,\tau} :=  \int_{\Delta^{n-1}_{s,t}}  k_{p_1}(r_1,s)\prod_{a=1}^{n-1} k_{p_{a+1}}(r_{a+1},r_{a})\dd r_a, 
\end{equation}
with the usual convention $r_{n} = \tau$.

\end{prop}
\begin{proof}
    Clearly the statement holds for $j = 0$ as $\VSig{x}{K}^{(n),t_l}_{t_0, t_0}  = 0$ for all $n\ge 1$.
    Next note that for $j \ge 1$ it follows from Chen's relation \eqref{eq:chen_tensor} that
    \begin{equation*}
        \VSig{x}{K}^{(n),t_l}_{t_0, t_j} 
        =\VSig{x}{K}^{(n),t_l}_{t_0, t_{j-1}} - \sum_{l=1}^n \int_{t_{j-1}}^{t_j}\VSig{x}{K}^{(n-l), t}_{t_0, t_{j-1}}\otimes  \left(\frac{\dd}{\dd t}\VSig{x}{K}^{(l), t_l}_{t, t_j}\right) \dd{t}.
    \end{equation*}
    Now using \Cref{prop:sig term over piece segemtn} and Remark \ref{rem:d} we see that 
    $$-\frac{\dd}{\dd t}\VSig{x}{K}^{(l),t_l}_{t, t_j} = -\frac{\dd}{\dd t} \sum_{w\in \mathcal{W}_q^l}{\mathcal{K}}^{w,t_l}_{t, t_j}(Av_j)^w = \sum_{w\in \mathcal{W}_q^l}\dot{\mathcal{K}}^{w,t_l}_{t, t_j}(Av_j)^w \, ,
    $$
    and the claim \eqref{eq:weighted_sum_thm} follows by induction.
\end{proof}

We discuss the computational implications of the weighted monomial expansion \eqref{eq:weighted_sum_thm} in \Cref{sec:algo_quadratic,sec:algo_multiplicative}.
Like in Corollary \ref{cor:scalar_kernel}, we draw attention to the scalar case $K = k \mathrm{I}_d$, where the complexity of the inner summation reduces to the composition of the tensor level $n$ into $k$ parts.  

\begin{cor}\label{cor:scalar_expansion}
Let $x$ be a piecewise linear path as introduced in Notation \ref{not:comp_sec}. Similarly to Corollary \ref{cor:scalar_kernel}, assume that $K=kI_d$ for a scalar kernel $k\in L^{\infty,1}(\Delta^2;\RR)$. Then for all $0\leq j \leq l \leq J$ and $n\in\NN_{\ge 1}$ it holds
\begin{align}\label{eq:combinatorial_vsig}
    \VSig{x}{k}^{(n),t_l}_{t_0, t_j} 
    &= \sum_{k=1}^{j \wedge n}
    \;\sum_{\substack{0<i_1<\dots< i_k \le j\\n_1 + \cdots + n_k = n}}
    \mu^{n_1, \dots, n_k}_{i_1, \dots, i_k,l} \,v_{i_1}^{\otimes n_1}\otimes  \cdots\otimes  v_{i_k}^{\otimes n_k},
    \end{align}
    where $n_l > 0$ in the inner summation and the weights $\mu^{n_1, \dots, n_k}_{i_1, \dots, i_{k+1}}$ are defined by
    \begin{equation}\label{eq:weights_cor}
     \mu^{n_1, \dots, n_k}_{i_1, \dots, i_{k+1}} :=  \int_{t_{i_1-1}}^{t_{i_1}} \cdots \int_{t_{i_{k}-1}}^{t_{i_{k}}} \prod_{l=1}^k\dot{\kappa}^{n_l,r_{l+1}}_{r_l, t_{i_l}} \dd{r_l} \;\;\in\RR,
    \quad\text{where we set}\quad r_{k+1} = t_{i_{k+1}},
    \end{equation}
    with a kernel $\dot{\kappa}$ given by
\begin{equation}
\label{eq:kernel_cor}\dot{\kappa}^{n,\tau}_{s, t} 
=  
\int_{\Delta^{n-1}_{s,t}}  k(r_1,r_0)\prod_{a=1}^{n-1} k(r_{a+1}, r_{a}) \dd r_a, \quad\text{with the convention}\quad r_0 = s, \; r_{n} = \tau.
\end{equation}
\end{cor}

\begin{rem}
    The above form reveals the relation of the Volterra signature to the \emph{weighted iterated sum signature} \cite{diehl2024fruits}, although here the weights in general depend on the grading as well.
\end{rem}
\begin{rem}\label{rem:kappa_intro}
In the scalar kernel setting $K = k\,\mathrm{I}_{d}$, the kernel $\dot{\kappa}^{n}$ introduced in \eqref{eq:kernel_cor} naturally gives rise to the companion quantity
\begin{equation}\label{eq:kappa_general}
    \kappa^{n,\tau}_{s,t} := \int_{s}^{t} \dot{\kappa}^{n,\tau}_{u,t} \,\dd u, \qquad (s,t,\tau)\in \Delta^{3}.
\end{equation}
This is consistent with Corollary~\ref{cor:scalar_kernel}: a direct relabeling of the integration variables in the $n$-fold simplex integral therein shows that \eqref{eq:kappa_general} coincides with the expression for $\kappa^{n,\tau}_{s,t}$ given in that corollary. Both $\kappa^{n}$ and $\dot{\kappa}^{n}$ will be computed explicitly in the examples below.
\end{rem}
Remaining in the framework of scalar kernels $K = k\,\mathrm{I}_d$, we now turn to several concrete examples of~$k$ for which further explicit computations can be derived. We continue with various generalizations and additional examples in \Cref{sec:kernel_computations,sec:algo_multiplicative}.
\begin{example}[Constant kernel]\label{exmpl:constant_kernel}
We begin with the trivial scalar kernel $k\equiv 1$. %
In this case we are in the framework of Corollary \ref{cor:scalar_expansion} and the kernel $\dot{\kappa}$ in~\eqref{eq:kernel_cor} is such that \[
\dot{\kappa}_{s,t}^{n,\tau} = \int_{\Delta_{s,t}^{n-1}}\dd r_1 \cdots \dd r_{n-1} = \frac{1}{(n-1)!}(t-s)^{n-1}.
\]As a result, the factorized weights in  \eqref{eq:weights_cor} become $$\mu^{n_1, \dots, n_k}_{i_1, \dots, i_{k+1}} = \prod_{l=1}^k \frac{1}{n_l!}(t_{i_l}-t_{i_l-1}).$$ 
\end{example}
\begin{example}[Exponential kernel]\label{ex:exponential_case}
Let us now assume that the kernel $K$  can be written as \[
K(t,s)=\alpha e^{-\lambda(t-s)} \equiv k_{\alpha,\lambda}(t,s), \quad \text{with } \alpha,\lambda >0.
\]
Then for all $n\in \mathbb{N}_{\geq 1}$ the kernels $\dot{\kappa}^n$ and $\kappa^n$ (respectively defined by \eqref{eq:kernel_cor} and \eqref{eq:kappa_general}) can be computed thanks to elementary methods as 
\begin{equation}    \label{eq:kappas_exp}\dot{\kappa}^{n,\tau}_{s,t} = e^{-\lambda(\tau-s)}\frac{\alpha^n(t-s)^{n-1}}{(n-1)!}, \quad \text{and} \quad {\kappa}^{n,\tau}_{s,t} = e^{-\lambda(\tau -s)}\alpha^n (t-s)^{n}\varphi_n(\lambda(t-s)),
\end{equation}
 where for all $\delta >0$ the function $\varphi_n$ is defined by $\varphi_n(\delta) = \frac{1}{(n-1)!}\int_{0}^{1} e^{\delta(1-u)}u^{n-1} \dd{u}$. Moreover, one can compute the values of $\varphi_n$ through $\varphi_1(\delta)=\delta^{-1}(e^{\delta}-1)$ and the recurrence relation\footnote{See \cite[Section~10.7.4]{higham2008functions}, which also describes more stable evaluation techniques when $\delta$ is small.} $$\varphi_n(\delta) = \delta \varphi_{n+1}(\delta) + \frac{1}{n!}, \quad \delta >0, \; n\in\NN.$$
 Plugging \eqref{eq:kappas_exp} into \eqref{eq:weights_cor}, the resulting weights factorize as follows
\begin{align*}
    \mu^{n_1, \dots, n_k}_{i_1, \dots, i_{k+1}} &= \int_{t_{i_1-1}}^{t_{i_1}} \cdots \int_{t_{i_{k}-1}}^{t_{i_{k}}} e^{-\lambda(\tau - r_1)} \prod_{l=1}^k\frac{\alpha^{n_l}}{(n_l-1)!} (t_{i_l} - r_{l})^{n_l -1} \dd r_l \\
  &=  e^{-\lambda(t_{i_{k+1}} - t_{i_1-1})}\varphi_{n_1}(\lambda (t_{i_1}-t_{i_1-1}))\,n_1!\prod_{l=1}^k \frac{\alpha^{n_l}}{n_l!}(t_{i_l} - t_{i_l-1})^{n_l}.
\end{align*}
\end{example}
\begin{example}[Fractional kernel]\label{xmpl:fractional_kernel} 
We consider here a constant $\beta >0$, and a kernel $K$ given by
\begin{equation}\label{eq:1d_frac_def}
    K(t,s) = \Gamma(\beta)^{-1}(t-s)^{\beta-1} \equiv k_\beta(t, s).
  \end{equation}
Note that we will elaborate on this type of kernel in~\eqref{eq:multivariate_frac_def}.
 Then the expressions for $\dot{\kappa}^n$ and $\kappa^n$ in \eqref{eq:weights_cor} and \eqref{eq:kappa_pcwcst} are written in terms of the incomplete Beta functions $I_x(a,b)$, defined for $x\in [0,1]$ and $a,b>0$ by \begin{equation}\label{eq:regul_beta_def}
    I_x(a, b) := \frac{\Gamma(a+b)}{\Gamma(a)\Gamma(b)}\int_0^x u^{a-1} (1-u)^{b-1}\dd u.\end{equation} We obtain (see Section \ref{sec:apx_fractional} for details, taking $q=1$ and $\beta_1=\beta$ therein)
\begin{align*}
\dot{\kappa}^{n,\tau}_{s, t} = \frac{(\tau - s)^{n\beta -1}}{\Gamma(n\beta)} I_\frac{t-s}{\tau -s}((n-1)\beta, \beta), \qquad {\kappa}^{n,\tau}_{s, t} = \frac{(\tau - s)^{n\beta}}{\Gamma(n\beta +1)} I_\frac{t-s}{\tau -s}((n-1)\beta +1 , \beta),
\end{align*} For the weights $\mu$ in \eqref{eq:weights_cor} we then obtain
$$\mu^{n_1, \dots, n_k}_{i_1, \dots, i_{k+1}} =  \int_{t_{i_1-1}}^{t_{i_1}} \cdots \int_{t_{i_{k}-1}}^{t_{i_{k}}} \prod_{l=1}^k \frac{(r_{l+1} - r_{l})^{n_l\beta -1}}{\Gamma(n_l\beta)} I_\frac{{t_{i_l}}-r_l}{r_{l+1} - r_l}((n_l-1)\beta, \beta) \dd{r_l},\quad r_{k+1} = t_{i_{k+1}}.$$
This term seems to have no elementary closed form expression, nor satisfy any factorization properties as in the previous examples. Nevertheless, we can still make use of the explicit form of $\kappa$ to approximate the Volterra signature with the algorithm described in \Cref{sec:algo_quadratic}.
\end{example}

We conclude this section with a combinatorial symmetry that is observed among all kernels of convolutional form. 

\begin{lem}\label{lem:cK_symmetry}
We work under the framework of \Cref{prop:chen_full_breakdown}. That is we consider a piecewise linear signal $x$ and a kernel $K$ admitting the decompositions \eqref{eq:pcsws_linear_x}-\eqref{eq:decomposed_kernel}. In addition, assume $k_1,\dots,k_q$ are of convolution form, that is  $k_p(s,t)=k_p(t-s)$ for all $(s,t)\in\Delta^2$. Recall that the kernel $\dot{\cK}$ is defined in \eqref{eq:def_dot_cK} and for $w \in \cW_q$ we have \begin{equation}
    \label{eq:cK_again}\cK_{s,t}^{w,\tau}= \int_s^t\dot{\cK}_{u,t}^{w,\tau}\dd u.
\end{equation}
Then for all $n\in\mathbb N$, all permutations $\sigma\in\mathcal S_n$ of $\{1,\dots,n\}$, and all $(s,t,\tau)\in\Delta^3$ it holds
$$
\dot{\cK}_{s,t}^{wp,\tau}=\dot{\cK}_{s,t}^{\sigma(w)p,\tau}
\quad\text{and}\quad
\cK_{s,t}^{wp,\tau}=\cK_{s,t}^{\sigma(w)p,\tau},
\qquad w\in\mathcal W_q^n,\ p\in\mathcal A_q,
$$
where we set $\sigma(p_1\cdots p_n):=p_{\sigma(1)}\cdots p_{\sigma(n)}$ for $p_1\cdots p_n\in\mathcal W_q^n$.
Moreover, if $\tau=t$, then for all $(s,t)\in\Delta^2$ it holds
$$
\dot{\cK}_{s,t}^{w,t}=\dot{\cK}_{s,t}^{\sigma(w),t}
\quad\text{and}\quad
\cK_{s,t}^{w,t}=\cK_{s,t}^{\sigma(w),t},
\qquad w\in\mathcal W_q^n.
$$
\end{lem}

\begin{proof}
First we consider $\tau=t$.
Using the convolution form of the kernels, a change of variables in \eqref{eq:def_dot_cK} yields
$$
\dot{\cK}_{s,t}^{w,t}
=\int_{\{u_i\ge0,\ \sum_{i=1}^n u_i=t-s\}} \prod_{i=1}^n k_{p_i}(u_i)\,du_1\cdots du_{n-1},
$$
for all $w=p_1\cdots p_n\in\mathcal W_q^n$.
Since the integration domain and the integrand are invariant under permutations of $(u_1,\dots,u_n)$, we obtain
$\dot{\cK}_{s,t}^{w,t}=\dot{\cK}_{s,t}^{\sigma(w),t}$ for all $\sigma\in\mathcal S_n$.
Integrating in $s$ yields $\cK_{s,t}^{w,t}=\cK_{s,t}^{\sigma(w),t}$.
For general $\tau \geq t$, fix $p \in \mathcal{A}_{q}$, $w \in \mathcal{W}^{n}_{q}$, and $\sigma \in \mathcal{S}_{n}$. By definition~\eqref{eq:def_dot_cK} and the symmetry $\dot{\mathcal{K}}^{w,t}_{s,t} = \dot{\mathcal{K}}^{\sigma(w),t}_{s,t}$ just established, we have
$$
\dot{\cK}_{s,t}^{wp,\tau}
=\int_s^t\dot{\cK}_{s,r}^{w,r}
k_p(\tau-r)\,dr = \int_s^t\dot{\cK}_{s,r}^{\sigma(w),r}
k_p(\tau-r)\,dr= \dot{\cK}_{s,t}^{\sigma(w)p,\tau}.
$$
Finally, integration in $s$ yields once again $\cK_{s,t}^{wp,\tau}=\cK_{s,t}^{\sigma(w)p,\tau}$.
\end{proof}

\section{General approximative Volterra scheme}\label{sec:algo_quadratic}

In the last section we provided methods to reduce the computation of a Volterra signature to the evaluation of weights. In the setting of Notation~\ref{not:comp_sec}, our basic result was formula \eqref{eq:smpl_identity}, specialized to the scalar case in \eqref{eq:combinatorial_vsig}. 
The practical relevance of these combinatorial expansions hinges on the availability of efficient evaluation procedures for the corresponding iterated weighted sums. In \Cref{sec:kernel_computations}, we gather several central examples of kernels for which these weights can be computed---at least partially---in analytic form, thereby largely extending the three examples in \Crefrange{exmpl:constant_kernel}{xmpl:fractional_kernel} from the previous section.

Even when the weights in \eqref{eq:def_dot_cK} or \eqref{eq:kernel_cor} can be evaluated efficiently, the direct computation of the corresponding formulae is generally still infeasible.\footnote{Given precomputed weights and monomials, a heuristic computation in the case $q=1$ yields a cost of the order of $\binom{j+|w|-2}{|w|}$.} In this section we therefore present a higher-order dynamic algorithm which, when applicable, overcomes this computational bottleneck. As already motivated in the introduction, the starting point is Chen's relation (see Proposition~\ref{prop:Chen_part_ii})
\begin{equation*}
    \VSig{x}{K}^{\tau}_{t_0,t_j}
    -
    \VSig{x}{K}_{t_0,t_{j-1}}^{\tau}
    =
    - \int_{t_{j-1}}^{t_j}
    \VSig{x}{K}_{t_0,t_{j-1}}^{s}
    \otimes
    \frac{\dd}{\dd s}
    \VSig{x}{K}_{s,t_j}^{\tau}
    \,\dd s .
\end{equation*}
Let us write $\bv_j:=\VSig{x}{K}_{0,t_j}^{t_j}$ and suppose again that $x$ is piecewise linear as in \eqref{eq:pcsws_linear_x}. Applying Lemma~\ref{prop:sig term over piece segemtn} together with a telescoping argument gives
\begin{equation}\label{eq:exact_Chen_scheme}
    \bv_j
    =
    1-
    \sum_{i=1}^{j}
    \int_{t_{i-1}}^{t_i}
    \VSig{x}{K}_{0,t_{i-1}}^{s}
    \otimes
    \dot\cE_{s,t_i}^{t_j}
    \,\dd s,
\end{equation}
where $\cE$ denotes the right-hand side in \eqref{eq:smpl_identity}, that is \begin{equation}\label{eq:mathcalE}
    \mathcal{E}^{\tau}_{s,t}(y_1, \dots, y_q) = \sum_{n=1}^N \sum_{p_1\cdots p_n\in\cW_q^n}\frac{1}{(t - s )^n}\mathcal{K}_{s,t}^{p_1\cdots p_n,\tau}y_{p_1} \otimes \cdots \otimes y_{p_n}, \qquad (s,t,\tau) \in \Delta^3,
\end{equation} with $\cK_{s,t}^{p_1\cdots p_n,\tau}$ given by \eqref{eq:recast_Y_2}. Since the local weights entering $\cE$ are available in closed form for the kernel classes considered in \Cref{sec:kernel_computations}, a first natural numerical algorithm is obtained by applying a left-point approximation to the upper-parameter
$s\mapsto \VSig{x}{K}_{0,t_{i-1}}^{s}$. This yields the explicit weighted recursion
\begin{equation}\label{eq:naive_quadratic_stencil}
    \bv_j
    \approx
    1+
    \sum_{i=1}^{j}
    \bv_{i-1}
    \otimes
    \cE_{t_{i-1},t_i}^{t_j}.
\end{equation}
However, as we shall see below, for singular kernels $K$ this naive scheme may converge slowly as the mesh size tends to zero. For example, when $ 
k_p(t,s)=(t-s)^{\beta-1}$ with $
\beta\in(0,1)$ in \eqref{eq:decomposed_kernel}, the computations in \Cref{sec:kernel_computations} show that the relevant upper-parameter maps are only $\beta$-Hölder continuous near the diagonal. Consequently, a left-point approximation of this dependence cannot be expected to yield high-order convergence. In the following, we introduce a higher-order generalization of \eqref{eq:naive_quadratic_stencil} which is tailored to this fractional-type singular behavior.

\subsection{Higher-order fractional expansion scheme}\label{sec:algo_thm}In order to introduce a higher-order version of \eqref{eq:naive_quadratic_stencil}, we require more precise information on the upper-parameter regularity of the Volterra signature. More precisely, for each interval $[t_i,t_{i+1}]$ we consider the map
\[
F_i(u):=\VSig{x}{K}_{0,t_i}^{t_i+u},
\qquad
u\in[0,t_{i+1}-t_i].
\]
 The naive scheme \eqref{eq:naive_quadratic_stencil} amounts to replacing this map by its left endpoint value $F_i(0)$. The higher-order scheme below instead assumes that $F_i$ can be approximated locally by a finite expansion in fractional powers. This is the precise additional regularity property needed to improve the left-point approximation.
 
 Fix a truncation level $N\in \mathbb{N}$ and consider a finite set of exponents \[
    \mathcal B=\{\rho_0,\ldots,\rho_{|\mathcal B|-1}\}\subset[0,\infty),
    \qquad
    0=\rho_0<\rho_1<\cdots<\rho_{|\mathcal B|-1}.
\]
The higher-order scheme described below relies on the expansion
\begin{equation}\label{eq:frac_expansion_ass}
    \pi_{\leq N}(F_i(u))
    =
    \sum_{\rho\in\mathcal B}\mathbf a_{i,\rho}u^\rho
    +
    R_i(u),
    \qquad u\in[0,h_i],
\end{equation} where $h_i:=t_{i+1}-t_i$ and the coefficients $\mathbf a_{i,\rho}\in T^{(N)}(\mathbb R^m)$ are the interpolation
coefficients determined by
\[
    \sum_{\rho\in\mathcal B}
    \mathbf a_{i,\rho}(\theta_a h_i)^\rho
    =
    F_i(\theta_a h_i),
    \qquad a=0,\ldots,|\mathcal B|-1,
\] for some interpolation nodes $\theta_0,\dots,\theta_{|\decoSet|-1}$ with invertible Vandermonde matrix $(\theta_a^\rho)$.

Given the expansion~\eqref{eq:frac_expansion_ass}, we propose  the following natural higher-order version of \eqref{eq:naive_quadratic_stencil} \begin{equation}\label{eq:higher_order_scheme}
    \pi_{\leq N}(\bv_j) = 1 + \sum_{i=1}^{j}\sum_{\rho \in \decoSet} \bC_{i-1,\rho} \otimes_N \cE^{t_j;\rho}_{t_{i-1},t_{i}}, \qquad \cE^{\tau;\rho}_{u,v}:=-\int_u^v(s-u)^\rho \dot \cE_{s,v}^{\tau} \dd s,
\end{equation} where the coefficients \(\bC_{i,\rho}\) are determined recursively by the interpolation conditions
\begin{equation}\label{eq:coeffcicients_higher_order}
    \sum_{\rho\in \decoSet}(\theta_a h_i)^\rho \bC_{i,\rho}
    =
    1+
    \sum_{b=0}^{i-1}
    \sum_{\sigma\in \decoSet}
    \bC_{b,\sigma}
    \otimes_N
    \cE_{t_{b},t_{b+1}}^{\,t_i+\theta_a h_i;\sigma},
    \qquad a=0,\ldots,{\decoNum}.
\end{equation}

\begin{rem}\label{rem:generalized_cE}
	It is important to note that the generalized weights $\cE^{\tau;\rho}_{s,t}$ in
	\eqref{eq:higher_order_scheme} preserve the structure of $\cE^{\tau}_{s,t}$,
	up to the introduction of one additional kernel $k_\rho(t,s) := (t-s)^\rho .$
	More precisely, using the representation of $\dot\cK$ in
	\eqref{eq:def_cK_words}, we have
	\[
		\int_u^v (s-u)^\rho
		\dot\cK_{s,v}^{p_1\cdots p_n,\tau}\,\dd s
		=
		\dot\cK_{u,v}^{p_1\cdots p_n 0,\tau},
	\]
	where the appended symbol $0$ denotes the additional kernel $k_\rho$.
	Consequently, for all kernel classes considered in this paper, whenever
	closed-form expressions for the weights $\dot\cK$ in
	\eqref{eq:def_cK_words} are available, the generalized weights
	$\cE^{\tau;\rho}$ admit closed-form expressions as well. Concrete examples
	will be derived in Section~\ref{sec:kernel_computations}.
\end{rem}

\begin{example}[Naive scheme]\label{ex:naive_stencil} Choosing $\decoSet=\{0\}$, we have $\cE^{\tau;0}=\cE^\tau$, and the coefficient equation \eqref{eq:coeffcicients_higher_order} reads \[
\bC_{i,0}= 1+ \sum_{b=1}^{i}\bC_{b-1,0} \otimes_N \cE_{t_{b-1},t_{b}}^{t_i},
\] and we immediately recover the scheme  \eqref{eq:naive_quadratic_stencil}.
    
\end{example}
\begin{example}[Higher-order fractional scheme]\label{ex:fractional_scheme} We will see below that for fractional kernels $k_p(t,s)=(t-s)^{\beta-1}$, we may choose $\decoSet=\{0,\beta,1\}$ and interpolation nodes $\theta_0=0, \theta_1= \frac{1}{2}$ and $\theta_2 = 1$. On a cell $[t_i,t_{i+1}]$, define the local update 
\[
\widehat{\mathbf F}_i^\theta
:=
1+
\sum_{b=0}^{i-1}
\sum_{\rho\in\{0,\beta,1\}}
\mathbf C_{b,\rho}\otimes_N
\mathcal E_{t_b,t_{b+1}}^{t_i+\theta h_i;\rho}.
\] The system of for the coefficients \eqref{eq:coeffcicients_higher_order} is then solved by \[
\mathbf C_{i,0}
=
\widehat{\mathbf F}_i^0, \quad 
\mathbf C_{i,\beta}
=
\frac{
\widehat{\mathbf F}_i^{1/2}-\widehat{\mathbf F}_i^0
-\frac12\left(\widehat{\mathbf F}_i^1-\widehat{\mathbf F}_i^0\right)
}{
h_i^\beta(2^{-\beta}-1/2)
}, \quad 
\mathbf C_{i,1}
=
\frac{
2^{-\beta}\left(\widehat{\mathbf F}_i^1-\widehat{\mathbf F}_i^0\right)
-
\left(\widehat{\mathbf F}_i^{1/2}-\widehat{\mathbf F}_i^0\right)
}{
h_i(2^{-\beta}-1/2)
}.
\]
\end{example}
In order to quantify the convergence rate of the proposed scheme \eqref{eq:higher_order_scheme}-\eqref{eq:coeffcicients_higher_order}, we require precise assumptions on the local errors $R_i$ in \eqref{eq:frac_expansion_ass}, as well as convolution stability estimates for the weights $\cE$. These conditions can be verified for the kernels considered in this article, and we do so for fractional kernels (see Example~\ref{xmpl:fractional_kernel}) in Theoerem~\ref{thm:frac_scheme} below.

\begin{hyp}\label{hyp:fractional_expansion}
Given a set of exponents $\mathcal{B}$ and interpolation nodes $(\theta_a)$ as above, assume that there exist an  $\alpha>0$ and non-negative kernels
$\kappa_m:(0,T]\to[0,\infty)$ such that
\[
    \left|
        \dot{\cE}_{s,t}^{(m),\tau}
    \right|
    \le C_\xi \kappa_m(\tau-s),
    \qquad s<t\le \tau,\quad 1\le m\le N.
\]
Moreover, we assume that the following two conditions hold.
\begin{itemize}
    \item[(i)] The local error $R_i$ in \eqref{eq:frac_expansion_ass} satisfies,
\[
    \sum_{i=1}^{J}
    \sum_{l=0}^{n-1}
    \int_{t_{i-1}}^{t_i}
    \left|
        R_{i-1}^{(l)}(s-t_{i-1})
    \right|
    \kappa_{n-l}(T-s)
    \dd s
    \le C_H\delta^\alpha, \qquad 1\le n\le N.
\]

\item[(ii)] The kernels $(\kappa_m)_{m=1}^N$ satisfy for every $\theta_a$ and a.e. $s\in[t_{p-1},t_p]$
\[
    \sum_{i=p+1}^{J}
    \kappa_q(t_i+\theta_a h_i-s)
    \int_{t_{i-1}}^{t_i}
    \kappa_r(T-u)\dd u
    \le
    C_\kappa \kappa_{q+r}(T-s),
\]
for every $1\le p\le J-1$, every $q,r\ge1$ with $q+r\le N$.
\end{itemize}
Here
$\delta:=\max_{[u,v]\in\mathcal P}|v-u|$, and the constants $C_H,C_\xi,C_\kappa$
are independent of $\delta$.
\end{hyp}
\begin{thm}\label{thm:quadratic_aglo}
Let $x$ be a piecewise linear path on some grid $\cP=\{0=t_0<t_1<\dots <t_J=T\}$, and $K$ be a kernel as in Notation \ref{not:comp_sec}. Suppose  Hypothesis~\ref{hyp:fractional_expansion} holds with respect to some $\alpha>0$. Then, the explicit scheme determined by \eqref{eq:higher_order_scheme}-\eqref{eq:coeffcicients_higher_order} is of order $\alpha$, that is 
\begin{equation}\label{eq:rate_algo_thm}
        \max_{n=1, \dots, N}
        \left\vert\VSig{x}{K}^{(n)}_{0,T}-\bv^{(n)}_J\right\vert\leq C \delta^\alpha, \qquad \delta := \max_{[u,v]\in \cP} |v-u|,
\end{equation} where the constant $C$ depends on the constants $C_\kappa,C_H$ from Hypothesis~\ref{hyp:fractional_expansion}.
\end{thm}

\begin{proof}
First let us note that in our notation for piecewise linear paths $x_t=x_{t_j}+ v_j(t-t_j)$ in \eqref{eq:pcsws_linear_x}, we may equivalently write $v_j= \frac{x_{t_{j+1}}-x_{t_j}}{t_{j+1}-t_j}$. In particular, $A_{p}(x_{t_i}-x_{t_{i-1}}) = (t_i-t_{i-1})A_{p}v_i$, so the factor $(t_i-t_{i-1})^n$ arising from the $n$-fold tensor product cancels the $(t-s)^{-n}$ in \eqref{eq:mathcalE}. By definition of $\cE$ in \eqref{eq:mathcalE}, we have \[
\mathcal{E}^{t_j}_{t_{i-1}, t_{i}}\big(A_1(x_{t_i} - x_{t_{i-1}}), \dots, A_q (x_{t_i} - x_{t_{i-1}})\big) = \sum_{n=1}^N \sum_{p_1\cdots p_n\in\cW_q^n}\cK_{t_{i-1},t_i}^{p_1\cdots p_n,t_j} (A_{p_1}v_i) \otimes \cdots \otimes (A_{p_n}v_i).
\] From the identity   \eqref{eq:smpl_identity} in Lemma \ref{prop:sig term over piece segemtn}, it follows therefore that \begin{equation*}
    \mathcal{E}^{t_j}_{t_{i-1}, t_{i}}\big(A_1(x_{t_i} - x_{t_{i-1}}), \dots, A_q (x_{t_i} - x_{t_{i-1}})\big) = \pi_N \left (\VSig{x}{K}_{t_{i-1},t_i}^{t_j}-1 \right ).
\end{equation*}
Writing $\bz = \VSig{x}{K}$, and projecting the exact identity \eqref{eq:exact_Chen_scheme} to levels $1\leq n \leq N$ reads \begin{equation}\label{eq:projected_eq_proof}
{\bz}_{0,t_j}^{(n),t_j}= -\sum_{i=1}^{j} \sum_{l=0}^{n-1} \int_{t_{i-1}}^{t_i}\bz_{0,t_{i-1}}^{(l),s}\otimes \dot \cE_{s,t_i}^{(n-l),t_j}\dd s.
\end{equation}

Let us now prove the claimed convergence rate in \eqref{eq:rate_algo_thm}.
For $n=0$, we have $\bz^{(0)}_{0,T}=\bv^{(0)}_J$ by definition, and for $1\leq n \leq N$, we use~\eqref{eq:projected_eq_proof} to find  \begin{equation}\label{eq:error_levels_proof}
\big |\bz^{(n),T}_{0,T}-\bv_J \big | \leq \sum_{i=1}^J\sum_{l=0}^{n-1} \left |  \int_{t_{i-1}}^{t_i}\left (\bz_{0,t_{i-1}}^{(l),s}  - \sum_{\rho \in \decoSet}\bC_{i-1,\rho}^{(l)}(s-{t_{i-1}})^\rho \right ) \otimes \dot{\cE}_{s,t_i}^{(n-l),T}\dd s\right | 
\end{equation}
Next, adding and subtracting the expansion in \eqref{eq:frac_expansion_ass}, we have
\begin{align}\label{eq:estimator_proof_algo}
    |\bz^{(n),T}_{0,T}-\bv_J \big | 
    & \leq \sum_{i=1}^J\sum_{l=0}^{n-1}   \int_{t_{i-1}}^{t_i}\Big  |\bz_{0,t_{i-1}}^{(l),s}  - \sum_{\rho \in \decoSet}\mathbf{a}_{i-1,\rho}^{(l)}(s-{t_{i-1}})^\rho \Big | \cdot  |\dot{\cE}_{s,t_i}^{(n-l),T}|\dd s 
    \nonumber \\ 
    & \quad +   \sum_{i=1}^J\sum_{l=0}^{n-1} \sum_{\rho \in \decoSet} h_{i-1}^\rho\big  |\mathbf{a}_{i-1,\rho}^{(l)}  - \bC_{i-1,\rho}^{(l)}\big |
    \int_{t_{i-1}}^{t_i}    |\dot{\cE}_{s,t_i}^{(n-l),T}|\dd s\nonumber \\ 
   & =: \Sigma_1+\Sigma_2.
\end{align} Using Hypothesis~\ref{hyp:fractional_expansion} (i), we have \[
\Sigma_1 \leq \sum_{i=1}^J \sum_{l=0}^{n-1} \int_{t_{i-1}}^{t_i}|R^{(l)}_{i-1}(s-t_{i-1})| |\dot \cE_{s,t_i}^{(n-l),T}| \dd s \leq C_HC_\xi\delta^\alpha.
\] 

On the other hand, we know $\bC$ solves the interpolation system \eqref{eq:coeffcicients_higher_order} and that the matrix $(\theta^\rho_a)$ is invertible, so that we have \[
\sum_{\rho \in \decoSet}h_i^\rho |\mathbf{a}_{i,\rho}^{(l)}-\mathbf{C}_{i,\rho}^{(l)}| \lesssim \max_{a=0,\dots,{\decoNum}} |\bz_{0,t_i}^{(l),t_i+\theta_ah_i}-\hat{\bF}^{(l),\theta_a}_i|=:e_i^{(l)},
\] where $\hat{\bF}$ is defined by the right-hand side of \eqref{eq:coeffcicients_higher_order}. 
Let us introduce the notation $\Omega_{l,i}^{\tau}:=\int_{t_{i-1}}^{t_i}\kappa_{l}(\tau-s)\dd s $. Using  Hypothesis~\ref{hyp:fractional_expansion}, we have \[
\Sigma_2 \lesssim \sum_{i=1}^{J} \sum_{l=0}^{n-1}e_{i-1}^{(l)} \Omega_{n-l,i}^T=: \sum_{l=0}^{n-1}B_n^{(l)}. 
\]
Note we can do the same error decomposition \eqref{eq:estimator_proof_algo} again for $e_i^{(l)}$, but with upper-variable $t_i+\theta_ah_i$, which reads \begin{align}\label{eq:scheme_prooof_estimate.}
e_{i-1}^{(l)} & \lesssim \sum_{a=0}^{|\decoSet|-1}\sum_{p=1}^{i-1}\sum_{r=0}^{l-1} \int_{t_{p-1}}^{t_p} (|R^{(r)}_{p-1}(s-t_{p-1})|+e_{p-1}^{(r)})\kappa_{l-r}(t_i+\theta_ah_i-s) \dd s \nonumber \\ & :=d_{i-1}^{(l)}+f_{i-1}^{(l)}
\end{align}

Next, multiplying   $f_{i-1}^{(l)}$ by $\omega_{n-l,i}^T$ and summing over $i$ gives \begin{align*}
    \sum_{a=0}^{|\decoSet|-1}\sum_{i=1}^J \sum_{p=1}^{i-1}\sum_{r=0}^{l-1}e_{p-1}^{(r)}\Omega_{l-r,p-1}^{t_{i}+\theta_ah_i}\Omega_{n-l,i}^{T} & = \sum_{a=0}^{|\decoSet|-1}\sum_{r=0}^{l-1} \sum_{p=1}^{J}e_{p-1}^{(r)} \sum_{i=p+1}^{J}\omega_{l-r,p-1}^{t_{i}+\theta_ah_i}\Omega_{n-l,i}^{T}\\ & \leq C_\kappa \sum_{r=0}^{l-1} \sum_{p=1}^{J}e_{p-1}^{(r)} \Omega_{n-r,p-1}^{T} = C_\kappa \sum_{r=0}^{l-1}B_n^{(r)},
\end{align*} where we used Fubini for the first equality, and Hypothesis~\ref{hyp:fractional_expansion} (ii) for the inequality.
 Doing the same Fubini-trick with $d_{i-1}^{(l)}$ in \eqref{eq:scheme_prooof_estimate.}, leads to \begin{align*}
 & \sum_{a=0}^{|\decoSet|-1}\sum_{r=0}^{l-1} \sum_{p=1}^{J}\int_{t_{p-1}}^{t_p}|R_{p-1}(s-t_{p-1}) \sum_{i=p+1}^{J}   \kappa_{l-r}({t_i+\theta_ah_i}-s)\ \Omega_{n-l,i-1}^{T} \dd s   \\ & \qquad  \leq C_\kappa \sum_{r=0}^{l-1}\sum_{p=1}^{J}\int_{t_{p-1}}^{t_p}|R_{p-1}(s-t_{p-1}) |  \dot \kappa_{l-r,p-1}^{T} \dd s  \leq C_\kappa C_H \delta^{\alpha},
 \end{align*} where we first used (ii) and then (i) from Hypothesis~\ref{hyp:fractional_expansion}.
 We can therefore see that $1\leq l \leq n$, it holds that \[
 B_n^{(l)}  \lesssim \delta^{\alpha}+\sum_{r=0}^{l-1}B_n^{(r)}.
 \] An induction over $1\leq l \leq n$ then shows $$\Sigma_2 \lesssim \sum_{l=0}^{n-1}B_n^{(l)} \lesssim \delta^{\alpha}.$$ Combined with the estimate for $\Sigma_1$, we can conclude the proof.
\end{proof}
Let us now apply our main result to quantify convergence rates for fractional kernels.

\begin{thm}\label{thm:frac_scheme}
    Suppose we are in the setting of Example~\ref{xmpl:fractional_kernel}, that is $q=1$ with fractional kernel $k(t,s)=\Gamma(\beta)^{-1}(t-s)^{\beta-1}$ for some $\beta \in (0,1)$.  Then the naive scheme in Example~\ref{ex:naive_stencil} is of order $\alpha=\beta$, while the higher-order scheme in  Example~\ref{ex:fractional_scheme} is of order $\alpha=1+\beta$.
\end{thm}

    \begin{proof}
We verify Hypothesis~\ref{hyp:fractional_expansion}. From our computations in Example~\ref{xmpl:fractional_kernel}, we see that the choice
\[
    \kappa_m(t):=\frac{t^{m\beta-1}}{\Gamma(m\beta)}, \qquad 1\le m\le N,
\] 
leads to the required control in Hypothesis~\ref{hyp:fractional_expansion}
\[
    \left|\dot{\cE}_{s,t}^{(m),\tau}\right|
    \lesssim
    \kappa_m(\tau-s),
    \qquad s<t\le \tau .
\]
Moreover, since $\kappa_q*\kappa_r=\kappa_{q+r}$, a standard discrete
Beta-convolution estimate gives, uniformly in the interpolation node
$\theta_a$
\[
    \sum_{i=p+1}^{J}
    \kappa_q(t_i+\theta_a h_i-s)
    \int_{t_{i-1}}^{t_i}\kappa_r(T-u)\dd u
    \lesssim
    \kappa_{q+r}(T-s),
\]
for $s\in[t_{p-1},t_p]$ and $q+r\le N$. Thus Hypothesis~\ref{hyp:fractional_expansion} (ii)
holds.

It remains to check that (i) holds true. For the naive scheme we chose 
$\mathcal B=\{0\}$ and the local interpolant is constant, so that
$R_i(u)=F_i(u)-F_i(0)$. The fractional upper-variable regularity gives
$|R_i(u)|\lesssim u^\beta$, uniformly up to level $N$. Therefore, for every
$1\le n\le N$,
\[
    \sum_{i=1}^{J}\sum_{l=0}^{n-1}
    \int_{t_{i-1}}^{t_i}
    |R_{i-1}^{(l)}(s-t_{i-1})|
    \kappa_{n-l}(T-s)\dd s
    \lesssim
    \delta^\beta
    \int_0^T (T-s)^{\beta-1}\dd s
    \lesssim
    \delta^\beta .
\]
Thus Hypothesis~\ref{hyp:fractional_expansion} (i) holds with $\alpha=\beta$,
and Theorem~\ref{thm:quadratic_aglo} yields the claimed order of the naive
scheme.

We now consider the higher-order scheme with $\mathcal B=\{0,\beta,1\}$. For
$l\ge1$, write the local level-$l$ map as
\[
    F_i^{(l)}(u)
    =
    \mathbf z_{0,t_i}^{(l),t_i+u}
    =
    \int_0^{t_i}
    \frac{(t_i+u-r)^{\beta-1}}{\Gamma(\beta)}
    G_l(r)\dd r,
    \qquad
    G_l(r):=\mathbf z_{0,r}^{(l-1),r}\otimes \dot x_r .
\]

We now split $G_l(r)=G_l(t_i)+(G_l(r)-G_l(t_i))$. The constant part gives
\[
    \frac{G_l(t_i)}{\Gamma(\beta+1)}
    \bigl((t_i+u)^\beta-u^\beta\bigr).
\]
Expanding only the regular term $(t_i+u)^\beta$ around $u=0$ yields the
fractional term $-G_l(t_i)u^\beta/\Gamma(\beta+1)$. Thus
\[
    F_i^{(l)}(u)
    =
    A_i^{(l)}
    +
    B_i^{(l)}u^\beta
    +
    C_i^{(l)}u
    +
    \overline R_i^{(l)}(u),
\]
where
\[
    A_i^{(l)}=F_i^{(l)}(0),
    \qquad
    B_i^{(l)}=-\frac{G_l(t_i)}{\Gamma(\beta+1)},
\]
and
\[
    C_i^{(l)}
    =
    \frac{G_l(t_i)t_i^{\beta-1}}{\Gamma(\beta)}
    +
    \frac{\beta-1}{\Gamma(\beta)}
    \int_0^{t_i}
    (t_i-r)^{\beta-2}
    \bigl(G_l(r)-G_l(t_i)\bigr)\dd r .
\]
The last integral is finite because we use the following endpoint estimate:
\[
    |G_l(r)-G_l(t_i)|
    \le L_{l,i}(t_i-r),
    \qquad 0\le r\le t_i,
\]
where
\[
    L_{1,i}\lesssim 1,
    \qquad
    L_{l,i}\lesssim t_i^{(l-1)\beta-1},\quad l\ge2.
\]
Indeed, $G_l(r)=\bz_{0,r}^{(l-1),r}\otimes \dot x_r$ and since $x$ is
smooth
\[
    |\bz_{0,t}^{(l-1),t}|\lesssim t^{(l-1)\beta},
    \qquad
    |\bz_{0,t}^{(l-1),t}-\bz_{0,r}^{(l-1),r}|
    \lesssim t^{(l-1)\beta-1}(t-r),
\]
we obtain the bound for $G_l$. Hence
\[
    |\overline R_i^{(l)}(u)|
    \lesssim
    u^2 t_i^{l\beta-2}
    +
    u^{1+\beta}\Lambda_l(t_i),
    \qquad
    \Lambda_1(t_i)=1,\quad
    \Lambda_l(t_i)=t_i^{(l-1)\beta-1}\text{ for }l\ge2.
\]

By stability of the interpolation operator in the basis $\{1,u^\beta,u\}$, the
actual interpolation error satisfies the same bound as the analytic remainder.
Thus, still denoting the interpolation error by $R_i^{(l)}$, we have, for
$0\le u\le h_i$,
\[
    |R_i^{(l)}(u)|
    \lesssim
    h_i^2t_i^{l\beta-2}
    +
    h_i^{1+\beta}\Lambda_l(t_i),
    \qquad
    \Lambda_1(t_i)=1,\quad
    \Lambda_l(t_i)=t_i^{(l-1)\beta-1}\quad(l\ge2).
\]
Finally, we now insert this estimate into Hypothesis~\ref{hyp:fractional_expansion}. It remains to show that, for every $1\le l\le n-1$,
\[
    \sum_i
    \int_{t_i}^{t_{i+1}}
    |R_i^{(l)}(s-t_i)|\kappa_{n-l}(T-s)\dd s
    \lesssim \delta^{1+\beta}.
\]
We split the estimate according to the two terms in the bound for $R_i^{(l)}$.
For the first part we have 
\[
\begin{aligned}
&\sum_i
\int_{t_i}^{t_{i+1}}
h_i^2t_i^{l\beta-2}\kappa_{n-l}(T-s)\dd s
\lesssim
\delta^{n\beta}
\sum_i
i^{l\beta-2}(J-i+1)^{(n-l)\beta-1}
\lesssim
\delta^{1+\beta}.
\end{aligned}
\]
In the singular case $l\beta\le1$, the sum
is controlled by the left edge and gives
$J^{(n-l)\beta-1}$; hence the contribution is
$\delta^{l\beta+1}\le \delta^{1+\beta}$. In the regular case $l\beta>1$,
the convolution gives the even smaller order $\delta^2$. For the second part, if $l=1$, then $\Lambda_1=1$ and
\[
\sum_i
\int_{t_i}^{t_{i+1}}
h_i^{1+\beta}\kappa_{n-1}(T-s)\dd s
\lesssim
\delta^{1+\beta}.
\]
If $l\ge2$, then
\[
\begin{aligned}
&\sum_i
\int_{t_i}^{t_{i+1}}
h_i^{1+\beta}t_i^{(l-1)\beta-1}\kappa_{n-l}(T-s)\dd s
\lesssim
\delta^{n\beta}
\sum_i
i^{(l-1)\beta-1}(J-i+1)^{(n-l)\beta-1}
\lesssim
\delta^{1+\beta}.
\end{aligned}
\]
Thus
\[
    \sum_i
    \int_{t_i}^{t_{i+1}}
    |R_i^{(l)}(s-t_i)|\kappa_{n-l}(T-s)\dd s
    \lesssim \delta^{1+\beta}.
\]
Summing over the finitely many levels $l=0,\ldots,n-1$ gives
\[
    \sum_i\sum_{l=0}^{n-1}
    \int_{t_i}^{t_{i+1}}
    |R_i^{(l)}(s-t_i)|\kappa_{n-l}(T-s)\dd s
    \lesssim \delta^{1+\beta}.
\]
Hence Hypothesis~\ref{hyp:fractional_expansion} (i) holds with
$\alpha=1+\beta$.
\end{proof}

\subsection{Algorithmic aspects}\label{sec:quad_algo_aspects}
In this section we discuss the algorithmic aspects of \Cref{thm:quadratic_aglo}. The recursion \eqref{eq:higher_order_scheme}-\eqref{eq:coeffcicients_higher_order} directly translates into the iterative scheme summarized in \Cref{alg:higher_order_vsig}. At step $j$ this scheme requires the summation of $j$ truncated tensors of the form $\sum_\rho{\bC}_{i-1,\rho}\otimes_N \mathcal{E}^{t_j;\rho}_{t_{i-1},t_i}(y_i^1,\dots,y_i^q)$, where $y_i^p=A_p(x_{t_i}-x_{t_{i-1}})$, $p=1,\dots,q$, given the coefficient system $\bC$ from the interpolation \eqref{eq:coeffcicients_higher_order}. Hence, the central computational task is the evaluation of $\mathcal{E}^\rho$, or rather, of the products ${\bC}_{i,\rho}\otimes_N\mathcal{E}_{t_{i-1},t_i}^{\tau;\rho}$, for many triples $(t_{i-1},t_i,\tau)$ on the time grid.

In general, no further simplification is available and one is led to a direct implementation of the definition \eqref{eq:mathcalE} (resp. $\cE^\rho$ in \eqref{eq:higher_order_scheme}), i.e.\ a brute-force summation over all words $p_1\cdots p_n\in\cW_q^n$.
However, for kernels whose weights satisfy the symmetry property \Cref{def:symmetric_cK}, one can exploit the resulting permutation invariance of the coefficients $\cK$ to obtain a shuffle-recursive evaluation scheme, which we discuss in \Cref{sec:quad_conv} and summarize in \Cref{alg:quad_evalVtE}. This symmetry holds, in particular, for convolution kernels (cf.\ \Cref{lem:cK_symmetry}) and for piecewise constant kernels (cf.\ \Cref{exmpl:pcsws_constant_matrix}).
Moreover, in the scalar case $q=1$ there is a Horner-type scheme for the direct evaluation of ${\bv}\otimes_N\mathcal{E}^{\rho}$, treated in \Cref{sec:quadq1} and given in \Cref{alg:quad_evalVtE_q1}.
We treat these cases separately below and discuss the associated computational costs.

\begin{algorithm}[p]
\caption{Higher-order approximative algorithm for Volterra signatures}
\label{alg:higher_order_vsig}
\begin{algorithmic}[1]
\STATE \textbf{Input:} grid $0=t_0<t_1<\cdots<t_J\le T$, path values
$\{x_{t_j}\}_{j=0}^J$, kernel data as in Notation~\ref{not:comp_sec}
$\{(k_1,\dots,k_q),(A_1,\dots,A_q)\}$, truncation level $N\in\NN$,
exponent set $\decoSet$, interpolation nodes $\theta_0,\dots,\theta_{\decoNum}$.
\STATE \textbf{Optional:} readout times $\tau_0,\dots,\tau_J$ with
$\tau_j\ge t_j$ \hfill \emph{// default: $\tau_j=t_j$}
\STATE \textbf{Output:} ${\bv}_j\approx
\pi_{\le N}\VSig{x}{K}^{t_j}_{t_0,t_j}$ and
$\widetilde{\bv}_j\approx
\pi_{\le N}\VSig{x}{K}^{\tau_j}_{t_0,t_j}$ for $j=0,\dots,J$.

\STATE \textbf{Precomputation:}
$\widehat{\cK}_{i,j}^{\rho}\approx \cK^{\,t_j;\rho}_{t_{i-1},t_i}$ and
$\widetilde{\cK}_{i,j}^{\rho}\approx \cK^{\,\tau_j;\rho}_{t_{i-1},t_i}$ for
$i\le j$, $\rho\in \decoSet$ \hfill \emph{// (*)}

\FOR{$i=1,\dots,J$}
    \FOR{$p=1,\dots,q$}
        \STATE $y_i^p\gets A_p(x_{t_i}-x_{t_{i-1}})\in\RR^m$
    \ENDFOR
\ENDFOR

\STATE Initialize the coefficient arrays $\mathbf C_{i,\rho}$.

\FOR{$i=1,\dots,J$}
    \FOR{$a=0,\dots,{\decoNum}$}
        \STATE Compute the interpolation value
        $\widehat{\mathbf F}_{i}^{\theta_a}$ by summing the previous-cell
        contributions
        \[
            \widehat{\mathbf F}_{i}^{\theta_a}
            \gets
            1+
            \sum_{b=1}^{i-1}
            \sum_{\rho\in \decoSet}
            \textsc{EvalVtE}
            \bigl(
                \mathbf C_{b,\rho},
                y_b^1,\dots,y_b^q,
                N,
                \cK_{b,i,a}^{\rho}
            \bigr).
        \]
    \ENDFOR
    \STATE Solve componentwise the local interpolation system from
    \eqref{eq:coeffcicients_higher_order},
    \[
        \sum_{\rho\in \decoSet}
        (\theta_a h_i)^\rho\,\mathbf C_{i,\rho}
        =
        \widehat{\mathbf F}_{i}^{\theta_a},
        \qquad a=0,\dots,{\decoNum},
    \]
    to obtain $\mathbf C_{i,\rho}$, $\rho\in \decoSet$.
\ENDFOR

\STATE ${\bv}_0,\widetilde{\bv}_0\gets 1$.

\FOR{$j=1,\dots,J$}
    \STATE ${\bv}_j,\widetilde{\bv}_j\gets 1$
    \FOR{$i=1,\dots,j$}
        \FOR{$\rho\in \decoSet$}
            \STATE ${\bv}_j\gets {\bv}_j+
            \textsc{EvalVtE}
            \bigl(
                \mathbf C_{i,\rho},
                y_i^1,\dots,y_i^q,
                N,
                \widehat{\cK}_{i,j}^{\rho}
            \bigr)$
            \IF{$\tau_j\neq t_j$}
                \STATE $\widetilde{\bv}_j\gets \widetilde{\bv}_j+
                \textsc{EvalVtE}
                \bigl(
                    \mathbf C_{i,\rho},
                    y_i^1,\dots,y_i^q,
                    N,
                    \widetilde{\cK}_{i,j}^{\rho}
                \bigr)$
            \ENDIF
        \ENDFOR
    \ENDFOR
    \IF{$\tau_j=t_j$}
        \STATE $\widetilde{\bv}_j\gets {\bv}_j$
    \ENDIF
\ENDFOR
\end{algorithmic}
\vspace{1mm}
\par\noindent\raggedright
\emph{(*)} The coeffcients $\widehat{\cK}_{i,j}, \widetilde{\cK}_{i,j}\in T^N(\RR^{q})$ depend only on the time grid (and the chosen readout times $\{\tau_j\}$) and can be precomputed whenever analytic formulas are available (cf.\ \Cref{sec:kernel_computations}). In the convolution case, symmetry implies that it suffices to compute these coefficients for a reduced subset of words (cf.\ \Cref{sec:quad_conv}).
The computation of the decorated coefficients $\widehat{\cK}^{\rho}_{i,j}, \widetilde{\cK}^{\rho}_{i,j}\in T^N(\RR^{q})$ follows analogously (cf. \Cref{rem:generalized_cE}).
\\
\emph{Parallelization.} The coefficient loop over $i$ is sequential. For fixed
$i$, the interpolation values for the different nodes $\theta_a$ can be
computed in parallel, and for fixed $(i,\theta_a)$ the previous-cell
contributions $b<i$ can be parallelized and summed. After the coefficients have
been computed, the final readout values $\mathbf v_j$ can be computed
independently for different $j$, again with parallelization over the
contributing cells $i<j$ and exponents $\rho\in \decoSet$.
\end{algorithm}

\subsubsection{Symmetric weights for $q > 1$}\label{sec:quad_conv}
We now specialize to the case $q>1$ under the symmetry assumption
\Cref{def:symmetric_cK}. There is no essential difference between the
undecorated and $\rho$-dependent coefficient tensors in this discussion:
the exponent $\rho$ only changes the scalar coefficients, while the
permutation symmetries and tensor operations are unchanged; see
\Cref{rem:generalized_cE}. We therefore discuss the undecorated
case without loss of generality.

In this setting, the permutation invariance of the coefficients $\cK$ yields a
shuffle-polynomial representation of $\mathcal{E}$ and, consequently, an
efficient recursive evaluation scheme for ${\bv}\otimes_N\mathcal{E}$. This
symmetry holds, in particular, for convolution kernels
(cf.\ \Cref{lem:cK_symmetry}) and for piecewise constant kernels
(cf.\ \Cref{exmpl:pcsws_constant_matrix}). Specifically, we work under the
following symmetry assumption.

\begin{hyp}\label{def:symmetric_cK}
The kernels $k_1, \dots, k_q$ are such that for all $n\in\mathbb N$, all permutations $\sigma\in\mathcal S_n$ of $\{1,\dots,n\}$ and all $(s,t,\tau)$ the coefficients $\cK$ in \eqref{eq:def_calK}-\eqref{eq:recast_Y_2} satisfy
$$
\cK_{s,t}^{wp,\tau}=\cK_{s,t}^{\sigma(w)p,\tau},
\qquad w\in\mathcal W_q^n,\ p\in\mathcal A_q.
$$
This implies that for each fixed $(s,t,\tau) \in \Delta^3$ and each $p\in\mathcal{A}_q$ we can represent $\mathcal{K}$ by a function $\alpha_p: \NN^{q} \to \RR$ such that
\begin{align}\label{def:alpha}
    \frac{\cK_{s,t}^{p_1 \cdots p_{n-1} p,\tau}}{(t-s)^{n+1}} = \alpha_p(\ell_1, \dots, \ell_q), \qquad \ell_{j} := \big\vert\{i = 1, \dots, n-1\;\vert\; p_i = j\}\big\vert,\quad (j\in \mathcal{A}_q).
\end{align}
\end{hyp}
In \Cref{lem:cK_symmetry} it is proven that the Hypothesis~\ref{def:symmetric_cK} holds for all kernels of convolutional type, but it also holds for the piecewise constant kernel in \Cref{exmpl:pcsws_constant_matrix} and for $q=1$ in general.
To reveal the form of $\mathcal{E}$ that allows a more efficient computation we will need the following
\begin{nota}\label{not:multii_and_shuffle}
    For multindices $\ell = (\ell_1, \dots, \ell_q) \in \NN^q$ we define $|\ell| = \ell_1 + \dots +\ell_q$ and $\ell! = \ell_1! \cdots \ell_q!$.
    Further, for any $\ell  \in \NN^q$ and $p\in\mathcal{A}_q$ we set
    $$\ell +  1_p := (\ell_1, \dots, \ell_p + 1, \dots, \ell_q),$$
    and similar for any $p\in\mathcal{A}_q$ with $\ell_p > 0$ we set
    $$\ell -  1_p := (\ell_1, \dots, \ell_p - 1, \dots, \ell_q).$$
\end{nota}
The following result shows that we can evaluate $\mathcal{E}(y_1, \dots, y_q)$ in terms of a shuffle polynomial over the variables $(y_1, \dots, y_q)$.
As discussed in \Cref{rem:mult_horner} below, and for the case $q=1$ in \Cref{sec:quadq1}, this allows for the employment of efficient Horner-type evaluation schemes.
\begin{prop}\label{prop:mathcalE}
Assume that \Cref{def:symmetric_cK} is satisfied and for fixed $(s,t,\tau)\in \Delta^3$ let
$\alpha_p:\NN^q \to \RR$ $(p\in \mathcal{A}_q)$ denote the corresponding representation of $\mathcal{K}$ as in \eqref{def:alpha}.
For any $y_1, \dots, y_q \in \RR^m$ the coefficient $\cE_{s,t}^\tau(y_1,\dots,y_n)$ in \eqref{eq:mathcalE} can be computed as
\begin{equation}\label{eq:cE_shuffle}
    \mathcal{E}^{\tau}_{s,t}(y_1, \dots, y_q)
    = \sum_{p\in\cA_q} \left(\sum_{|\ell| \le N - 1} \frac{\alpha_p(\ell)}{\ell !} y_{1}^{\shuffle \ell_1} \shuffle \cdots \shuffle y_{q}^{\shuffle \ell_q} \right)\otimes y_{p},
\end{equation}
where $\shuffle$ denotes the shuffle product introduced in \Cref{def:shuffle}.
\end{prop}

\begin{algorithm}[h]
\caption{Subroutine \textsc{EvalVtE}: convolution case}\label{alg:quad_evalVtE}
\begin{algorithmic}[1]
\STATE \textbf{Input:} ${\bv}\in T^N(\RR^m)$, $y_1,\dots,y_q\in\RR^m$, truncation level $N\in\NN$
\STATE \textbf{Input:} coefficients $\alpha_p(\ell)=\cK^{w(\ell)p,\tau}_{s,t}(t-s)^{-|\ell|-2}$ for $p=1,\dots,q$ and all $\ell\in\NN^q$ with $|\ell|\le N-1$
\STATE \textbf{Output:} ${\bv}' = {\bv}\otimes_N\mathcal{E}^{\tau}_{s,t}(y_1,\dots,y_q)\in T^N(\RR^m)$

\FOR{$n=N-1,\dots,0$}
    \FOR{each $\ell\in\NN^q$ with $|\ell|=n$}
        \FOR{$p=1,\dots,q$}
            \STATE $E_p(\ell)\gets \alpha_p(\ell)/\ell!$
        \ENDFOR
        \IF{$n\le N-2$}
            \FOR{$r=1,\dots,q$}
                \FOR{$p=1,\dots,q$}
                    \STATE $E_p(\ell)\gets E_p(\ell)+\big(E_p(\ell+1_r)\shuffle y_r\big)(\ell_r+1)/(|\ell|+1)$
                \ENDFOR
            \ENDFOR
        \ENDIF
    \ENDFOR
\ENDFOR

\STATE $\mathcal{E}\gets 0$
\FOR{$p=1,\dots,q$}
    \STATE $\mathcal{E}\gets \mathcal{E}+E_p(0,\dots,0)\otimes y_p$
\ENDFOR
\STATE ${\bv}'\gets {\bv}\otimes_N \mathcal{E}$
\end{algorithmic}
\vspace{1mm}
\par\noindent\raggedright
\emph{Parallelization.} The recursion in the total degree $n$ is sequential. For fixed $n$, the updates over all multi-indices $\ell$ with $|\ell|=n$ are independent and can be parallelized. For fixed $\ell$, the loops over $r$ and $p$ are also parallelizable.
\end{algorithm}
\begin{proof}[Proof of \Cref{prop:mathcalE}]
    As in the statement we let $(s,t,\tau) \in \Delta^3$ and $y_1, \dots, y_q$ be fixed.
    For each multi-index $\ell = (\ell_1, \dots, \ell_q) \in \NN^{q}$ define the subset of words that each have $\ell_p$ occurrences of the letter $p\in\mathcal{A}_q$: 
    \begin{equation}\label{eq:def_words_by_ell}
        \mathcal{W}^{\ell}_q = \Big\{ p_1\cdots p_{|\ell|} \in \mathcal{W}^{|\ell|}_q \; \Big\vert \; \big\vert\{i = 1, \dots, n\;\vert\; p_i = p\}\big\vert = \ell_p, \; p \in \mathcal{A}_q\Big\}. 
    \end{equation}
    We then notice that by \Cref{def:symmetric_cK} it holds the following.
    Starting from \eqref{eq:mathcalE} and separating the last letter of each word $p_{1}\cdots p_{n}\in\mathcal{W}_{q}^{n}$ by writing $p_{n}=p$, we sum over $p\in\mathcal{A}_{q}$ and the prefix $p_{1}\cdots p_{n-1}$.
    By \Cref{def:symmetric_cK}, specifically \eqref{def:alpha}, the coefficient associated to a prefix $p_{1}\cdots p_{n-1}$ depends on the prefix only through the multi-index $\ell$ recording how many times each letter appears.
    Grouping all prefixes with the same $\ell$ into the set $\mathcal{W}_{q}^{\ell}$ introduced in \eqref{eq:def_words_by_ell} and factoring out the common value $\alpha_{p}(\ell)$ therefore gives
    \begin{equation}\label{eq:proof_permutation}
    \pi_{n}\mathcal{E}^{\tau}_{s,t}(y_1, \dots, y_q) =
    \sum_{p=1}^q \sum_{|\ell| = n-1}\alpha_p(\ell)\sum_{p_1\cdots p_{n-1} \in \mathcal{W}^{\ell}_q}  y_{p_1} \otimes \cdots \otimes y_{p_{n-1}}\otimes y_{p}. 
    \end{equation}
    Next, for all $\ell\in\NN^q$ define 
    \begin{align*}
        Y(\ell) = \sum_{p_1\cdots p_{|\ell|} \in \mathcal{W}^{\ell}_q}  y_{p_1} \otimes \cdots \otimes y_{p_{|\ell|}},
    \end{align*}
    so that 
    \begin{equation}\label{eq:proof_permutation-2}
    \pi_{n}\mathcal{E}^{\tau}_{s,t}(y_1, \dots, y_q) =
    \sum_{p=1}^q \sum_{|\ell| = n-1}\alpha_p(\ell)\,Y(\ell)\otimes y_{p}. 
    \end{equation}
    Then by induction it follows that 
    \begin{equation}\label{eq:proof_permutation-3}
    Y(\ell) = y_{1}^{\otimes \ell_1} \shuffle \cdots \shuffle y_{q}^{\otimes \ell_q}.
    \end{equation}
    Indeed, one readily verifies that this is the case for $|\ell| \in \{0,1\}$.
    Then using the commutativity of the shuffle product we see that both expressions satisfy the recursion
    \begin{align*}
        Y(\ell) = \sum_{p=1}^q 1_{\ell_p > 0}Y(\ell - 1_p) \otimes y_p.
    \end{align*}
    The final expression then follows after noting that $y^{\shuffle k} = k! y^{\otimes k}$.
    Indeed, applying this identity to each factor in~\eqref{eq:proof_permutation-3} gives $y_{p}^{\otimes \ell_{p}} = \frac{1}{\ell_{p}!} y_{p}^{\shuffle \ell_{p}}$, so that
    \begin{equation*}
        Y(\ell) = y_{1}^{\otimes \ell_1} \shuffle \cdots \shuffle y_{q}^{\otimes \ell_q} = \frac{1}{\ell!}\, y_{1}^{\shuffle \ell_1} \shuffle \cdots \shuffle y_{q}^{\shuffle \ell_q}.
    \end{equation*}
    Substituting this into \eqref{eq:proof_permutation-2} yields \eqref{eq:cE_shuffle}.
    \end{proof}

\begin{rem}\label{rem:mult_horner}
The term in parentheses in \Cref{eq:cE_shuffle} is a shuffle polynomial of degree $N-1$ in the variables $y_1,\dots,y_q\in\RR^m$. A straightforward evaluation is implemented in \Cref{alg:quad_evalVtE}. This scheme already reduces the computational cost for evaluating $\mathcal{E}^{\tau}_{s,t}(y_1,\dots,y_q)$ from the brute-force order $m^Nq^N$ to the order $Nm^{N}$, asymptotically in $N$ for fixed $m$ and $q$ (see the discussion of computational costs in \Cref{sec:cost_quad}). More elaborate multivariate Horner-type factorizations are available in the literature (see, e.g., \cite{pena2000multivariate}), however, we find that for fixed $q$ and $m$ the asymptotic order of costs in terms of $N$ cannot be further improved in this way.

Moreover, for each $p\in \mathcal{A}_q$ the term in parentheses in \Cref{eq:cE_shuffle} is a linear combination of shuffle monomials
$y_{1}^{\shuffle \ell_1}\shuffle\cdots\shuffle y_{q}^{\shuffle \ell_q}$ and hence a symmetric tensor.
Consequently, one may evaluate and store these intermediate symmetric tensors in a compressed representation, which can reduce memory consumption and the cost of evaluating $\mathcal{E}^{\tau}_{s,t}(y_1,\dots,y_q)$.
However, when subsequently forming $\bv\otimes \mathcal{E}^{\tau}_{s,t}(y_1,\dots,y_q)$, the components are in general no longer symmetric, so that a passage to the full tensor basis is unavoidable.
Thus, while symmetry can be exploited to save memory during the evaluation of \Cref{eq:cE_shuffle}, it seems not to improve the overall asymptotic costs once the final product with $\bv$ is formed.
\end{rem}

\subsubsection{The case $q=1$}\label{sec:quadq1}
A Horner scheme for a fused exponentiation-and-multiplication step was introduced in the popular \emph{signatory} package \cite{kidger2021signatory} for the classical signature case. Here, for the case $q=1$, we can employ a similar Horner-type scheme for the direct evaluation of ${\bv}\otimes_N \mathcal{E}_{s,t}^\tau(y)$, which further improves overall efficiency. The resulting fused evaluation is summarized in \Cref{alg:quad_evalVtE_q1}.

\begin{algorithm}[h]
\caption{Subroutine \textsc{EvalVtE}: scalar case $q=1$ (Horner scheme)}\label{alg:quad_evalVtE_q1}
\begin{algorithmic}[1]
\STATE \textbf{Input:} ${\bv}\in T^N(\RR^m)$, $y\in\RR^m$, truncation level $N\in\NN$
\STATE \textbf{Input:} coefficients $\beta_n=\kappa^{n,\tau}_{s,t}(t-s)^{-n}$ for $n=1,\dots,N$
\STATE \textbf{Output:} ${\bv}'= {\bv}\otimes_N\mathcal{E}^{\tau}_{s,t}(y)\in T^N(\RR^m)$

\STATE ${\bv}'^{(0)}\gets 0$
\FOR{$n=1,\dots,N$}
    \STATE $W\gets {\bv}^{(0)}\beta_n$
    \FOR{$k=1,\dots,n-1$}
        \STATE $W\gets (W\otimes y)+{\bv}^{(k)}\beta_{n-k}$
    \ENDFOR
    \STATE ${\bv}'^{(n)}\gets W\otimes y$
\ENDFOR
\end{algorithmic}
\vspace{1mm}
\par\noindent\raggedright
\emph{Parallelization.} The loop over $n$ is independent and can be parallelized. For fixed $n$, the recursion over $k$ is sequential.
\end{algorithm}

\subsection{FFT acceleration for convolution kernels}\label{sec:fft}

We now explain how the quadratic summation in
\Cref{alg:higher_order_vsig} can be accelerated on uniform grids for
convolution kernels. We restrict the derivation to the scalar case $q=1$ and the diagonal readout $\tau=t_j$; the
multi-component case $q>1$ is analogous to previous derivations, using the symmetric-weight representation
from \Cref{sec:quad_conv}.

\begin{hyp}\label{hyp:fft_convolution}
Assume for a fixed $h>0$ that $t_j=t_0+jh$ $(j=1, \dots, J)$ and that the kernel $K\in\Lkernel$ is of convolution form, i.e., for some $\bar{K}: [0, \infty)\to \mathcal{L}(\RR^d, \RR^m)$ it holds
\[
    K(t,s)=\bar K(t-s), \qquad (s,t)\in \Delta^2.
\]
\end{hyp}

Consider a fixed exponent set
$\decoSet=\{\rho_0,\dots,\rho_{\decoNum}\}$ and interpolation nodes
$\{\theta_0,\dots,\theta_{\decoNum}\}$.
Then under \Cref{hyp:fft_convolution}, the local coefficient tensors appearing in
\eqref{eq:higher_order_scheme} and \eqref{eq:coeffcicients_higher_order} only depend the number of interval lags. More precisely, for $\rho\in\decoSet$, tensor level $n\in\{1, \dots, N\}$,
and interpolation node $\theta_a$ $(a = 0, \dots, \decoNum)$ it holds
\[
    \pi_n\cE^{t_j+\theta_a h;\rho}_{t_{i-1},t_{i}}(y)
    ~=~ \pi_n\cE^{(j-i + \theta_a)h;\rho}_{0, h}(y) ~=: 
    \omega^{a,\rho}_{n}(j-i)y^{\otimes n},
    \qquad 1\leq i<j\leq J,
\]
where $\cE^{\tau;\rho}_{s,t}$ was defined in \eqref{eq:mathcalE} and \eqref{eq:higher_order_scheme}. 

Let $y_i:=A(x_{t_{i+1}}-x_{t_i})$. At tensor level $n$, the interpolation
equation \eqref{eq:coeffcicients_higher_order}, \emph{written levelwise}, then takes
the form
\[
    \pi_n\widehat{\mathbf F}_{j}^{\theta_a}
    =
    \sum_{i=0}^{j-1}
    \sum_{\rho\in\decoSet}
    \sum_{k=1}^{n}
    \omega^{a,\rho}_{k}(j-i)
    \left(
        \pi_{n-k}\mathbf C_{i,\rho}
        \otimes
        y_i^{\otimes k}
    \right).
\]
Thus, for each fixed $(a,\rho,k,n)$, the sum over $i$ is a causal convolution
in the lag variable. After zero-padding, these causal convolutions are
evaluated by FFT \cite{cooley1965algorithm}. Once the values
$\{\widehat{\mathbf F}_{j}^{\theta_a}\}_{a=0}^{\decoNum}$ are known, the
coefficients $(\mathbf C_{j,\rho})_{\rho \in \decoSet}$ are recovered componentwise from the local interpolation
system \eqref{eq:coeffcicients_higher_order},
\[
    \sum_{\rho\in\decoSet}
    (\theta_a h)^\rho\mathbf C_{j,\rho}
    =
    \widehat{\mathbf F}_{j}^{\theta_a},
    \qquad a=0,\dots,\decoNum.
\]
The final diagonal readouts $\bv_j$ are obtained from the same convolution
formula, using the ${\theta =0}$ lag weights $(\omega^{0,\rho}_{n}(j))_{j=0, \dots, J-1}$.
The resulting procedure is summarized in \Cref{alg:fft}. 

\begin{algorithm}[h]
\caption{FFT-accelerated higher-order Volterra signature algorithm}
\label{alg:fft}
\begin{algorithmic}[1]
\STATE \textbf{Input:} uniform grid $t_j=t_0+jh$, increments $\Delta x_i$,
kernel data $(\bar K,A)$, truncation level $N$, exponent set
$\decoSet=\{\rho_0,\dots,\rho_{\decoNum}\}$, nodes
$\theta_0,\dots,\theta_{\decoNum}$.
\STATE \textbf{Output:} $\bv_j\approx
\pi_{\leq N}\VSig{x}{K}^{t_j}_{t_0,t_j}$, $j=0,\dots,J$.
\STATE Here $\operatorname{FFT}_L$ and $\operatorname{IFFT}_L$ denote the fast
Fourier transform subroutine of length $L=2^\ell$ and its inverse.

\STATE $y_i\gets A\Delta x_i$, $i=0,\dots,J-1$.
\STATE Choose $L\ge 2J$ and precompute the zero-padded FFTs
\[
    \widehat\omega^{a,\rho}_{r}
    :=
    \operatorname{FFT}_L
    \bigl(0,\omega^{a,\rho}_{r}(1),\dots,\omega^{a,\rho}_{r}(J)\bigr)
\]
for $a=0,\dots,\decoNum$, $\rho\in\decoSet$, $r=1,\dots,N$, and for the
readout weights corresponding to $\theta=0$.
\STATE $\pi_0\mathbf C_{i,0}\gets 1$, \quad
$\pi_0\mathbf C_{i,\rho}\gets 0$ for $\rho\neq0$.

\FOR{$n=1,\dots,N$}
    \FOR{$a=0,\dots,\decoNum$}
        \STATE $\pi_n\widehat{\mathbf F}_{j}^{\theta_a}\gets 0$,
        $j=0,\dots,J-1$.
        \FOR{$\rho\in\decoSet$, $r=1,\dots,n$}
            \STATE $G_i^{n,r,\rho}\gets
            \pi_{n-r}\mathbf C_{i,\rho}\otimes y_i^{\otimes r}$,
            $i=0,\dots,J-1$.
            \STATE $\widehat G^{n,r,\rho}\gets
            \operatorname{FFT}_L(G_0^{n,r,\rho},\dots,G_{J-1}^{n,r,\rho},0,\dots,0)$.
            \STATE $\pi_n\widehat{\mathbf F}_{j}^{\theta_a}
            \gets
            \pi_n\widehat{\mathbf F}_{j}^{\theta_a}
            +
            \left[
            \operatorname{IFFT}_L
            \bigl(
                \widehat\omega^{a,\rho}_{r}
                \widehat G^{n,r,\rho}
            \bigr)
            \right]_j$,
            $j=0,\dots,J-1$.
        \ENDFOR
    \ENDFOR
    \FOR{$j=0,\dots,J-1$}
        \STATE Solve componentwise
        \[
            \sum_{\rho\in\decoSet}
            (\theta_a h)^\rho\,\pi_n\mathbf C_{j,\rho}
            =
            \pi_n\widehat{\mathbf F}_{j}^{\theta_a},
            \qquad a=0,\dots,\decoNum.
        \]
    \ENDFOR
\ENDFOR

\STATE $\bv_0\gets 1$.
\FOR{$n=1,\dots,N$}
    \STATE Compute $(\pi_n\bv_j)_{j=1}^{J}$ by the same FFT multiplication
    and inverse FFT step, using the readout weights corresponding to
    $\theta=0$.
\ENDFOR

\STATE \textbf{return} $\bv_0,\dots,\bv_J$.
\end{algorithmic}
\vspace{1mm}
\par\noindent\raggedright
For $q>1$, the same algorithm applies after replacing the scalar source
$y_i^{\otimes r}$ by the shuffle-polynomial terms
$M_\ell(y_i)\otimes y_i^p$ from \Cref{sec:quad_conv}, where
\[
    M_\ell(y_i)
    =
    \frac{1}{\ell!}
    (y_i^1)^{\shuffle \ell_1}
    \shuffle\cdots\shuffle
    (y_i^q)^{\shuffle \ell_q},
\]
with $p=1,\dots,q$ and $|\ell|=r-1$. The tensor-level loop is sequential,
while the FFT convolutions over interpolation nodes, exponents, tensor orders,
and shuffle-polynomial terms are independent.
\end{algorithm}

\subsection{Kernel Computations}\label{sec:kernel_computations}

So far, we have treated the weights $\cK$ and $\dot{\cK}$ defined in~\eqref{eq:weights_cor} as given, and derived exact and approximate decompositions of Volterra signatures in terms of them. In this section, we turn to explicit computations of these weights for various classes of kernels~$K$.
\subsubsection{Piecewise constant kernels}
We begin with the class of piecewise constant kernels, that is kernels $K$ of the form \eqref{eq:decomposed_kernel}, where for all $1\leq r \leq q$ we have \begin{equation}\label{eq:piecewise_constant_general}
    k_r(t,s) = \sum_{i,j=1}^J(b_i^j)^r1_{[t_{i-1},t_i)}(s)1_{[t_{j-1},t_j)}(t),
\end{equation} where $0=t_0 < \cdots <t_J=T$ is a partition of $[0,T]$, and $\{(b_i^j)^r: 1\leq i,j \leq J, 1 \leq r \leq q \}$ a family of real valued coefficients. We treat the scalar and general cases separately. In both cases, the proof is straightforward and therefore omitted.
\begin{prop}[Scalar case]\label{exmpl:pcws_cnstnt}
Suppose $q=1$ and $K$ is a piecewise constant kernel with respect to the kernel $k=k_1$ given in \eqref{eq:piecewise_constant_general}, and coefficients $b_i^j=(b_i^j)^1$. Recall that the function $\dot{\kappa}^n$ is defined by \eqref{eq:kernel_cor} in Corollary \ref{cor:scalar_expansion}, and also recall that in~\eqref{eq:kappa_general} we have defined 
\begin{equation}\label{eq:kappa_pcwcst}\kappa_{s,t}^{n,\tau} = \int_s^t\dot{\kappa}_{u,t}^{n,\tau} \dd u.
    \end{equation} Then for any $(s,t,\tau) \in \Delta^3$ with $(s,t) \in [t_{i-1},t_i]$ and $\tau \in (t_j, t_{j+1}]$, we have
    \begin{align*}
        \dot{\kappa}^{n, \tau}_{s, t} = \frac{(t-s)^{n-1}}{(n-1)!}(b_{i}^i)^{n-1} b_{i}^j, \qquad\quad \kappa^{n, \tau}_{s, t} = \frac{(t-s)^{n}}{n!}(b_{i}^i)^{n-1} b_{i}^j.
    \end{align*}
   Moreover, we obtain the  factorization \eqref{eq:combinatorial_vsig}-\eqref{eq:weights_cor} with 
    \begin{align*}
        \mu^{n_1, \dots, n_k}_{i_1, \dots, i_{k}} = \kappa^{n_1, t_{i_{2}}}_{t_{i_1-1}, t_{i_1}}  \cdots \kappa^{n_k, t_{i_{k+1}}}_{t_{i_k-1}, t_{i_k}}.
    \end{align*}

\end{prop}
\begin{rem}
    Notice that due to the exact factorization observed in the last identity, the algorithm presented in \Cref{sec:algo_thm} with $\decoSet = \{0\}$ provides an exact scheme to compute the Volterra signature for piecewise constant kernels and piecewise linear paths. This can be seen directly from the fact that the maps $F_i(u):=\bz_{0,t_i}^{t_i+u}$ are piecewise constant: \[ \forall \tau \in (t_{i-1},t_i] \quad   \bz^{\tau} \equiv \bz^{t_{i-1}} \quad \Longrightarrow \quad 
    [\bz_{0,t_{i-1}}^{(l)} \oast \bz^{(n-l)}]_{t_{i-1},t_i}^{t_j} = \bz_{0,t_{i-1}}^{(l),t_{i-1}} \otimes \bz_{t_{i-1},t_i}^{(n-l),t_j},
    \] so that \eqref{eq:naive_quadratic_stencil} is exact. Moreover, it is interesting to note that when the kernel is piecewise constant on a coarser grid than $x$, say $t_0 = t_{m_0} < t_{m_1} < \dots < t_{m_L} = t_J$ $(L\le J)$, then we have for any $1 \le k < j \le L$ and $\tau \in (t_{m_{j-1}}, t_{m_j}]$ that
    \begin{align*}
        \VSig{x}{k}^{\tau}_{t_{m_{k-1}},t_{m_{k}}} &=~  \sum_{k=1}^{j}
    \;\sum_{\substack{m_{k-1}<i_1<\dots< i_k \le m_k}}
    (b_{k}^k)^{n-1} b_{k}^j \frac{(t_{i}-t_{i-1})^{n}}{n!}(v_i)^{\otimes n}
    \\
    &=~ \Big.\mathcal{D}_{b_k^k, b_k^j}\Sig{x}_{t_{m_{k-1}},t_{m_{k}}},
    \end{align*}
    where for every $a,b \in \RR$, the mapping $\cD_{a,b}$ is defined from $T((\RR^m))$ to $T((\RR^m))$ as $$\mathcal{D}_{a, b}\,\bx= (\bx^{(0)},b\bx^{(1)},ba\bx^{(2)},\dots, ba^{n-1}\bx^{(n)},\dots).$$
\end{rem}
Let us now generalize the last corollary to the general case $q>1$, where the same remark applies, that is the algorithm in \Cref{sec:algo_thm} is exact.

\begin{prop}[General case]\label{exmpl:pcsws_constant_matrix}
    Suppose now that $K$ is a general piecewise constant kernel as displayed in~\eqref{eq:piecewise_constant_general}. Then the kernels $\cK$ and $\dot{\cK}$, respectively defined by \eqref{eq:cK_again} and \eqref{eq:def_dot_cK} are expressed in the following way for $(s,t,\tau) \in \Delta^3$ with $s,t \in [t_{i-1},t_i)$ and $\tau \in (t_j, t_{j+1}]$:
    \begin{align*}
        \dot{\mathcal{K}}^{p_1\cdots p_n, \tau}_{s, t} = \frac{(t-s)^{n-1}}{(n-1)!}(b_{i}^i)^{p_1}\cdots (b_{i}^i)^{p_{n-1}}(b_{i}^j)^{p_n}, \qquad\quad {\mathcal{K}}^{p_1\cdots p_n, \tau}_{s, t} = \frac{t-s}{n} \, \dot{\mathcal{K}}^{p_1\cdots p_n, \tau}_{s, t}.
    \end{align*} Moreover, the coefficients $\cK^w_{i_1\cdots i_n}$ in the decomposition \eqref{eq:weighted_sum_thm} are written as
    \begin{align*}
        \mathcal{K}^{(w_1, \dots, w_k)}_{i_1, \dots, i_{k+1}} = \prod_{l=1}^k  {\mathcal{K}}^{w_l, t_{i_{l+1}}}_{t_{i_l-1}, t_{i_l}}.
    \end{align*}
    \end{prop}

\subsubsection{Multivariate fractional kernels}\label{sec:apx_fractional}
For our second class of kernels we generalize the scalar \Cref{xmpl:fractional_kernel} of fractional kernels to the multivariate setting. Specifically, we consider $K$ of the form \eqref{eq:decomposed_kernel}, where each $k_r$ is given by \eqref{eq:1d_frac_def}:
    \begin{equation}\label{eq:multivariate_frac_def}k_r(t,s) = k_{\beta_r}(t-s)= \Gamma(\beta_r)^{-1}(t-s)^{\beta_r-1},\end{equation} for a collection $\{\beta_r:r=1,\dots,q \}$ of coefficients $\beta_r>0$.
\begin{prop}\label{ex:multivariate_fractional}
     Consider a multivariate fractional kernel $K$ expressed as \eqref{eq:multivariate_frac_def}. For any word $p_1\cdots p_n \in \cW_q$, we have  %
\begin{equation}\label{eq:K_multivariate_frac}
    \cK^{p_1\cdots p_{n},\tau}_{s,t}  =\frac{I_{\frac{t-s}{\tau-s}}(\sum_{l=1}^{n-1}\beta_{p_l}+1,\beta_{p_{n}})(\tau-s)^{\sum_{l=1}^{n}\beta_{p_l}}}{\Gamma(\sum_{l=1}^{n}\beta_{p_l}+1)}
\end{equation} and consequently %
    \begin{equation}\label{eq:Kdot_multivariate_frac} \dot{\cK}_{s,t}^{p_1\cdots p_{n},\tau} = \frac{I_{\frac{t-s}{\tau-s}}(\sum_{l=1}^{n-1}\beta_{p_l},\beta_{p_{n}})(\tau-s)^{\sum_{l=1}^{n-1}\beta_{p_l}-1}}{\Gamma(\sum_{l=1}^{n}\beta_{p_l})},\end{equation} in terms of the regularized incomplete Beta function introduced in \eqref{eq:regul_beta_def} in Example \ref{xmpl:fractional_kernel}.
\end{prop}
\begin{rem}\label{rem:multivariate_frac_weights}
     Similar to the scalar case, we could further write down the semi-explicit weights $\cK^{(w_1,\dots,w_k)}_{p_1,\dots,p_{k+1}}$, which, however, do not admit a factorization as seen in Example \ref{xmpl:fractional_kernel}. As emphasized before, the formulas for $\cK$ already allow the application of the algorithm in Section \ref{sec:algo_thm}, as the corresponding $\cE$ can be evaluated explicitly.
\end{rem}
\begin{proof}[Proof of Proposition~\ref{ex:multivariate_fractional}]
     
We recall that the fractional Riemann-Liouville integral operator of order $\beta > 0$ is defined by
$$I^\beta(f)_{s,t}=\int_s^t f(t) k_\beta(t-r)\dd r, \qquad 0 \le s \le t \le T, \quad f\in L^1([0,T]).$$
It is well known that this family of operators satisfy the the semi-group relation $$I^\alpha_{s,\cdot }\circ I^\beta_{s,\cdot} = I^{\alpha +\beta}_{s,\cdot },$$ see, e.g.,  \cite[Section~2.3]{samko1993fractional}. On the other hand, we have 
$$
    I^{\beta}(1)_{s,t} 
    = \frac{1}{\Gamma(\beta)}\int_s^t (t-r)^{\beta - 1}\dd r 
    = \frac{(t-s)^{\beta}}{\beta \Gamma(\beta)} 
    =  \frac{(t-s)^{\beta}}{\Gamma(\beta +1)} = k_{\beta+1}(t-s).
$$
Combining these observations, for any word $p_1\cdots p_n \in \cW_q$ we obtain the following relation on the diagonal $t=\tau$:
\[
\cK^{p_1\cdots p_n,t}_{s,t} = I^{\sum_{l=1}^n\beta_{p_l}}(1)_{s,t} =  \frac{(t-s)^{\sum_{l=1}^n\beta_{p_l}}}{\Gamma(\sum_{l=1}^n\beta_{p_l} + 1)} = k_{\sum_{l=1}^n\beta_{p_l} + 1}(t-s).
\]
We now move to the off-diagonal case $t < \tau$. By definition \eqref{eq:recast_Y_2} in Lemma~\ref{prop:sig term over piece segemtn} and the last identity, for any word $p_1\cdots p_{n+1} \in \cW_q$ we have 
\begin{equation}\label{eq:app_frac}
\cK^{p_1\cdots p_{n+1},\tau}_{s,t} = \int_s^tk_{\beta_{p_{n+1}}}(\tau-u) \cK_{s,u}^{p_1\cdots p_n,u} \dd u = \int_s^tk_{\beta_{p_{n+1}}}(\tau-u) k_{\sum_{l=1}^n\beta_{p_l}+1}(u-s) \dd u.
\end{equation}
Applying the change of variables $v = \frac{u-s}{\tau-s}$ and we find \begin{align}
    \cK^{p_1\cdots p_{n+1},\tau}_{s,t} & = \frac{(\tau-s)}{\Gamma(\beta_{p_{n+1}}) \Gamma(\sum_{l=1}^n\beta_{p_l}+1)} \int_0^{\frac{t-s}{\tau-s}} (\tau-s-v(\tau-s))^{\beta_{p_{n+1}}-1}(v(\tau-s))^{\sum_{l=1}^n\beta_{p_l}}\dd v \notag \\ & =\frac{(\tau-s)^{\sum_{l=1}^{n+1}\beta_{p_l}}}{\Gamma(\beta_{p_{n+1}}) \Gamma(\sum_{l=1}^n\beta_{p_l}+1)}\int_0^{\frac{t-s}{\tau-s}} (1-v)^{\beta_{p_{n+1}}-1}v^{\sum_{l=1}^n\beta_{p_l}}\dd v \notag \\ & =\frac{I_{\frac{t-s}{\tau-s}}(\sum_{l=1}^n\beta_{p_l}+1,\beta_{p_{n+1}})(\tau-s)^{\sum_{l=1}^{n+1}\beta_{p_l}}}{\Gamma(\sum_{l=1}^{n+1}\beta_{p_l}+1)},\label{eq:cK_frac_calc}
\end{align} where $I_{z}(a,b)$ denotes the regularized incomplete Beta function defined in \eqref{eq:regul_beta_def}. This proves our first claim~\eqref{eq:K_multivariate_frac}.

We now turn to the proof of \eqref{eq:Kdot_multivariate_frac}. Using the identities \[  \dot{\cK}_{s,t}^{p_1\cdots p_{n+1},\tau} = -\frac{\dd}{\dd s}\cK_{s,t}^{p_1\cdots p_{n+1},\tau}, \qquad -\frac{\dd }{\dd s}\cK_{s,t}^{p_1\cdots p_{n},t} = -\frac{\dd }{\dd s}k_{\sum_{l=1}^n\beta_{p_l}+1}(t-s) = k_{\sum_{l=1}^n\beta_{p_l}}(t-s),\] 
together with the Leibniz rule applied to differentiate \eqref{eq:app_frac} with respect to $s$ (the boundary term $k_{\beta_{p_{n+1}}}(\tau-s)\cK_{s,s}^{p_1\cdots p_n,s}$ vanishes since $\cK_{s,s}^{p_1\cdots p_n,s}=0$), we find \begin{eqnarray}\label{eq:dot_cK_frac}
\dot{\cK}_{s,t}^{p_1\cdots p_{n+1},\tau} &=& \int_s^t k_{\beta_{p_{n+1}}}(\tau-u) \left(-\frac{\dd}{\dd s}\cK_{s,u}^{p_1\cdots p_n,u}\right) \dd u \nonumber \\
&=& \int_s^t k_{\beta_{p_{n+1}}}(\tau-u)k_{\sum_{l=1}^n\beta_{p_l}}(u-s) \dd u.
\end{eqnarray} Repeating the same techniques as in \eqref{eq:cK_frac_calc}, we find \[
\dot{\cK}_{s,t}^{p_1\cdots p_{n+1},\tau} = \frac{I_{\frac{t-s}{\tau-s}}(\sum_{l=1}^n\beta_{p_l},\beta_{p_{n+1}})(\tau-s)^{\sum_{l=1}^{n+1}\beta_{p_l}-1}}{\Gamma(\sum_{l=1}^{n+1}\beta_{p_l})}.
\]
This proves \eqref{eq:Kdot_multivariate_frac}, and completes the proof of Proposition~\ref{ex:multivariate_fractional}.
\end{proof}

\subsubsection{Gamma kernel}\label{sec:apx_gamma}

We close this section by examining the case of a scalar Gamma kernel, that is $q=1$ and $K=k\,\mathrm{I}_d$, where \[
k(t,s)= \alpha e^{-\lambda (t-s)}k_\beta(t-s) \equiv k_{\lambda,\alpha,\beta}(t-s), \qquad \alpha,\beta,\gamma >0, 
\] and $k_\beta$ is the kernel given by \eqref{eq:1d_frac_def}.
\begin{prop}\label{ex:gamma_kernel}  For the scalar Gamma kernel $K$ and any $n\in \mathbb{N}$, we have (with $\dot{\kappa}^{n,\tau}_{s,t}$ as defined in \eqref{eq:kernel_cor})
\[
\dot{\kappa}_{s,t}^{n,\tau}=\alpha ^{n}e^{-\lambda(\tau-s)}(\tau-s)^{n\beta-1}\frac{I_{\frac{t-s}{\tau-s}}((n-1)\beta,\beta))}{\Gamma(n \beta)}, \qquad (s,t,\tau) \in \Delta^3,
\] where the incomplete Beta function $I$ is defined in \eqref{eq:regul_beta_def}. Moreover, the function $\kappa^{n}$ defined in \eqref{eq:kappa_pcwcst} becomes \[
\kappa_{s,t}^{n,\tau}= \frac{\lambda^{-(n-1)\beta}\alpha^{n-1}}{\Gamma((n-1)\beta)}\int_s^tk_{\lambda,\alpha,\beta}(\tau-u)\gamma((n-1)\beta,\lambda(u-s)) \dd u,
\] expressed in terms of the incomplete Gamma function $\gamma(a,b)= \int_0^b e^{-v}v^{a-1}\dd v$. Finally, we obtain a factorization \eqref{eq:combinatorial_vsig}-\eqref{eq:weights_cor} with weights $$\mu^{n_1, \dots, n_k}_{i_1, \dots, i_{k+1}} =  \alpha^{\sum_i n_i}e^{-\lambda \tau}\int_{t_{i_1-1}}^{t_{i_1}} \cdots \int_{t_{i_{k}-1}}^{t_{i_{k}}} e^{\lambda r_1}\prod_{l=1}^k \frac{(r_{l+1} - r_{l})^{n_l\beta -1}}{\Gamma(n_l\beta)} I_\frac{{t_{i_l}}-r_l}{r_{l+1} - r_l}((n_l-1)\beta, \beta) \dd{r_l},$$ where we use the convention $r_{k+1} = t_{i_{k+1}}$.
\end{prop}
\begin{rem}\label{rem:gamma_appendix}
    While here, as in Remark~\ref{rem:multivariate_frac_weights}, the weights $\kappa$ are only in semi-closed form, remarkably the $\dot{\kappa}$ remains closed and is of the same type as the weights we obtained in the fractional kernel case in \Cref{sec:apx_fractional}. Denoting the latter by $\eta$, we can in particular note that \[
\kappa_{s,t}^{(n+1),\tau} = \int_s^t \dot{\kappa}_{u,t}^{(n+1),\tau} \dd u = \alpha^{n+1}\int_s^te^{-\lambda(\tau-u)} \dot{\eta}_{u,t}^{(n+1),\tau} \dd u \approx \alpha^{n+1}e^{-\lambda(\tau-s)}\eta_{s,t}^{(n+1),\tau},
\]
where $\eta$ is given in closed-form (see \Cref{sec:apx_fractional}). Thanks to the smoothness of the exponential, on small time-scales $|t-s|$, the weights $\kappa$ are well-approximated by the exponentially averaged fractional weights. 
\end{rem}
\begin{proof}[Proof of Proposition~\ref{ex:gamma_kernel}]
    Starting with the diagonal case, it holds for \emph{any} convolutional kernel $k(t,s) = k(t-s)$ that 
    \[
\kappa_{s,t}^{(n),t} = \int_{\Delta_{s,t}^n}\prod_{l=1}^n k(r_{l+1}-r_l) \dd r_1 \cdots \dd r_n = \int_{\Delta_{0,t-s} ^n} \prod_{l=1}^n k(u_{l+1}-u_l) \dd u_1 \cdots \dd u_n = \kappa_{0,t-s}^{(n),t-s} ,
\] where $u_{n+1}:= t-s$. Moreover, an inductive argument shows that \begin{equation}\label{eq:kappa_convolution_power}
\kappa_{0,t}^{(n),t} = (k \ast \kappa^{(n-1)})(t) = \cdots =(k^{\ast (n)} \ast 1)(t).
\end{equation} Denoting by $L_f(t) = \int_0^\infty e^{-tx}f(x)\dd x$ the Laplace-transform, we have \[
L_{k^{\ast (n)}}(t) = L_k(t)^n = \frac{\alpha^n}{\Gamma(\beta)^n} \left (\int_0^\infty e^{-x(\lambda+t) }x^{\beta-1} \dd x \right )^n = \alpha^n(\lambda+t)^{-n \beta}.
\] On the other hand, from standard properties of the Laplace transform we know \[ L_{e^{-\lambda (\bullet)}\alpha^nk_{n\beta}}(t) = L_{\alpha^n k_{n\beta}}(\lambda+t) = \alpha^n (\lambda+t)^{-n\beta} .
\] By injectivity of the Laplace transform, we conclude that 
$$
k^{\ast (n)}(t)=\alpha^n e^{-\lambda t}\frac{t^{n\beta-1}}{\Gamma(n\beta)}=k_{\lambda,\alpha^n,n\beta}(t).
$$ 
Combining this with the identity $\kappa_{0,t}^{(n),t} = (k^{\ast (n)} \ast 1)(t)$ in \eqref{eq:kappa_convolution_power}, we find \[
\kappa_{0,t}^{(n),t} = \int_0^{t}k_{\lambda,\alpha^n,n\beta}(t-u) \dd u = \frac{\alpha^n}{\Gamma(n\beta)} \int_0^{t}e^{-\lambda u}u^{n\beta-1}\dd u =\frac{\lambda^{-n\beta }\alpha^n}{\Gamma(n\beta)} \gamma(n\beta,\lambda t),
\] where $\gamma$ is the incomplete gamma function $\gamma(a,b)= \int_0^b e^{-v}v^{a-1}\dd v$. In particular, on the diagonal $t = \tau$:
\begin{equation}\label{eq:kappa_diag_gamma}
\kappa_{s,t}^{(n),t}= \frac{\lambda^{-n\beta }\alpha^n}{\Gamma(n\beta)} \gamma(n\beta,\lambda (t-s)).
\end{equation}

We now move to the off-diagonal case $t < \tau$. Using \eqref{eq:kappa_diag_gamma} and arguing as in \eqref{eq:app_frac}, for any $(s,t,\tau)\in\Delta^3$ we find \[
\kappa_{s,t}^{(n+1),\tau}= \frac{\lambda^{-n\beta}\alpha^n}{\Gamma(n\beta)}\int_s^tk_{\lambda,\alpha,\beta}(\tau-u)\gamma(n\beta,\lambda(u-s)) \dd u.
\]
On the other-hand, we have \[ -\frac{\dd}{\dd s} \kappa_{0,t-s}^{(n),t-s} = \frac{\alpha^n}{\Gamma(n\beta)}e^{-\lambda(t-s)}(t-s)^{n\beta-1},
\] and thus in particular \begin{align*}
\dot{\kappa}_{s,t}^{(n+1),\tau}&= \frac{\alpha^{n+1}}{\Gamma(\beta)\Gamma(n\beta)}\int_s^t e^{-\lambda(\tau-u)}(\tau-u)^{\beta-1}(u-s)^{n\beta-1}e^{-\lambda(u-s)} \dd s  \\ & = \alpha ^{n+1}e^{-\lambda(\tau-s)} \int_s^tk_\beta(\tau-u)k_{n\beta}(u-s)\dd u 
\\ & =\alpha ^{n+1}e^{-\lambda(\tau-s)}(\tau-s)^{(n+1)\beta-1}\frac{I_{\frac{t-s}{\tau-s}}(n\beta,\beta))}{\Gamma((n+1) \beta)}.
\end{align*}
This establishes the formulas for $\dot{\kappa}^{n,\tau}_{s,t}$ and $\kappa^{n,\tau}_{s,t}$. The expression for the weights $\mu^{n_1,\dots,n_k}_{i_1,\dots,i_{k+1}}$ then follows by substituting the formula for $\dot{\kappa}$ into \eqref{eq:weights_cor}, completing the proof of Proposition~\ref{ex:gamma_kernel}.
\end{proof}

\section{Finite state space kernels}\label{sec:algo_multiplicative}
In this section, we specialize to kernels whose scalar components are of the general form
\begin{multline}\label{eq:prony_form_scalar}
k(t,s)
=
\sum_{r=1}^{Q}\sum_{\ell=1}^{m_r}
e^{-\lambda_r(t-s)}\frac{(t-s)^{\ell-1}}{(\ell-1)!}\,\alpha_{r,\ell}
\\
+\sum_{r=1}^{P}\sum_{\ell=1}^{n_r}
e^{-a_r(t-s)}\frac{(t-s)^{\ell-1}}{(\ell-1)!}
\Big(\beta_{r,\ell}\cos\big(\omega_r(t-s)\big)+\delta_{r,\ell}\sin\big(\omega_r(t-s)\big)\Big),
\end{multline}
for $(s,t)\in\Delta^2$, where $P,Q\in\NN_0$, the degrees satisfy $m_r,n_r\in\NN$, the real rates
$\lambda_1,\dots,\lambda_Q\in\RR$ are pairwise distinct, and the oscillatory parameters
$(a_1,\omega_1),\dots,(a_P,\omega_P)\in\RR\times(0,\infty)$ are pairwise distinct.
All coefficients $\alpha_{r,\ell},\beta_{r,\ell},\delta_{r,\ell}$ are real.

Such kernels are widely used in several contexts and, in particular, form a general approximation class for convolution (memory) kernels.
On the one hand, kernels of exponential--polynomial type are exactly the impulse responses of finite-dimensional continuous-time linear time-invariant state space models, making them canonical in systems and control---hence the terminology \emph{finite state space kernel}. On the other hand, sums of exponentials arise naturally as Laplace approximations of general memory kernels and underpin many fast numerical methods for Volterra and fractional-type models, including convolution quadrature and sum-of-exponentials compression. Moreover, such parametrizations are routinely fitted to data in system identification and signal processing via Prony-type and related spectral estimation methods. In stochastic settings, the same approximation principle yields Markovian state-augmentation (multi-factor) representations of non-Markovian Volterra dynamics, enabling efficient simulation and calibration.

Here we show that, for piecewise linear paths, an exact computation of the Volterra signature for finite state space kernels of the form \eqref{eq:prony_form_scalar} is possible via a dynamic algorithm that requires only a single iteration over the number of time steps.
This complements the finding in \cite[Section~2.4]{i_part} that the Volterra signature associated with such kernels solves a linear controlled differential equation in the tensor algebra.
To transition from this dynamical representation to an exact computation for a piecewise linear path, one may solve the system on each interval and then join the solutions across adjacent intervals.
Since such an increment-level decomposition has already been provided in \Cref{prop:chen_full_breakdown} in great generality, we use it as the starting point of our construction.

To this end, we first rewrite the kernel in matrix form in \Cref{sec:fssk_matrix_from}, which significantly simplifies the subsequent computations.
In \Cref{sec:fssk_weights} we prove that the weight coefficients $\mathcal{K}^{(w_1,\dots,w_k),\tau}$ (defined in \eqref{eq:def_dot_cK} and recalled below) factorize into a matrix product.
Moreover, in \Cref{sec:fssk_weights_numerics} we show that the individual matrix factors can be evaluated efficiently for arbitrary truncation levels using tools from numerical linear algebra.
The factorization of the weights, together with the semigroup property of the intermediate factors, then allows us to present in \Cref{sec:fssk_recursion} a linear recursion for Volterra signatures at the state-space level.
This recursion translates into an algorithm, which we describe in \Cref{sec:fssk_algorithmic}.
Similarly to \Cref{sec:quad_algo_aspects}, but now on the matrix level, we obtain efficient implementation schemes due to symmetry of the weight factors, as well as a Horner-type scheme for scalar kernels.

\subsection{Kernel in matrix form}\label{sec:fssk_matrix_from}
As a first step, we give a matrix representation for kernels of the form \eqref{eq:prony_form_scalar}, which will
drastically simplify the computations later on.

\begin{lem}\label{lem:prony_implies_real_jordan_scalar}
Let $k:\Delta^2\to\RR$ be a scalar kernel. Then $k$ is of the analytic form \eqref{eq:prony_form_scalar}
if and only if there exist $R\in\NN$, $\Lambda\in\RR^{R\times R}$ and $a,b\in\RR^{R}$ such that
\begin{equation}\label{g1}
k(t,s)= a^\top e^{-\Lambda (t-s)} b,\qquad (s,t)\in\Delta^2.
\end{equation}
Moreover, if $k$ is given by \eqref{eq:prony_form_scalar} with parameters
$\{m_r,\lambda_r,\alpha_{r,\ell}\}_{r=1}^Q$ and $\{n_r,a_r,\omega_r,\beta_{r,\ell},\delta_{r,\ell}\}_{r=1}^P$,
then one may choose
$
a={\bf 1}:=(1,\dots,1)^\top\in\RR^{R}
$
and $\Lambda$ so that $-\Lambda$ is in real Jordan normal form with diagonal blocks corresponding to
\eqref{eq:prony_form_scalar}, namely
$$
-\Lambda=\mathrm{diag}\Big(J_{m_1}(-\lambda_1),\dots,J_{m_Q}(-\lambda_Q),\,J_{n_1}(-a_1,\omega_1),\dots,J_{n_P}(-a_P,\omega_P)\Big),
$$
where
$$
J_m(\lambda):= \begin{pmatrix}
\lambda & 1 &  \\
&  \ddots & \ddots \\
&   & \lambda & 1\\
 &  &  &  \lambda
\end{pmatrix}\in\RR^{m\times m},
$$
and
$$
J_m(a, \omega) := \begin{pmatrix}
C(a,\omega) & I_2 & &  \\
&  \ddots & \ddots \\
&  & C(a,\omega) & I_2\\
  & & & C(a,\omega)
\end{pmatrix}\in\RR^{2m\times 2m},\qquad
C(a,\omega):=\begin{pmatrix} a & -\omega\\ \omega & a\end{pmatrix}.
$$
In particular $R=m_1+\cdots+m_Q+2(n_1+\cdots+n_P)$. In addition, the vector $b\in\RR^{R}$ in~\eqref{g1} is explicitly given by
$$
b=(d_1^\top,\dots,d_Q^\top,\,c_1^\top,\dots,c_P^\top)^\top,
\qquad d_r\in\RR^{m_r},\quad c_r\in\RR^{2n_r},
$$
where for $r=1,\dots,Q$ we set $\alpha_{r,m_r+1}:=0$, and define $d_r\in\RR^{m_r}$ by
$$
(d_r)_\ell := \alpha_{r,\ell}-\alpha_{r,\ell+1},\qquad \ell=1,\dots,m_r.
$$
In the same way for $r=1,\dots,P$ we set $\beta_{r,n_r+1}:=0$ and $\delta_{r,n_r+1}:=0$, and define $c_r\in\RR^{2n_r}$ by
$$
(c_r)_j := \frac12\Big(\beta_{r,\ell}-\beta_{r,\ell+1}+(-1)^j\big(\delta_{r,\ell}-\delta_{r,\ell+1}\big)\Big),
\qquad \ell:=\lceil j/2\rceil,\quad j=1,\dots,2n_r.
$$
\end{lem}

\begin{proof}
We first show that if $k(t,s)=a^\top e^{-\Lambda(t-s)}b$ for some $\Lambda\in\RR^{R\times R}$ and $a,b\in\RR^R$, then $k$ is of the analytic form \eqref{eq:prony_form_scalar}.
By \cite[Theorem~3.4.1.5]{horn2013matrix} there exists an invertible matrix $S\in\RR^{R\times R}$ such that
$\Lambda^\circ:=S^{-1}(-\Lambda)S$ is in real Jordan normal form. Setting $\tilde b:=S^{-1}b$ and $\tilde a^\top:=a^\top S$ and using similarity invariance of the exponential we obtain
$$
k(t,s)=\tilde a^\top e^{(t-s)\Lambda^\circ}\tilde b.
$$
Writing $\Lambda^\circ$ blockwise as in the statement and decomposing
$$
\tilde b=(d_1^\top,\dots,d_Q^\top,c_1^\top,\dots,c_P^\top)^\top
\qquad\text{and}\qquad
\tilde a^\top=(v_1^\top,\dots,v_Q^\top,u_1^\top,\dots,u_P^\top)
$$
accordingly, we have
$$
k(t,s)
=\sum_{r=1}^Q v_r^\top e^{(t-s)J_{m_r}(-\lambda_r)}\,d_r
+\sum_{r=1}^P u_r^\top e^{(t-s)J_{n_r}(-a_r,\omega_r)}\,c_r.
$$

For $r\in\{1,\dots,Q\}$ and $\delta>0$ it follows from the standard exponentiation of matrices in Jordan normal form \cite[Method 16]{moler03} that
$$
v_r^\top e^{\delta J_{m_r}(-\lambda_r)} d_r
=e^{-\lambda_r\delta}\sum_{\ell=1}^{m_r}\frac{\delta^{\ell-1}}{(\ell-1)!}\big(v_r^\top N^{\ell-1}d_r\big),
$$
where $N\in\RR^{m_r\times m_r}$ is the nilpotent shift matrix (ones on the superdiagonal and zeros elsewhere), so that $N^{m_r}=0$.

Similarly, for $r\in\{1,\dots,P\}$ and $\delta>0$, exponentiation of the real block Jordan matrix
$J_{n_r}(-a_r,\omega_r)$ (see \cite{teschl2012ordinary}) yields
\begin{equation}\label{eq:block_jordan_exp}
u_r^\top e^{\delta J_{n_r}(-a_r,\omega_r)}c_r
=
\sum_{\ell=1}^{n_r}\frac{\delta^{\ell-1}}{(\ell-1)!}\,
u_r^\top \mathbf N^{\ell-1}\,
\mathrm{diag}\big(e^{\delta C(-a_r,\omega_r)},\dots,e^{\delta C(-a_r,\omega_r)}\big)\,c_r,
\end{equation}
where $\mathbf N\in\RR^{2n_r\times 2n_r}$ is the nilpotent block shift matrix (with $I_2$ on the first superdiagonal block and zeros elsewhere), so that $\mathbf N^{n_r}=0$.
Write the $2$-block decompositions
$$
u_r^\top=(u_{r,1}^\top,\dots,u_{r,n_r}^\top),\qquad
c_r=(c_{r,1}^\top,\dots,c_{r,n_r}^\top)^\top,
\qquad u_{r,j},c_{r,j}\in\RR^2.
$$
Then the block shift structure of $\mathbf N$ implies, for each $\ell=1,\dots,n_r$,
$$
u_r^\top \mathbf N^{\ell-1}\,
\mathrm{diag}\big(e^{\delta C(-a_r,\omega_r)},\dots,e^{\delta C(-a_r,\omega_r)}\big)\,c_r
=
\sum_{j=\ell}^{n_r}u_{r,j-\ell+1}^\top e^{\delta C(-a_r,\omega_r)}c_{r,j}.
$$
Moreover,
$$
e^{\delta C(-a_r,\omega_r)}
=
e^{-a_r\delta}
\begin{pmatrix}
\cos(\omega_r\delta) & -\sin(\omega_r\delta)\\
\sin(\omega_r\delta) & \cos(\omega_r\delta)
\end{pmatrix}.
$$
Writing $u_{r,k}=(u_{r,k,1},u_{r,k,2})^\top$ and $c_{r,j}=(c_{r,j,1},c_{r,j,2})^\top$, we obtain for each $\ell=1,\dots,n_r$,
$$
\sum_{j=\ell}^{n_r}u_{r,j-\ell+1}^\top e^{\delta C(-a_r,\omega_r)}c_{r,j}
=
e^{-a_r\delta}\Big(\beta_{r,\ell}\cos(\omega_r\delta)+\delta_{r,\ell}\sin(\omega_r\delta)\Big),
$$
where
$$
\beta_{r,\ell}:=\sum_{j=\ell}^{n_r}\big(u_{r,j-\ell+1,1}c_{r,j,1}+u_{r,j-\ell+1,2}c_{r,j,2}\big),
\quad
\delta_{r,\ell}:=\sum_{j=\ell}^{n_r}\big(u_{r,j-\ell+1,2}c_{r,j,1}-u_{r,j-\ell+1,1}c_{r,j,2}\big).
$$
Hence, plugging things together we obtain
\begin{equation}\label{eq:block_jordan_inner}
u_r^\top \mathbf N^{\ell-1}\,
\mathrm{diag}\big(e^{\delta C(-a_r,\omega_r)},\dots,e^{\delta C(-a_r,\omega_r)}\big)\,c_r
=
e^{-a_r\delta}\Big(\beta_{r,\ell}\cos(\omega_r\delta)+\delta_{r,\ell}\sin(\omega_r\delta)\Big).
\end{equation}
Substituting \eqref{eq:block_jordan_inner} back into the expansion \eqref{eq:block_jordan_exp} yields
$$
u_r^\top e^{\delta J_{n_r}(-a_r,\omega_r)}c_r
=
\sum_{\ell=1}^{n_r}
e^{-a_r\delta}\frac{\delta^{\ell-1}}{(\ell-1)!}
\Big(\beta_{r,\ell}\cos(\omega_r\delta)+\delta_{r,\ell}\sin(\omega_r\delta)\Big),
$$
which is the required oscillatory contribution in \eqref{eq:prony_form_scalar}. This concludes the proof of the first direction.

Conversely, assume that $k$ is given by \eqref{eq:prony_form_scalar} and let $\Lambda$, $b$, $c$ and $d$ be as in the second part of the statement.
We first note that, by definition of $(d_r)_j=\alpha_{r,j}-\alpha_{r,j+1}$ and $\alpha_{r,m_r+1}=0$, we obtain the telescoping identity
$$
{\bf 1}^\top N^{\ell-1}d_r
=\sum_{j=\ell}^{m_r}(d_r)_j
=\alpha_{r,\ell},
$$
for all $\ell=1,\dots,m_r$.
Similarly, writing $c_r=(c_{r,1}^\top,\dots,c_{r,n_r}^\top)^\top$ with $c_{r,j}\in\RR^2$ and ${\bf 1}_2^\top:=(1,1)$ we have
$$
{\bf 1}^\top \mathbf N^{\ell-1}\,
\mathrm{diag}\big(e^{(t-s)C(-a_r,\omega_r)},\dots,e^{(t-s)C(-a_r,\omega_r)}\big)\,c_r
=
{\bf 1}_2^\top e^{(t-s)C(-a_r,\omega_r)}\sum_{j=\ell}^{n_r}c_{r,j}.
$$
By the definition of $c_r$ (with $\beta_{r,n_r+1}=\delta_{r,n_r+1}=0$), the tail sums telescope to
$$
\sum_{j=\ell}^{n_r}c_{r,j}
=
\frac12
\begin{pmatrix}
\beta_{r,\ell}-\delta_{r,\ell}\\
\beta_{r,\ell}+\delta_{r,\ell}
\end{pmatrix}.
$$
Since $-\Lambda$ is already in real Jordan normal form and block diagonal, applying the block expansions established above with $a={\bf 1}$ and inserting the above expressions for the coefficients, we readily obtain that
$
k(t,s)={\bf 1}^\top e^{-\Lambda(t-s)}b.
$
\end{proof}
Motivated by \Cref{lem:prony_implies_real_jordan_scalar}, we now define the class of (matrix-valued) finite state space kernels in matrix form.

\begin{defn}\label{def:finite_state_space_kernels}
Let $q,R\in\NN$ and let $\Lambda\in\RR^{R\times R}$. For parameters $\{A_r,b_r:1\le r\le q\}$ with
$A_r\in\RR^{m\times d}$ and $b_r\in\RR^{R}$ we define the kernel $K_{A,b}^{\Lambda}\in\Lkernel$ by
\begin{equation}\label{eq:algo_exponential_kernel}
K_{A,b}^{\Lambda}(t,s)=\sum_{r=1}^q \big({\bf 1}^\top e^{-\Lambda(t-s)}b_r\big)\,A_r,\qquad (s,t)\in\Delta^2,
\end{equation}
where ${\bf 1}^\top = (1,\dots,1) \in \RR^{1\times R}$.
\end{defn}
\begin{rem}
    As follows from \Cref{lem:prony_implies_real_jordan_scalar}, $\Lambda$ can be chosen, without loss of generality, in real Jordan normal form with distinct eigenvalues across diagonal blocks.
\end{rem}

\subsection{Computation of weight factors}\label{sec:fssk_weights}

We now treat, for the case of finite state space kernels as defined in \Cref{def:finite_state_space_kernels}, the computation of the weight factors appearing in the discretization of the Volterra signature for piecewise linear paths in \Cref{prop:chen_full_breakdown}.
In particular, we will prove that these weights factorize into a matrix product and give precise instructions for their efficient numerical evaluation.

To this end, given kernel data $\{A_r,b_r:1\le r\le q\}$ with
$A_r\in\RR^{m\times d}$, $b_r\in \RR^R$, and $\Lambda \in \RR^{R\times R}$, fix $K_{A,b}^{\Lambda}$ according to \Cref{def:finite_state_space_kernels} and also define the corresponding scalar kernels as in \Cref{not:comp_sec} (ii) by
\begin{align}
	\label{eq:apx_fssk}
	k_r(s,t) = \mathbf{1}^{\top} e^{-\Lambda (t-s)} b_r, \qquad (s,t)\in \Delta^2, \quad r =1, \dots, q.
\end{align}
We further fix a partition $0 = t_0 < t_1 < \dots < t_J = T$ and recall that the weight factors of interest from \Cref{prop:chen_full_breakdown} are then given by
\begin{equation}
	\label{eq:apx_mu_tens}
	\mathcal{K}^{(w_1, \dots, w_k)}_{i_1,\dots, i_{k+1}} := \int_{t_{i_1-1}}^{t_{i_1}} \cdots \int_{t_{i_{k}-1}}^{t_{i_{k}}} \prod_{l=1}^k \dot{\cK}^{w_l,r_{l+1}}_{r_l, t_{i_l}} \dd r_l\;\;\in \RR,\quad r_{k+1} = t_{i_{k+1}},
\end{equation}
for $0< i_1 < \dots <i_{k+1}\le J$, $w_1, \dots, w_k \in \mathcal{W}_q$, where
\begin{align}
	\label{eq:dot_cK_local}
	\dot{\cK}_{s,t}^{p_1\cdots p_n,\tau} =&\;  \int_{\Delta^{n-1}_{s,t}}  k_{p_1}(r_{1},r_0)\prod_{l=1}^{n-1} k_{p_l}(r_{l+1},r_l)\dd r_l, \qquad r_0 = s, \;r_{n} = \tau,\\
	{\cK}_{s,t}^{p_1\cdots p_n,\tau} =&\; \int_s^t \dot{\cK}_{u,t}^{p_1\cdots p_n,\tau}  \dd{u}.\nonumber
\end{align}
for any $p_1\cdots p_n \in \mathcal{W}_q$ and $(s,t,\tau) \in \Delta^3$.
For notational convenience, set
\begin{equation}\label{eq:E_C_def}
E(\delta) = e^{-\Lambda \delta} \quad (\delta > 0), \qquad  C_{r} = b_r{\bf 1}^\top = (b_r, \dots, b_r) \in \RR^{R\times R} \quad (r = 1, \dots, q).
\end{equation}

We divide the computation into two parts. First, we prove alternative equivalent analytic characterizations of $\dot{\cK}$ and then prove the factorization property for $\mathcal{K}^{(w_1, \dots, w_k)}_{i_1,\dots, i_{k+1}}$.
We then discuss a concrete numerical scheme for the evaluation of these weight factors in a second subsection.

\subsubsection{Weight factorization}\label{sec:fssk_weight_factorization}

The key observation of this subsection is that the kernel increment $\dot{\mathcal{K}}^{p_1\cdots p_n,\tau}_{s,t}$ admits a factorization in which the dependence on the look-ahead variable $\tau-t$ and the dependence on the increment $t-s$ decouple. Using the semigroup property $E(t+\delta)=E(t)E(\delta)$, this is made precise in \Cref{lem:apx_exp_K}. Indeed, this lemma re-expresses $\dot{\mathcal{K}}$ in terms of a family of matrix-valued functions $\Phi^{w}(\delta)$ satisfying two equivalent convolution recursions. Resorting to this structure, \Cref{lem:exp_weight_factorization} then establishes that each weight factor $\mathcal{K}^{(w_1,\dots,w_k)}_{i_1,\dots,i_{k+1}}$ can be written as an ordered product of matrices $\Phi^{w}$ and $\Psi^{w}$ evaluated on the mesh sub-intervals. This will be the starting point for the numerical scheme developed in \Cref{sec:fssk_weights_numerics}. Throughout, we use the shorthand $E(\delta) = e^{-\Lambda\delta}$ introduced in \eqref{eq:E_C_def}.
\begin{lem}
	\label{lem:apx_exp_K} Let $k_1, \dots, k_q \in L^{\infty,1}(\Delta^2; \RR)$ be as in \eqref{eq:apx_fssk}.
	Then for all $(s,t,\tau)\in \Delta^3$, $p_1\cdots p_n\in\mathcal{W}_{q}$ and $\sigma\in\mathcal{S}_{n-1}$ (where $\sigma(p_1\cdots p_n):=p_{\sigma(1)}\cdots p_{\sigma(n)}$ denotes the permutation action on words, as in \Cref{lem:cK_symmetry}) we have
	\begin{align*}
		\dot{\mathcal{K}}^{p_1\cdots p_n,\tau}_{s, t} =  {\bf 1}^\top  \Phi^{\sigma(p_1 \cdots p_{n-1})}(t-s) E(\tau - t) b_{p_n},
	\end{align*}
	where $C_{r} = b_r{\bf 1}^\top \in \RR^{R\times R}$ is as in \eqref{eq:E_C_def}, and $\Phi^w(\delta) \in \RR^{R\times R}$ $(\delta > 0, w\in \mathcal{W}_{q})$ is given by  $\Phi^{\varnothing}(\delta) = E(\delta)$ and any of the following equivalent recursions
	\begin{enumerate}[label=(\alph*), itemsep=0.8em] %
		\item $\displaystyle \Phi^{p_1 \cdots p_n}(t) = \int_{0}^t \Phi^{p_1 \dots p_{n-1}}(u) C_{p_{n}} E(t - u) \dd{u} \quad (t\ge0),$
		\item $\displaystyle \Phi^{p_1 \cdots p_n}(t) = \int_{0}^t E(t-u) C_{p_1} \Phi^{p_2 \cdots p_n}(u)  \dd{u} \quad (t\ge0),$
	\end{enumerate}
	Furthermore, for all $(s,t,\tau)\in \Delta^3$, $p_1\cdots p_n\in\mathcal{W}_{q}$ and $\sigma\in\mathcal{S}_{n-1}$ it holds
	\begin{align*}
		\dot{\mathcal{K}}^{p_1\cdots p_n,\tau}_{s, t} =  {\bf 1}^\top \Psi^{\sigma(p_1\cdots p_{n-1})}(t-s) E(\tau - t) b_{p_n},
	\end{align*}
	where $$\Psi^{w}(t) = \int_0^t \Phi^{w}(u)\dd{u}.$$
\end{lem}
\begin{proof}
	Plugging the expression $k_{p}(s,t) = \mathbf{1}^\top e^{-\Lambda(t-s)} b_{p}$ from \eqref{eq:apx_fssk} into the definition \eqref{eq:dot_cK_local} of $\dot{\mathcal{K}}^{p_1\cdots p_n,\tau}_{s,t}$, and applying the semigroup property $E(t+\delta) = E(t)E(\delta)$ (which follows from~\eqref{eq:E_C_def}), we directly observe that
	\begin{align*}
		\dot{\mathcal{K}}^{p_1\cdots p_n,\tau}_{s, t} =  {\bf 1}^\top\left(\int_{\Delta^{n-1}_{0, t-s}} E(r_{1} - r_0)\prod_{l=1}^{n-1} C_{p_{l}} E(r_{l+1} - r_l) \dd r_l\right)E(\tau - t)b_{p_n},
	\end{align*}
	with $r_0 = 0$ and $r_{n} = t-s$.
	One then easily verifies that the term in the parenthesis above satisfies the recursion (a).
	We next verify the equivalence between the recursions (a) and (b).
	Using a change of variable one easily verifies that this is the case for $n=1$.
	Now assume the claim holds for some $n\in\NN$ then using Fubini's theorem we obtain
	\begin{align*}
		\int_{0}^\delta  E(u)C_{p_1}&\Phi^{p_2\cdots p_n}(\delta - u)  \dd{u} \\
		&= \int_{0}^\delta E(\delta -u)C_{p_1} \left(\int_{0}^{u} \Phi^{p_2\cdots p_{n-1}}(r) C_{p_n} E(u-r) \dd{r}\right)\dd{u} \\
		&= \int_{0}^\delta  \left(\int_{r}^{\delta} E(\delta - u)C_{p_{1}}  \Phi^{p_2\cdots p_{n-1}}(u-r) \dd{u}\right) C_{p_n}E(r) \dd{r} \\
		&= \int_{0}^\delta  \left(\int_{0}^{\delta-r} E(u)C_{p_{1}}  \Phi^{p_2\cdots p_{n-1}}(\delta -r - u) \dd{u}\right) C_{p_n}E(r) \dd{r} \\
		&= \int_{0}^\delta  \Phi^{p_1\cdots p_{n-1}}(\delta -r) C_{p_n}E(r) \dd{r}.
	\end{align*}
	Well-posedness of these linear matrix differential equations is straightforward.
	The representation of  $\mathcal{K}$ follows directly by integrating the representation of $\dot{\mathcal{K}}$.
	The invariance under permutation of all but the last letter is a direct consequence of \Cref{lem:cK_symmetry}, since $(k_1,\dots,k_q)$ are of convolution form.
\end{proof}

As mentioned above, the main role of \Cref{lem:apx_exp_K} is to factor $\dot{\mathcal{K}}^{w,\tau}_{s, t}$ as a matrix depending only on $(t-s)$ times a matrix depending only on $(\tau-t)$. We exploit this to establish the following factorization of the full weight $\mathcal{K}^{(w_1,\dots,w_k)}_{i_1,\dots,i_{k+1}}$.
\begin{lem}
	\label{lem:exp_weight_factorization} Let $k_1, \dots, k_R \in L^{\infty,1}(\Delta^2; \RR)$ be as in \eqref{eq:apx_fssk}, and recall that $\mathcal{K}^{(w_1, \dots, w_k)}_{i_1,\dots, i_{k+1}}$ is given by~\eqref{eq:apx_mu_tens}.	Then for any $w_1, \dots, w_k \in \mathcal{W}_q$, $p_1, \dots, p_k \in \mathcal{A}_q$ and $\sigma_i\in\mathcal{S}_{|w_i|}$ $(i=1, \dots, k)$ it holds
	\begin{align*}
		\mathcal{K}^{(w_1p_1, \dots, w_kp_k)}_{i_1,\dots, i_{k+1}} = {\bf 1}^{\top}\Psi^{\sigma_1(w_1)}(\delta_{i_1}) \left(\prod_{l=2}^k E(\Delta_{i_{l-1},i_{l}})\Phi^{p_{l-1}\sigma_l(w_l)}(\delta_{i_l})\right)E(t_{i_{k+1}}-t_{i_k})b_{p_k},
	\end{align*}
	where $\Delta_{i,j} = t_{j-1} - t_i$ $(i< j)$, $\delta_i = t_{i} - t_{i-1}$ where $\Phi^{w}(\delta),\Psi^{w}(\delta) \subset \RR^{R\times R}$ $(\delta > 0, w\in \mathcal{W}_q)$ are defined as in \Cref{lem:apx_exp_K}.
\end{lem}
\begin{proof}
	Using \Cref{lem:apx_exp_K} we observe that
	\begin{align*}
		\prod_{l=1}^k\dot{\mathcal K}^{w_lp_l,r_{l+1}}_{r_l, t_{i_l}} &= \prod_{l = 1}^k {\bf 1}^{\top} \Phi^{w_l}(t_{i_l} - r_l) E(r_{l+1} - t_{i_{l}})b_{l} \\
		&=  \prod_{l = 1}^k {\bf 1}^{\top} \Phi^{w_l}(t_{i_l} - r_l) E(t_{i_{l+1}-1} - t_{i_{l}}) E(r_{l+1} -t_{i_{l+1}-1})b_l \\
		&=  {\bf 1}^{\top} \Phi^{w_1}(t_{i_1} - r_1) E(t_{i_{2}-1} - t_{i_{1}})\\
		&\qquad\cdot \prod_{l = 2}^k E(r_{l} - t_{i_{l}-1})C_{p_{l-1}} \Phi^{w_l}(t_{i_l} - r_l) E(t_{i_{l+1}-1} - t_{i_{l}}) \\
		&\qquad\cdot  E(r_{k+1}-t_{i_{k+1}-1})b_k.
	\end{align*}
	Hence the integrals in \eqref{eq:apx_mu_tens} factorize and using the recursion (b) in \Cref{lem:apx_exp_K} for $\Phi$ we directly obtain the claimed identity.
	The permutation invariance follows directly from the permutation invariance of $\dot\cK$ established in \Cref{lem:apx_exp_K}.
\end{proof}

When $q=1$, i.e., there is only one scalar kernel $k = k_1$ we can present the weight factors in a simplified form.

\begin{lem}
	\label{lem:apx_exp_kappa}
	Let $k(t,s) = {\bf 1}^{\top} e^{-\Lambda(t-s)} b$ for $(t,s)\in \Delta^2$, $\Lambda \in \RR^{R\times R}$, $b\in \RR^R$, and let $\dot{\kappa}^{n,\tau}_{s,t}$ be as in \eqref{eq:kernel_cor}. Then for all $(s,t,\tau)\in \Delta^3$ and $n\in \NN_{\ge 1}$ it holds
	\begin{align*}
		\dot{\kappa}^{n,\tau}_{s, t} =  {\bf 1}^\top  \phi_{n-1}(t-s) E(\tau - t) b, \qquad          {\kappa}^{n,\tau}_{s, t} =  {\bf 1}^\top \psi_n(t-s) E(\tau - t) b,
	\end{align*}
	where $\phi_n(\delta), \psi_n(\delta) \in \RR^{R\times R}$ $(\delta >0)$ are defined by $\phi_0(\delta) = E(\delta) \in \RR^{R\times R}$ and the system of matrix differential equation
	\begin{equation*}
		\left\{\begin{split}
			       \frac{\dd}{\dd t} \phi_n(t) =&\; \Lambda \phi_n (t) - C \phi_{n-1}(t) \\
			       \frac{\dd}{\dd t} \psi_n(t) =&\; \phi_{n-1} (t)
		\end{split}\right.,\quad  \phi_n(0) =  \psi_n(0) = \mathrm{0}_{R\times R}, \quad C:=b {\bf 1}^\top.
	\end{equation*}
	Furthermore, recalling from \eqref{eq:weights_cor} that
	\begin{equation*}
		\mu^{n_1, \dots, n_k}_{i_1, \dots, i_{k+1}} := \int_{t_{i_1-1}}^{t_{i_1}} \cdots \int_{t_{i_{k}-1}}^{t_{i_{k}}} \prod_{l=1}^k \dot{\kappa}^{n_l,r_{l+1}}_{r_l, t_{i_l}} \dd r_l, \quad r_{k+1} = t_{i_{k+1}},
	\end{equation*}
	it holds
	\begin{align*}
		\mu^{n_1, \dots, n_k}_{i_1, \dots, i_{k+1}} = {\bf 1}^{\top}\psi_{n_1}(\delta_{i_1}) \left(\prod_{l=2}^k E(\Delta_{i_{l-1},i_{l}})\phi_{n_l}(\delta_{i_l})\right)E(t_{i_{k+1}}-t_{i_k})b,
	\end{align*}
	where $\Delta_{i,j} = t_{j-1} - t_i$ $(i < j)$, $\delta_i = t_{i} - t_{i-1}$.
\end{lem}
\begin{proof}
	This follows directly from the previous lemmas by setting $\phi_n(\delta) := \Phi^{w}(\delta)$ and $\psi_{n-1}(\delta) := \Psi^{w}(\delta)$ with  $w=1\cdots 1 \in \mathcal{W}^n_1$ and differentiating the integral recursion (b).
\end{proof}

Finally we draw a connection between $\Phi$ and higher order Fr\'echet derivatives of the matrix exponential map, which opens up various possibilities for the numerical evaluation of $\Phi$ and $\Psi$ discussed in the next subsection.

\begin{lem}
	\label{lem:frechet_relation}
	Recall from \Cref{lem:apx_exp_K} that $\Phi^{w}(\delta) \in \RR^{R \times R}$ $(\delta > 0,\, w \in \mathcal{W}_q)$ is the matrix-valued function defined by the convolution recursions (a)--(b) therein. For any word $w = p_1 \cdots p_n \in \mathcal{W}^{n}_{R}$ we have  $$\sum_{\sigma \in \mathcal{S}_n}\Phi^{\sigma(w)}(t) = \frac{\partial^n}{\partial h_1 \cdots \partial h_n}
	\exp\bigg(t\bigg(-\Lambda + \sum_{i=1}^n h_i C_{p_i}\bigg)\bigg)\bigg|_{h_1,\dots,h_n=0}, \quad C_{p} = b_{p}{\bf 1}^\top,
	$$
	where the summation is over all permutations $\sigma\in\mathcal{S}_n$ of $\{1, \cdots, n\}$, and as in \Cref{lem:cK_symmetry} we set $\sigma(p_1 \cdots p_n) := p_{\sigma(1)} \cdots p_{\sigma(n)}$ for any $p_1\cdots p_n \in \mathcal{W}^n_q$.
	In particular, when $q=1$ with $b = b_1$ we have
	$$\phi_{n}(t) = \frac{1}{n!} \left.\frac{\partial^n}{(\partial h)^{n}}
    \exp\Big(t\Big(-\Lambda + h C\Big)\Big)\right|_{h=0}, \quad C = b{\bf 1}^\top.$$
\end{lem}
\begin{proof}
	Let $p_1\cdots p_n\in\mathcal{W}^n_q$ and set
	 $A(h) = A(h_1, \dots, h_n) := -\Lambda + \sum_{i=1}^n h_i C_{p_i}$. Then it follows from the classical integral representation of first derivative of the exponential map (see e.g. \cite[identity (2)]{najfeld1995derivatives}) that
	\begin{align*}
		\frac{\partial}{\partial h_i}e^{t A(h)} = \int_0^t e^{rA(h)}C_{p_i}e^{(t-r)A(h)}\dd{r}, \qquad h \in \RR^n, \quad t\ge0.
	\end{align*}
	Iterating this expression we arrive at
	\begin{align*}
		\frac{\partial^n}{\partial h_1 \dots \partial h_n}e^{t A(h)} = \sum_{\sigma \in \mathcal{S}_n}\int_{\Delta^n_{0,t}} \prod_{i=1}^n \left(e^{(r_i-r_{i-1})A(h)}C_{p_{\sigma(i)}}\right)e^{(t-r_{n})A(h)}\dd{r_1}\dots \dd{r_n}, \quad r_0 = 0.
	\end{align*}
	Setting $h=(0, \cdots, 0)$ and comparing with \Cref{lem:apx_exp_K} we observe that the integral term on the right-hand side is precisely given by $\Phi^{\sigma(p_1\cdots p_n)}(t)$, which concludes the proof of the first identity.
	The second identity for the $q=1$ case the follows directly by observing that in this case $\Phi^{\sigma(w)} = \Phi^{w} = \phi_{|w|}$ for all words $w\in \mathcal{W}_q$ and permutations $\sigma \in \mathcal{S}_{|w|}$.
\end{proof}

\subsubsection{Numerical evaluation}\label{sec:fssk_weights_numerics}

This subsection addresses the numerical evaluation of $\Phi^{w}(\delta)$ and $\Psi^{w}(\delta)$.
First, \Cref{prop:fssk_symmetry} shows, using the permutation invariance of the weight factors established in \Cref{lem:exp_weight_factorization}, that it suffices to evaluate one representative $w(\ell) = 1^{\ell_1} \cdots q^{\ell_q}$ per multi-index $\ell\in\NN^q$, rather than all words $w\in \mathcal{W}_q$ with $|w|\le N$.
This reduces the number of required evaluations significantly compared to the full set of Fréchet derivatives appearing in \Cref{lem:frechet_relation}.

The connection to Fréchet derivatives of the matrix exponential established in \Cref{lem:frechet_relation} makes it possible to draw on a range of methods from numerical linear algebra, including block embedding techniques \cite{higham2014higher}, complex step methods \cite{almohy2021complex}, and recursions in terms of divided differences \cite{rubensson2024unifying}.
Here we specialize to the algorithm of \cite{schweitzer2023integral}, based on approximating contour integral representations by quadrature.
This is particularly suitable in our setting since it exploits the available structure ($\Lambda$ in real Jordan normal form, rank-one directions $C_p = b_p\mathbf{1}^\top$) and admits a fully parallel evaluation across words.
The starting point is \Cref{lem:laplace_contour_representation}, which provides contour integral representations for $\Phi^{w}$ and $\Psi^{w}$; our proof, unlike \cite[Section~2.1]{schweitzer2023integral}, does not average over permutations and instead uses the complex inversion formula for the Laplace transform.
To this end, we use the following notation from complex analysis.

\begin{nota}
    For a matrix $A\in \mathbb{C}^{n\times n}$ we denote by $\mathrm{spec}(A) \subset \mathbb{C}$ the set of eigenvalues.
    For a contour $\mathfrak{C} \subset \mathbb{C}$ we write $\mathrm{spec}(A) \subset \mathrm{int}(\mathfrak{C})$ if $\mathfrak{C}$ encloses $\mathrm{spec}(A)$.
    Furthermore, for $\alpha > \max\{\Re(\lambda)\;:\;\lambda \in \mathrm{spec}(A)\}$ we call $\mathfrak{C} := \{ z\in \mathbb{C}\;:\; \Re(z) = \alpha\}$ a Bromwich line to the right of $\mathrm{spec}(A)$.
\end{nota}
\begin{rem}
    We assure the reader that these concepts from complex analysis are exclusively needed to state and prove the integral
    representations for $\Phi$ and $\Psi$ below, which are then used for a quadrature on the complex plane.
    For details on contour integrals in the context of Laplace transforms we refer to \cite[Section~3]{schiff1999laplace}.
\end{rem}

\begin{lem}\label{lem:laplace_contour_representation}
Let $w=p_1\cdots p_n\in\cW_q^{n}$, $n\in \NN_{\ge 1}$ and $\delta>0$.
For $\zeta\in\CC\setminus\mathrm{spec}(-\delta\Lambda)$ set
\[
R_\delta(\zeta):=(\zeta I+\delta\Lambda)^{-1},
\qquad
M^{w}_\delta(\zeta):=R_\delta(\zeta)\,C_{p_1}\,R_\delta(\zeta)\cdots C_{p_n}\,R_\delta(\zeta),
\]
where $C_{p} = b_{p}{\bf 1}^\top \in \RR^{R\times R}$ is as in \eqref{eq:E_C_def}, and $I\in \mathbb{R}^{R\times R}$ denotes the identity matrix.
Then, for any Bromwich line $\mathfrak{C}$ to the right of $\mathrm{spec}(-\delta\Lambda)$, or any simple, closed and positively oriented contour $\mathfrak{C}$ with
$\mathrm{spec}(-\delta\Lambda)\subset \mathrm{int}(\mathfrak{C})$, the matrices $\Psi^{w}(\delta),\Phi^{w}(\delta)\in\RR^{R\times R}$ of~\Cref{lem:apx_exp_K} respectively satisfy
\begin{eqnarray}\label{eq:laplace_contour_Phi_ordered}
\Phi^{w}(\delta)&=&\frac{\delta^{n}}{2\pi i}\int_{\mathfrak{C}} e^{\zeta}\,M^{w}_\delta(\zeta)\dd\zeta \\
\label{eq:laplace_contour_Psi_ordered}
\Psi^{w}(\delta)&=&\frac{\delta^{n+1}}{2\pi i}\int_{\mathfrak{C}} \varphi_1(\zeta)\,M^{w}_\delta(\zeta)\dd\zeta,
\end{eqnarray}
where $\varphi_1(\zeta) = \zeta^{-1}(e^{\zeta}-1)$ is defined on $\mathbb{C}$ via analytic continuation.
\end{lem}

\begin{proof}
Following \cite[Definition~1.10]{schiff1999laplace} we call a function $f:[0,\infty)\to\CC$ of exponential order $\alpha$
if there are $c,t_0\ge0$ such that $|f(t)|\le ce^{\alpha t}$ for all $t\ge t_0$.
We can then see that the components of the matrix-valued function $t\mapsto E(t):=e^{-t\Lambda}$ ($t\ge 0$) are of
exponential order $\alpha_0+\varepsilon$ for any $\varepsilon>0$, where
$\alpha_0 := -\min\{\Re(\lambda):\lambda\in\mathrm{spec}(\Lambda)\}$.
Here $\varepsilon$ accounts for the polynomial prefactor stemming from Jordan blocks; as noted after \Cref{def:finite_state_space_kernels}, $\Lambda$ can be assumed without loss of generality to be in real Jordan normal form, in which case this is immediate, and otherwise follows by passing to Jordan form.

It thus follows by \cite[Theorem~3.1]{schiff1999laplace} that the componentwise Laplace transform
\[
\widehat E(z):=\int_0^\infty e^{-zt}E(t)\,\dd t
\]
is well defined and analytic for all $z\in \CC$ with $\Re(z)>\alpha_0$.
Furthermore, for such $z$ the map $t\mapsto e^{-zt}E(t)$ is differentiable and satisfies
\[
\frac{d}{dt}\big(e^{-zt}E(t)\big)=-e^{-zt}(zI+\Lambda)E(t),
\]
Hence, integrating over $[0,\infty)$ yields
\[
(zI+\Lambda)\widehat E(z)=-\big[e^{-zt}E(t)\big]_{t=0}^{\infty}=I,
\]
and we readily see
\begin{equation}\label{eq:laplace_E_used}
\widehat E(z)=(zI+\Lambda)^{-1},\qquad \Re(z)>\alpha_0.
\end{equation}
Next, using recursion (a) in \Cref{lem:apx_exp_K} we verify inductively that for any word $w\in\mathcal{W}_q$ the components
of $t\mapsto \Phi^w(t)$ are again of exponential order $\alpha_0+\varepsilon$, and thus
\[
\widehat\Phi^w(z):=\int_0^\infty e^{-zt}\Phi^w(t)\,\dd t
\]
exists and is analytic in $\Re(z)>\alpha_0$.
The convolution theorem for Laplace transforms \cite[Theorem~2.39]{schiff1999laplace} applied componentwise then yields,
by the same recursion, that
\[
\widehat\Phi^{p_1\cdots p_n}(z)=\widehat\Phi^{p_1\cdots p_{n-1}}(z)\,C_{p_n}\,\widehat E(z).
\]
Since $\Phi^{\varnothing}(t)=E(t)$, iterating and using \eqref{eq:laplace_E_used} gives
\begin{equation}\label{eq:laplace_Phi_product_used}
\widehat{\Phi}^{p_1\cdots p_n}(z)
=
(zI+\Lambda)^{-1}C_{p_1}(zI+\Lambda)^{-1}\cdots C_{p_n}(zI+\Lambda)^{-1}.
\end{equation}
In particular, we also observe that the estimate $\|\widehat\Phi^{w}(z)\|\le c|z|^{-(n+1)}$ holds for all $z\in\CC$ with $\Re(z)>\alpha_0+1$
and a suitable constant $c>0$.
Now let $\delta>0$ and let $\mathfrak{C}$ be as in the statement of our lemma.
Due to our assumption $\mathrm{spec}(-\delta\Lambda)\subset \mathrm{int}(\mathfrak{C})$ and the fact that
 $\mathrm{spec}(-\Lambda) = \delta^{-1}\,\mathrm{spec}(-\delta\Lambda)$, the scaled contour $\mathfrak{C}^{\prime}:= \delta^{-1}\mathfrak{C}$ is again a Bromwich line to the right of (or a simple closed positively oriented contour enclosing) the set $\mathrm{spec}(-\Lambda)$.
Applying componentwise the Bromwich inversion formula for the Laplace transform \cite[Theorem~4.3]{schiff1999laplace} on the Bromwich line $\mathfrak{C}'$, and deforming to a closed contour using the Cauchy residue theorem \cite[p.~143]{schiff1999laplace}, we obtain
\begin{equation}\label{eq:rep_phi_proof}
\Phi^{w}(\delta)=\frac{1}{2\pi i}\int_{\mathfrak{C}'} e^{z\delta}\,\widehat\Phi^{w}(z)\,\dd z, \qquad \delta > 0.
\end{equation}
The identity \eqref{eq:laplace_contour_Phi_ordered} then follows by the change of variables $\zeta=\delta z$ using $(zI+\Lambda)^{-1}=\delta R_\delta(\zeta)$.

For $\Psi^{w}(\delta)=\int_0^\delta \Phi^{w}(u)\,\dd u$, we insert the representation of $\Phi^w(u)$ and use Fubini,
justified by the exponential order of $\Phi$ discussed above, to obtain
\[
\Psi^{w}(\delta)=\frac{1}{2\pi i}\int_{\mathfrak{C}'}
\left(\int_0^\delta e^{zu}\,\dd u\right)\widehat\Phi^w(z)\,\dd z
=\frac{1}{2\pi i}\int_{\mathfrak{C}'}\frac{e^{z\delta}-1}{z}\,\widehat\Phi^w(z)\,\dd z.
\]
Once again the change of variables $\zeta=\delta z$ leads to \eqref{eq:laplace_contour_Psi_ordered}, and the proof is complete.
\end{proof}

We now restrict to the case where $\mathrm{spec}(\Lambda)\subset[0,\infty)$,
which is the regime relevant for our applications.
Approximating the integrals in \eqref{eq:laplace_contour_Phi_ordered} and \eqref{eq:laplace_contour_Psi_ordered}
by the trapezoidal rule on a suitable parametrization of $\mathfrak{C}$ yields nodes
$\zeta_1,\dots,\zeta_m\in\CC$ and weights $\omega_1,\dots,\omega_m\in\CC$ and $\widetilde{\omega}_1,\dots,\widetilde{\omega}_m\in\CC$ such that
\begin{equation}\label{eq:apx_quad_Phi_Psi}
\Phi^{w}(\delta)\;\approx\;\delta^{n}\sum_{j=1}^{m}\omega_j\,M^{w}_\delta(\zeta_j),
\qquad
\Psi^{w}(\delta)\;\approx\;\delta^{n+1}\sum_{j=1}^{m}\widetilde\omega_j\,M^{w}_\delta(\zeta_j).
\end{equation}

For the exponential function and spectra close to the negative real axis, optimized parabolic or hyperbolic contours were studied in \cite{weideman2006optimizing,weideman2007parabolic,weideman2006talbot}.
For instance, an optimized parabolic contour gives explicit nodes and weights \cite[Equation~(3.1)]{weideman2006talbot}
\[
\zeta_j = m\cdot\bigl(0.1309-0.1194\,\theta_j^{2}+0.25\,i\,\theta_j\bigr),
\qquad
\omega_j = \exp(\zeta_j)\cdot\bigl(0.2388\,i\,\theta_j+0.25\bigr),
\]
where $\theta_j$ are equispaced points in $[-\pi,\pi]$.
For the evaluation of the matrix exponential this quadrature yields an approximation error of the order $\mathcal{O}(2.85^{-m})$.
Moreover, \cite[Proposition~12]{schweitzer2023integral} shows that the corresponding error asymptotics translate to the approximation
of higher-order Fr\'echet derivatives, which transfers to the approximation of $\Phi^{w}(\delta)$ and $\Psi^{w}(\delta)$ in \eqref{eq:apx_quad_Phi_Psi}.
For the approximation of $\Psi$ we use the same nodes. In the quadrature weights, however, we replace $\varphi_1(\zeta)=(e^\zeta-1)/\zeta$ by $e^\zeta/\zeta$, obtained by removing the vanishing $1/\zeta$ contribution; see, e.g., \cite{schmelzer2007evaluating}. Thus we set
\[
\widetilde\omega_j
=
\frac{\exp(\zeta_j)}{\zeta_j}\,
\bigl(0.2388\,i\,\theta_j+0.25\bigr).
\]

Finally, we discuss the simplifications in our setting for the evaluation of $M^w_\delta(\zeta)$.
First note that, since $\Lambda$ is real, we have the conjugate identity
\[
\omega\,M^w_\delta(\zeta)+\overline{\omega}\,M^w_\delta(\overline{\zeta})
=
2\,\Re\big(\omega\,M^w_\delta(\zeta)\big),
\]
Thus it suffices to evaluate the right-hand side for one representative per conjugate pair of quadrature nodes.

Next we exploit the rank-one and Jordan structure (cf.\ the rank-one simplification in \cite[Section~4.3]{schweitzer2023integral}).
To this end, for nodes $\zeta$ on the contour we define the vectors
\[
u_p(\zeta):=R_\delta(\zeta)b_p\in\CC^{R},\qquad r(\zeta)^\top:={\bf 1}^\top R_\delta(\zeta)\in\CC^{1\times R},\qquad \beta_p(\zeta):=r(\zeta)^\top b_p\in\CC.
\]
Then we can express $M^{w}_\delta(\zeta)$ for $w = p_1\cdots p_n$ with $n>1$ without dense matrix--matrix multiplications as follows:
\begin{equation}\label{eq:apx_rankone_eval_Mw}
M^{w}_\delta(\zeta)
=
u_{p_1}(\zeta)r(\zeta)^\top \prod_{k=2}^{n}\beta_{p_k}(\zeta).
\end{equation}
Thus each quadrature node requires only shifted solves with $\zeta I+\delta\Lambda$ to form $u_p(\zeta)$ and $r(\zeta)^\top$.
Since $\Lambda$ is assumed to be in (real) Jordan normal form, the costs of these shifted solves are only of order $R$.

As discussed in \Cref{prop:fssk_symmetry} we only need to compute $\Psi^{w(\ell)}$ and $\Phi^{pw(\ell)}$ for multiindices $\ell\in \NN^q$ where $w(\ell) = 1^{\ell_1} \cdots q^{\ell_q}$.
To this end we see that
\begin{equation}\label{eq:apx_rankone_eval_Mw_multi}
M^{pw(\ell)}_\delta(\zeta)
=
u_{p}(\zeta)r(\zeta)^\top \beta_{1}(\zeta)^{\ell_1} \cdots \beta_{q}(\zeta)^{\ell_q}.
\end{equation}
Moreover, as readily observed in \Cref{thm:algo_multiplicative} we only ever need to evaluate the row vector ${\bf 1}^{\top} \Psi^{w}(\delta)$ and not the full matrix $\Psi^{w}(\delta)$, so we obtain more simply
\[
{\bf 1}^{\top} M^{w(\ell)}_\delta(\zeta)
=
r(\zeta)^\top \beta_{1}(\zeta)^{\ell_1} \cdots \beta_{q}(\zeta)^{\ell_q}.
\]
In particular, the computation can be efficiently parallelized over multi-indices and quadrature nodes.

We summarize the above discussion of numerical evolution 
in \Cref{alg:fssk_weight_quad_eval}.

\begin{algorithm}[h]
\caption{Subroutine \textsc{EvalPhiPsi}: evaluation of normalized coefficients $\delta^{-|\ell|-1}\Phi^{pw(\ell)}(\delta)$ and $\delta^{-|\ell|-1}{\bf 1}^\top\Psi^{w(\ell)}(\delta)$}\label{alg:fssk_weight_quad_eval}
\begin{algorithmic}[1]
\STATE \textbf{Input:} $\delta>0$, $\Lambda\in\RR^{R\times R}$, $b_1,\dots,b_q\in\RR^R$, truncation level $N\in\NN$
\STATE \textbf{Parameter:} $m\in\NN$ according to machine precision; corresponds to $2m$ total nodes
\STATE \textbf{Output:} $\widehat\Phi_{p,\ell}\approx \delta^{-|\ell|-1}\Phi^{pw(\ell)}(\delta)$ for $p=1,\dots,q$, and $\widehat\psi_\ell\approx \delta^{-|\ell|-1}{\bf 1}^\top\Psi^{w(\ell)}(\delta)$, for all $\ell\in\NN^q$ with $|\ell|=\ell_1+\cdots+\ell_q \le N$
\FOR{$j=1,\dots,m$}
    \STATE $\theta_j \gets (2j-1)\pi(2m)^{-1}$ \label{line:evalphipsi_quad_start}
    \STATE $\zeta_j \gets 2m\big(0.1309-0.1194\,\theta_j^2+0.25\,i\,\theta_j\big)$
    \STATE $\omega_j \gets -\exp(\zeta_j)\big(0.2388\,i\,\theta_j+0.25\big)$
    \STATE $\widetilde\omega_j \gets -\exp(\zeta_j)\big(0.2388\,i\,\theta_j+0.25\big)/\zeta_j$
    \STATE Solve $(\zeta_j I+\delta\Lambda)^\top r_j={\bf 1}$ for $r_j\in\CC^R$
    \FOR{$p=1,\dots,q$}
        \STATE Solve $(\zeta_j I+\delta\Lambda)u_{p,j}=b_p$ for $u_{p,j}\in\CC^R$
        \STATE $\beta_{p,j}\gets  r_j^\top b_p\in\CC$ \label{line:evalphipsi_quad_end}
    \ENDFOR
\ENDFOR
\FOR{each $\ell\in\NN^q$ with $|\ell| \le N$}
    \STATE $\widehat\psi_\ell \gets 0 \in \RR^{1\times R}$ \label{line:evalphipsi_accumul_start}
    \FOR{$p=1,\dots,q$}
        \STATE $\widehat\Phi_{p,\ell} \gets 0 \in \RR^{R\times R}$
    \ENDFOR
    \FOR{$j=1,\dots,m$}
        \STATE $\gamma_{\ell,j}\gets \prod_{p=1}^q \beta_{p,j}^{\ell_p}$
        \STATE $\widehat\psi_\ell \gets \widehat\psi_\ell + 2\Re\!\left(\widetilde\omega_j\,\gamma_{\ell,j}\,r_j^\top\right)$
        \FOR{$p=1,\dots,q$}
            \STATE $\widehat\Phi_{p,\ell} \gets \widehat\Phi_{p,\ell} + 2\Re\!\left(\omega_j\,\gamma_{\ell,j}\,u_{p,j}r_j^\top\right)$ \label{line:evalphipsi_accumul_end}
        \ENDFOR
    \ENDFOR
\ENDFOR
\end{algorithmic}
\vspace{1mm}
\par\noindent\raggedright
\emph{Parallelization.} The loops over $j$ and $\ell$ are  independent respectively and can be parallelized. For fixed $j$ or fixed $\ell$, the loop over $p$ is also parallelizable.\\
\emph{Practical note.} For the evaluation of $\pi_{\le N}\VSig{x}{K^{\Lambda}_{A,b}}$ it is only necessary to evaluate ${\bf 1}^\top\Psi^{w(\ell)}$ for $|\ell|\le N-1$ and $\Phi^{pw(\ell)}$ for $|\ell|\le N-2$ (cf.~\Cref{thm:algo_multiplicative}).
\end{algorithm}

\subsection{The recursion}\label{sec:fssk_recursion}
We now demonstrate how the factorization of the weight coefficients in \Cref{lem:exp_weight_factorization} yields a recursion for the Volterra signature of the finite state space kernel $K^{\Lambda}_{A,b}$ along piecewise linear paths, involving only a linear number of recursion steps in the number of partition intervals.
To this end, we introduce the following notation, which makes both the presentation and the \emph{implementation} more convenient:
\begin{nota}\label{not:matrix_ring}
	We consider $\mathfrak{R}_N:= (T^N(\RR^m), \otimes_N)$ as a ring and, over it, the set of matrices $(\mathfrak{R}_N)^{k\times n}$ for $k, n \ge 1$.
	In particular, for two matrices $\mathcal{M} = (\mathcal{M}^{ij})\in (\mathfrak{R}_N)^{k\times n}$ and  $\mathcal{N} = (\mathcal{N}^{ij})\in (\mathfrak{R}_N)^{n\times q}$, the matrix  product is denoted by $\mathcal{M}.\mathcal{N} \in (\mathfrak{R}_N)^{k\times q}$ and is defined componentwise by
	$$\big(\mathcal{M}.\mathcal{N}\big)^{ij} = \sum_{l=1}^n\mathcal{M}^{il}\otimes_N\mathcal{N}^{lj}.$$
	Any linear operation defined on $T^N(\RR^m)$ extends to matrices $(\mathfrak{R}_N)^{k\times n}$ componentwise. This holds, in particular, for multiplication with a tensor $\by\in T^N(\RR^m)$, with
	$\by\otimes \mathcal{M} = (\by \otimes \mathcal{M}^{ij})$ and $\mathcal{M}\otimes\by = (\mathcal{M}^{ij}\otimes \by)$,
	and for the projection $\pi_l$, with $\pi_l \mathcal{M} = (\pi_{l}\mathcal{M}^{ij})$.
\end{nota}

\begin{thm}
	\label{thm:algo_multiplicative}
	Let $K_{A,b}^{\Lambda} \in \Lkernel$ be a kernel satisfying \Cref{def:finite_state_space_kernels}, i.e. of the form \eqref{eq:algo_exponential_kernel} for given data $(\Lambda, A, b)$, and ${x:[0,T]\to \RR^d}$ be continuous piecewise linear path on some grid $t_0 < t_1 < \dots < t_J \le T$.
	We set $E(\delta) = e^{-\Lambda \delta} \in \RR^{R\times R}$ $(\delta >0)$, $\delta_j = t_j - t_{j-1}$ and
	$$y_{j}^p= A_p(x_{t_j}-x_{t_{j-1}}), \qquad j=1, \dots, J, \quad p =1, \dots, q.$$
	Fixed a truncation level $N\in\NN$.
	Then for any $j = 0, \dots, J$ and $\tau \in [t_j, T]$ it holds \begin{equation}\label{eq:FSSK_scheme}
	    \pi_{\le N}\VSig{x}{K_{A,b}^{\Lambda}}^{\tau}_{t_0,t_j} = 1 + \sum_{p=1}^q\bZ^p_j.E(\tau - t_j)b_{p}, 
	\end{equation}
	where  each $\bZ^{j}$ is an element of $(\mathfrak{R}_{N})^{1\times R}$, where the products $.$ in \eqref{eq:FSSK_scheme} are understood in the sense of \Cref{not:matrix_ring}, and where we define the following iteration for all $p\in \mathcal{A}_q$:
	\begin{equation}\label{eq:iteration_Z}
		\left\{\begin{split}
		{\bf Z}^p_0 =&~\bigg. (0,\dots,0) \in \RR^{1\times R} \subset (\mathfrak{R}_N)^{1\times R},\\
		\bigg.{\bf Z}^p_{j} =&~ {\bf Z}^p_{j-1}.E(\delta_j) + \left({\bf 1}^{\top}.\mathcal{F}_{\delta_j}\!\Big(y_j^1, \dots, y_j^q\Big) +\sum_{l=1}^q {\bf Z}^l_{j-1}.{\mathcal{G}}^l_{\delta_j}\!\Big(y_j^1, \dots, y_j^q\Big)\right)\otimes y_j^p.
		\end{split}\right. 
	\end{equation}
	Note that in \eqref{eq:iteration_Z} we have used the following additional notation: for $\delta > 0$, $p_0\in \mathcal{A}_q$ and $y_1, \dots, y_q\in \RR^m$ we define
	\begin{equation}\label{eq:def_GF}
    \begin{split}
		\mathcal{F}_\delta(y_1,\dots, y_q) =&  \sum_{n=1}^{N-1}\frac{1}{\delta^{n+1}}\sum_{p_1\cdots p_n \in \mathcal{W}^n_q} \Psi^{p_1\cdots p_n}(\delta)  y_{p_1} \otimes \cdots \otimes y_{p_n} \quad \in(\mathfrak{R}_{N})^{R\times R},\\
		\mathcal{G}^{p_0}_\delta(y_1,\dots, y_q) =&  \sum_{n=1}^{N-2}\frac{1}{\delta^{n+1}}\sum_{p_1\cdots p_n \in \mathcal{W}^n_q} \Phi^{p_0p_1\cdots p_n}(\delta)  y_{p_1} \otimes \cdots \otimes y_{p_n}\quad \in(\mathfrak{R}_{N})^{R\times R},
	\end{split}
	\end{equation}
	where $\Psi^w(\delta), \Phi^w(\delta) \in \mathbb{R}^{R\times R}$ $(w\in\cW_q)$
	are defined in \Cref{lem:apx_exp_K}
	in terms of $(b,\Lambda)$.
\end{thm}
\begin{proof}
    We set $v^p_j := (\delta_{j})^{-1}y^p_j$ and verify inductively that $(\bZ^1, \dots, \bZ^q)$ takes the following expanded form:
    \begin{multline}\label{eq:Z_induction}
		\pi_{n+1}\bZ^p_j 
		=
		\;\sum_{0<i_1<\dots< i_k < j+1}
\;\sum_{\substack{p_1,\dots,p_{k-1}\in\mathcal{A}_q \\ w_1,\dots, w_k \in \mathcal{W}_q \\ |w_1| + \cdots + |w_k| = n-k+1}} 
\mathcal{M}^{(w_1p_1, \dots, w_{k-1}p_{k-1}, w_k)}_{i_1, \dots, i_k,j+1} \\
\times (v_{i_1})^{w_1p_1} \otimes \cdots \otimes (v_{i_{k-1}})^{w_{k-1}p_{k-1}} \otimes (v_{i_k})^{w_kp} , 
	\end{multline}
    for all $p \in \mathcal{A}_q$, $n = 0, \dots, N-1$ and $j = 0, \dots, J$,
    where $(v_j)^{p_1 \cdots p_n} := v_j^{p_1}\otimes \cdots \otimes v_j^{p_n}$ and
	\begin{equation}\label{eq:def_cM_fsskproof}
		\mathcal{M}^{(w_1p_1, \dots, w_{k-1}p_{k-1}, w_k)}_{i_1,\dots, i_{k+1}} = {\bf 1}^{\top}\Psi^{w_1}(\delta_{i_1}) \left(\prod_{l=2}^k E(\Delta_{i_{l-1},i_{l}})\Phi^{p_{l-1}w_l}(\delta_{i_l})\right)E(\Delta_{i_{k},i_{k+1}}),
	\end{equation}
    for all $w_1, \dots, w_k \in \cW_q$ and $p_1, \dots, p_{k-1}\in \cA_q$ 
    with $\Delta_{i,j} = t_{j-1} - t_{i}$ $(i < j$).
    Indeed, the claim clearly holds for $j=0$ as the first sum is empty and $\pi_n\bZ^p_0 = 0$ for all $n = 0, \dots, N$ by definition.
    Now for the induction step we consider the three  terms in the iteration step $(j-1) \to j$ in
    \eqref{eq:iteration_Z}.
    First, note that using the semigroup property of $E(\Delta_{i,j})E(\delta_{j}) = E(\Delta_{i,j+1})$ we obtain
    $$\mathcal{M}^{(w_1p_1, \dots, w_{k-1}p_{k-1}, w_k)}_{i_1,\dots, i_k,j}.E(\delta_j) = \mathcal{M}^{(w_1p_1, \dots, w_{k-1}p_{k-1}, w_k)}_{i_1,\dots, i_{k},j+1}.$$
    Thus the term ${\bf Z}^p_{j-1}.E(\delta_{j})$ constitutes all terms in \eqref{eq:Z_induction}  corresponding to indices $(i_1,\dots, i_k)$ where $i_k < j$.
    Next, we note that 
    $$\pi_{n+1}\left( {\bf 1}^{\top}.\mathcal{F}_{\delta_{j}}\!\Big(y_{j}^1, \dots, y_{j}^q\Big)\otimes y_j^p\right) =  \sum_{w \in \cW_q^n}\mathcal{M}^{(w)}_{j,j+1} (v_{j})^{wp},$$
    which hence constitute the terms corresponding to the indices $(i_1, \dots, i_k)$ with $k=1$ and $i_1 = j$.
    Finally in the case where $i_k = j$ and $k>1$ we have  $\Delta_{j, j+1} = 0$ and hence
	\begin{align*}
		\mathcal{M}^{(w_1p_1, \dots, w_{k-1}p_{k-1}, w_k)}_{i_1,\dots, i_{k},j+1} = \mathcal{M}^{(w_1p_1, \dots, w_{k-2}p_{k-2}, w_{k-1})}_{i_1,\dots, i_{k-1},j}\Phi^{p_{k-1}w_k}(\delta_{j}).
	\end{align*}
    Similarly, we factorize 
    $$(v_{i_1})^{w_1p_1} \otimes \cdots \otimes (v_{i_{k-1}})^{w_{k-1}p_{k-1}}\otimes (v_{i_k})^{w_{k}} = \left( (v_{i_1})^{w_1p_1} \otimes \cdots \otimes (v_{i_{k-1}})^{w_{k-1}p_{k-1}} \right)\otimes  (v_{j})^{w_k}.$$
    This, readily allows us to verify that the term
    $$\sum_{l=1}^q {\bf Z}^l_{j-1}.{\mathcal{G}}^l_{\delta_j}\!\Big(y_j^1, \dots, y_j^q\Big)\otimes y_j^p,$$
    constitutes all terms in \eqref{eq:Z_induction} corresponding to indices $(i_1,\dots, i_k)$ where $i_k = j$ and $k>1$.
    Having proved the identity \eqref{eq:Z_induction}, we next recall the expanded form of $\VSig{x}{K^{\Lambda}_{A,b}}$ from \Cref{prop:chen_full_breakdown} in terms of the coefficients $\cK$.
    The explicit form of $\cK$ for kernels $K^{\Lambda}_{A,b}$ satisfying \Cref{def:finite_state_space_kernels} in terms of $\Phi$ and $\Psi$ is provided in \Cref{lem:exp_weight_factorization}.
    Comparing the expression for $\cK$ from \Cref{lem:exp_weight_factorization} with the definition \eqref{eq:def_cM_fsskproof} of $\mathcal{M}$, and using the semigroup property to split the last exponential as $E(t_{l} - t_{i_k}) = E(\Delta_{i_k,j+1})\,E(t_{l} - t_j)$, we obtain for $l \ge j$ ($\tau = t_{l}$) that
    $$\cK^{w_1p_1, \dots, w_kp_k}_{i_1, \dots, i_{k},l} ~=~ \mathcal{M}^{w_1p_1, \dots, w_{k-1}p_{k-1},w_k}_{i_1, \dots, i_k,j+1}E(t_{i_{l}}-t_j)b_{p_k},$$
    from which we readily conclude relation \eqref{eq:FSSK_scheme} between $\bZ$ and $\pi_{\le N} \VSig{x}{K^{\Lambda}_{A,b}}$.
\end{proof}

As several things simplify in the $q=1$ case we add the following corollary, where we will without lost of generality, i.e., up to a linear transform of the path, already assume that $A_1 = \mathrm{I}_{m}$ is the identity matrix.
\begin{cor}\label{cor:fsskq1}
Let $k_{b}^{\Lambda} \in L^{\infty,1}(\Delta^2; \RR)$ be a kernel of the form 
$$k_{b}^{\Lambda}(s,t) = {\bf 1}^\top e^{-\Lambda(t-s)}b, \qquad \Lambda \in \RR^{R\times R}, \quad b\in \RR^{R},$$
and let ${x:[0,T]\to \RR^d}$ be continuous piecewise linear path on some grid $t_0 < t_1 < \dots < t_J \le T$.
	We set $E(\delta) = e^{-\Lambda \delta} \in \RR^{R\times R}$ $(\delta >0)$, $\delta_j = t_j - t_{j-1}$ and fix a truncation level $N\in\NN$.
	Then for any $j = 0, \dots, J$ and $\tau \in [t_j, T]$ it holds $$\pi_{\le N}\VSig{x}{k_{b}^{\Lambda}}^{\tau}_{0,t_j} = 1 + \bZ_j.E(\tau - t_j)b, \qquad .$$
	where $\bZ$ is defined by ${\bf Z}_0 = (0,\dots,0) \in \RR^{1\times R} \subset (\mathfrak{R}_N)^{1\times R}$ and the following iteration:\begin{equation}\label{eq:iteration_Z_q1}
		{\bf Z}_{j} =  {\bf 1}^\top.\sum_{n=1}^{N}\frac{1}{(\delta_j)^{n}} \psi_{n}(\delta_j)  (x_{j}-x_{j-1})^{\otimes n} + {\bf Z}_{j-1}.\sum_{n=0}^{N-1}\frac{1}{(\delta_j)^{n}} \phi_{n}(\delta_j)  (x_j - x_{j-1})^{\otimes n},
	\end{equation}
	for $j=1, \dots, J$ where $\psi_n(\delta), \phi_n(\delta) \in \mathbb{R}^{R\times R}$ $(n \in \NN)$
	are defined in \Cref{lem:apx_exp_kappa}
	in terms of $(b,\Lambda)$.    
\end{cor}

\subsection{Algorithmic aspects}\label{sec:fssk_algorithmic}

We now turn to the numerical implementation of the recursion derived in \Cref{thm:algo_multiplicative}. The state-space recursion \eqref{eq:iteration_Z} directly translates into an iterative scheme, and the Volterra signature for finite state space kernels is then obtained by a linear readout from this state. We summarize these two steps separately in \Cref{alg:fssk_state_recursion} and \Cref{alg:fssk_vsig_readout}, where the evaluation of $\cF$ and $\cG$ is left as a subroutine in the state update.

To obtain a full realization of \Cref{alg:fssk_state_recursion}, it therefore remains to specify the efficient evaluation of the operators $\cF$ and $\cG$. For the scalar-kernel case $q=1$, i.e., for the recursion in \Cref{cor:fsskq1}, the situation simplifies. In this case, we additionally recover an efficient Horner-type scheme (cf.~\Cref{rem:mult_horner}). We therefore subdivide this discussion into the cases $q>1$ and $q=1$.

\subsubsection{The case $q > 1$}

First, we note that the evaluation of $\cF$ and $\cG$ as defined in \eqref{eq:def_GF} requires the coefficient matrices $\Phi$ and $\Psi$ from \Cref{lem:apx_exp_K}.
The numerical evaluation of these matrices was already treated in \Cref{sec:fssk_weights}, where we proposed the Laplace-quadrature-based algorithm \Cref{alg:fssk_weight_quad_eval}.
This shuffle-recursive structure is directly analogous to the multivariate Horner-type evaluation of ${\bv}\otimes_N\mathcal{E}$ under symmetric weights in \Cref{sec:quad_conv} (cf.\ \Cref{alg:quad_evalVtE}), with the only difference that the coefficients here are $R\times R$ matrices and the propagated quantities take values in $(\mathfrak R_N)^{1\times R}$ and $(\mathfrak R_N)^{R\times R}$.

Once these coefficient matrices are available, additional simplifications arise from the convolution structure of kernels of the form \eqref{eq:algo_exponential_kernel}.
Indeed, the permutation symmetry of the weights (cf.\ \Cref{lem:cK_symmetry,lem:exp_weight_factorization}) yields a shuffle-polynomial representation of $\cF$ and $\cG$, which leads to an efficient recursive evaluation.
This is formalized in \Cref{prop:fssk_symmetry}, and the corresponding algorithm is given in \Cref{alg:fsskFG}.

\begin{algorithm}[h]
\caption{State recursion for finite state space kernels}\label{alg:fssk_state_recursion}
\begin{algorithmic}[1]
\STATE \textbf{Input:} partition $0=t_0<t_1<\cdots<t_J\le T$, piecewise linear path values $\{x_{t_j}\}_{j=0}^J$, kernel data $(\Lambda,\{A_p,b_p\}_{p=1}^q)$, truncation level $N\in\NN$
\STATE \textbf{Output:} states $\{{\bf Z}_j^p\in(\mathfrak R_N)^{1\times R}: j=0,\dots,J,\ p=1,\dots,q\}$

\STATE \textbf{Initialize:} for $p=1,\dots,q$ set ${\bf Z}_0^p \gets (0,\dots,0)\in\RR^{1\times R}\subset(\mathfrak R_N)^{1\times R}$
\FOR{$j=1,\dots,J$}
    \STATE $\delta_j \gets t_j-t_{j-1}$ \label{line:fssk_delta}
    \STATE $E_j \gets e^{-\Lambda\delta_j}\in\RR^{R\times R}$ \hfill \emph{// Jordan form, cf.\ \Cref{lem:prony_implies_real_jordan_scalar}} \label{line:fssk_exp}
    \STATE $(\widehat\psi_j,\widehat\Phi_j) \gets \textsc{EvalPhiPsi}(\delta_j,\Lambda,\{b_p\}_{p=1}^q,N-1)$ \hfill \emph{// cf.\ \Cref{alg:fssk_weight_quad_eval}} \label{line:fssk_coeffs}
    \FOR{$p=1,\dots,q$}
        \STATE $y_j^p \gets A_p(x_{t_j}-x_{t_{j-1}})\in\RR^m$
    \ENDFOR
    \STATE $(\widehat f_j,\widehat{\mathcal G}_j^1,\dots,\widehat{\mathcal G}_j^q)\gets \textsc{EvalFG}(y_j^1,\dots,y_j^q,N,\widehat\psi_j,\widehat\Phi_j)$ \hfill \emph{// cf.\ \Cref{alg:fsskFG}}
    \STATE ${\bf B}_j \gets \widehat f_j \in (\mathfrak R_N)^{1\times R}$
    \FOR{$l=1,\dots,q$}
        \STATE ${\bf B}_j \gets {\bf B}_j + {\bf Z}_{j-1}^l.\widehat{\mathcal G}_j^l$
    \ENDFOR
    \FOR{$p=1,\dots,q$}
        \STATE ${\bf Z}_j^p \gets {\bf Z}_{j-1}^p.E_j + {\bf B}_j\otimes y_j^p$
    \ENDFOR
\ENDFOR
\end{algorithmic}
\vspace{1mm}
\par\noindent\raggedright
\emph{Precomputation.} The computations in lines~\ref{line:fssk_delta}--\ref{line:fssk_coeffs} depend only on the interval length $\delta_j$.
Hence, these computations are reusable across repeated interval lengths.\\
\emph{Parallelization.} The loop over $j$ is sequential, while the loops over $p$ are independent and can be parallelized.
\end{algorithm}

\begin{algorithm}[h]
\caption{Readout of the Volterra signature from a finite state space recursion state}\label{alg:fssk_vsig_readout}
\begin{algorithmic}[1]
\STATE \textbf{Input:} state $\{{\bf Z}^p\}_{p=1}^q$ at time $t$ (starting at $t_0$), time $\tau\ge t$, kernel data $(\Lambda,\{b_p\}_{p=1}^q)$
\STATE \hfill \emph{// state computed by \Cref{alg:fssk_state_recursion}}
\STATE \textbf{Output:} ${\bf v}^{\tau}=\pi_{\le N}\VSig{x}{K_{A,b}^{\Lambda}}_{t_0,t}^{\tau}\in T^N(\RR^m)$
\STATE $E_{t,\tau}\gets e^{-\Lambda(\tau-t)}$ \hfill \emph{// Jordan form, cf.\ \Cref{lem:prony_implies_real_jordan_scalar}}
\FOR{$p=1,\dots,q$}
    \STATE $r_p \gets E_{t,\tau}b_p \in \RR^R$
\ENDFOR
\STATE ${\bf v}^{\tau} \gets 1+\sum_{p=1}^q {\bf Z}^p.r_p$
\end{algorithmic}
\end{algorithm}

\begin{prop}\label{prop:fssk_symmetry}
    We work in the setting of \cref{thm:algo_multiplicative} with a kernel $K_{A,b}^{\Lambda}$ of the form \eqref{eq:algo_exponential_kernel} for given data $(\Lambda, A, b)$.
    For any $\ell\in \NN^q$ we define the ordered word with letter frequencies according to $\ell = (\ell_1, \dots, \ell_q)$ by the concatenation
    \begin{equation*}
        w(\ell) = 1^{\ell_1}\cdots q^{\ell_q} = \prod_{p\in\mathcal{A}_q} \underbrace{p\cdots p}_{\ell_p \text{times}}\in \mathcal{W}_q^{|\ell|}.
    \end{equation*}
    Then, \cref{thm:algo_multiplicative} still holds after replacing the definition of $\mathcal{F}$ and $\mathcal{G}$ in \eqref{eq:def_GF} respectively by
    \begin{equation}\label{eq:GF_shuffle}
    \begin{split}
		{\mathcal{F}}_\delta(y_1,\dots, y_q) =&  \sum_{|\ell| \le N-1}  \frac{1}{\delta^{|\ell|+1}\ell!}{\Psi}^{w(\ell)}(\delta) \, y_1^{\shuffle \ell_1} \shuffle \cdots \shuffle y_q^{\shuffle \ell_q}\quad \in(\mathfrak{R}_{N})^{R\times R},\\
		{\mathcal{G}}^{p}_\delta(y_1,\dots, y_q) =&  \sum_{|\ell| \le N-2}  \frac{1}{\delta^{|\ell|+1}\ell!}{\Phi}^{pw(\ell)}(\delta) \, y_1^{\shuffle \ell_1} \shuffle \cdots \shuffle y_q^{\shuffle \ell_q}\quad \in(\mathfrak{R}_{N})^{R\times R},
	\end{split}
	\end{equation}
    where $\Psi(\delta)$ and $\Phi(\delta)$ are as in \Cref{lem:exp_weight_factorization}.
    
\end{prop}

\begin{algorithm}[h]
\caption{Subroutine \textsc{EvalFG}: shuffle-recursive evaluation of ${\bf 1}^\top\mathcal F$ and $\mathcal G$}\label{alg:fsskFG}
\begin{algorithmic}[1]
\STATE \textbf{Input:} $y_1,\dots,y_q\in\RR^m$, truncation level $N\in\NN$, normalized coefficients $\{\widehat\psi_\ell\in\RR^{1\times R}:|\ell|\le N-1\}$ and $\{\widehat\Phi_{p,\ell}\in\RR^{R\times R}:p=1,\dots,q,\ |\ell|\le N-2\}$
\STATE \textbf{Output:} $\widehat f\approx {\bf 1}^\top\mathcal F_\delta(y_1,\dots,y_q)\in(\mathfrak R_N)^{1\times R}$ and $\widehat{\mathcal G}^p\approx \mathcal G^p_\delta(y_1,\dots,y_q)\in(\mathfrak R_N)^{R\times R}$ for $p=1,\dots,q$
\FOR{$n=N-1,\dots,0$}
    \FOR{each $\ell\in\NN^q$ with $|\ell|=n$}
        \STATE $F(\ell)\gets \widehat\psi_\ell/\ell!$
        \FOR{$p=1,\dots,q$}
            \STATE $G_p(\ell)\gets
            \begin{cases}
                \widehat\Phi_{p,\ell}/\ell!, & n\le N-2,\\
                0, & n=N-1
            \end{cases}$
        \ENDFOR
        \IF{$n\le N-2$}
            \FOR{$r=1,\dots,q$}
                \STATE $F(\ell)\gets F(\ell)+\big(F(\ell+1_r)\shuffle y_r\big)(\ell_r+1)/(|\ell|+1)$
                \FOR{$p=1,\dots,q$}
                    \STATE  $G_p(\ell)\gets G_p(\ell)+\big(G_p(\ell+1_r)\shuffle y_r\big)(\ell_r+1)/(|\ell|+1)$
                \ENDFOR
            \ENDFOR
        \ENDIF
    \ENDFOR
\ENDFOR
\STATE $\widehat f\gets F(0,\dots,0)$
\FOR{$p=1,\dots,q$}
    \STATE $\widehat{\mathcal G}^p\gets G_p(0,\dots,0)$
\ENDFOR
\end{algorithmic}
\vspace{1mm}
\par\noindent\raggedright
\emph{Parallelization.} The recursion in the total degree $n$ is sequential. For fixed $n$, the updates over all multi-indices $\ell$ with $|\ell|=n$ are independent and can be parallelized. For fixed $\ell$, the loops over $r$ and $p$ are also parallelizable.
\end{algorithm}

\begin{proof}
Continuing in the proof of \Cref{thm:algo_multiplicative} we first note that for any $w_1, \dots, w_k\in \cW_q$ there exist permutations $\sigma_i \in \mathcal{S}_{|w_i|}$ ($i=1, \dots, k$) such that 
$\sigma_i(w_i) = w(\ell^{i})$, where $\ell^{i} = (\ell^i_1, \dots, \ell^i_q) \in \NN^q$ is defined as follows:
writing $w_i = p^{i}_1\cdots p^{i}_{|w_i|}$ with $p^{i}_r\in\cA_q$, we set
$$\ell^{i}_p = \big\vert\{r = 1, \dots, |w_i|\;\vert\; p^{i}_r = p\}\big\vert, \qquad p \in \mathcal{A}_q.$$
Hence, by the permutation invariance of $\cK$ established in \Cref{lem:exp_weight_factorization} we have for such $(w_i, \ell^i)$ and any $p_1, \dots, p_k \in \cA_q$ that 
$$\cK^{w_1p_1, \dots, w_kp_k}_{i_1, \dots, i_{k},l} ~=~ \cK^{\sigma_1(w_1)p_1, \dots, \sigma_k(w_k)p_k}_{i_1, \dots, i_{k},l}  ~=~  \cK^{w(\ell^1)p_1, \dots, w(\ell^{k})p_{k}}_{i_1, \dots, i_{k},l}.$$
In particular, for $\cM$ defined as in \eqref{eq:def_cM_fsskproof} we obtain
$$\cM^{(w_1p_1, \dots, w_{k-1}p_{k-1}, w_k)}.E(\tau - t_j)b_{p_k} ~=~ \cM^{(w(\ell^1)p_1, \dots, w(\ell^{k-1})p_{k-1}, w(\ell^k))}.E(\tau - t_j)b_{p_k}.$$
Consequently, after multiplying \eqref{eq:Z_induction} by $E(\tau - t_j)b_{p_k}$ from the right, we can group the coefficients in terms of $(\ell^1, \dots, \ell^k)$.
Then, with $\cW^{\ell}_q$ defined as in \eqref{eq:def_words_by_ell} for $\ell \in \NN^q$ we have for all $p\in\cA_q$ and $y_1, \dots, y_q \in \RR^m$ that
$$\sum_{p_1\cdots p_{|\ell|}\in \cW_q^{\ell}}y_{p_1} \otimes \cdots \otimes y_{p_{|\ell|}} \otimes y_p
~=~  \left(\frac{1}{\ell!}y_1^{\shuffle \ell_1}\shuffle \cdots \shuffle y_q^{\shuffle \ell_q}\right) \otimes y_p.$$
Indeed, this follows by the same inductive argument as in the proof of \Cref{prop:mathcalE}.
The proof now follows by verifying once again inductively that the resulting expression is the expanded form of $(\bZ^1, \dots, \bZ^q)$, now with $\cF$ and $\cG$ defined as in \eqref{eq:GF_shuffle}.
We omit this verification as it is analogous to the previous induction.
\end{proof}

\subsubsection{The case $q=1$}
Here, there is only a single state vector ${\bf Z}$, and the recursion simplifies substantially, as described in \Cref{cor:fsskq1}. 
In \Cref{alg:fsskq1_state_recursion} we state the corresponding state recursion for this scalar-kernel case. 
We do not repeat the Volterra-signature readout, since it is the same linear readout as in the general case and is already covered by \Cref{alg:fssk_vsig_readout}.

The main difference in the state recursion is the evaluation of the state update. 
In the case $q=1$, this update admits a Horner-type scheme, which significantly improves the computational costs (see \Cref{sec:compcostFSSK}). 
This specialized update is described in \Cref{alg:fsskq1_horner_update}.
This Horner-type update is the direct analogue of the fused Horner evaluation of ${\bv}\otimes_N\mathcal{E}$ in the scalar case from \Cref{sec:quadq1} (cf.\ \Cref{alg:quad_evalVtE_q1}), adapted to the matrix-valued state ${\bf Z}\in(\mathfrak R_N)^{1\times R}$.
Note that the evaluation of the coefficient matrices $\phi_n(\delta)$ and $\psi_n(\delta)$ is already covered by \Cref{alg:fssk_weight_quad_eval} (specialized to $q=1$), which is already sufficiently efficient in this setting, and we do not pursue further improvements here.

\begin{algorithm}[h]
\caption{State recursion for finite state space kernels in the scalar-kernel case $q=1$}\label{alg:fsskq1_state_recursion}
\begin{algorithmic}[1]
\STATE \textbf{Input:} partition $0=t_0<t_1<\cdots<t_J\le T$, piecewise linear path values $\{x_{t_j}\}_{j=0}^J$, kernel data $(\Lambda,b)$, truncation level $N\in\NN$
\STATE \textbf{Output:} states $\{{\bf Z}_j\in(\mathfrak R_N)^{1\times R}: j=0,\dots,J\}$

\STATE \textbf{Initialize:} ${\bf Z}_0 \gets (0,\dots,0)\in\RR^{1\times R}\subset(\mathfrak R_N)^{1\times R}$
\FOR{$j=1,\dots,J$}
    \STATE $\Delta x_j \gets x_{t_j}-x_{t_{j-1}}\in\RR^d$
    \STATE $\delta_j \gets t_j-t_{j-1}$  \label{line:fsskq1_delta}
    \STATE $\widehat\phi_{j,0} \gets e^{-\Lambda\delta_j}\in\RR^{R\times R}$ \hfill \emph{// Jordan form, cf.\ \Cref{lem:prony_implies_real_jordan_scalar}} \label{line:fsskq1_exp}
    \STATE $(\{\widehat\psi_{j,n}\}_{n=1}^{N},\{\widehat\phi_{j,n}\}_{n=1}^{N-1})\gets \textsc{EvalPhiPsi}(\delta_j,\Lambda,\{b\},N-1)$ \hfill \emph{// \Cref{alg:fssk_weight_quad_eval} (*)} \label{line:fsskq1_coeffs}
    \STATE ${\bf Z}_j \gets \textsc{UpdateState}({\bf Z}_{j-1},\Delta x_j,N,\widehat\psi_j,\widehat\phi_j)$ \hfill \emph{// cf.\ \Cref{alg:fsskq1_horner_update}}
\ENDFOR
\end{algorithmic}
\vspace{1mm}
\par\noindent\raggedright
\emph{Precomputation.} The computations in lines~\ref{line:fsskq1_delta}--\ref{line:fsskq1_coeffs} depend only on $\delta_j$.\\
\emph{(*).} \Cref{alg:fssk_weight_quad_eval} returns $\{{\bf 1}^\top \widehat\Psi_{\ell}\,:\,|\ell|\le N-1\}$ and $\{\widehat\Phi_{p,\ell}\;:\; p=1,\dots,q,\ |\ell|\le N-2\}$.
Here $q=1$, so $\ell\in\NN$ and $p=1$, and we identify
$\widehat\psi_{n}:={\bf 1}^\top \widehat\Psi_{n-1}$ and
$\widehat\phi_{j,n}:=\widehat\Phi_{1,n-1}$.
\end{algorithm}

\begin{algorithm}[h]
\caption{Subroutine \textsc{UpdateState}: Horner-type state update in the $q=1$ case}\label{alg:fsskq1_horner_update}
\begin{algorithmic}[1]
\STATE \textbf{Input:} previous state ${\bf Z}\in(\mathfrak R_N)^{1\times R}$, increment $\Delta x\in\RR^d$, truncation level $N\in\NN$, normalized coefficients
$\widehat\psi=\{\widehat\psi_n\in\RR^{1\times R}:n=1,\dots,N\}$ and
$\widehat\phi=\{\widehat\phi_n\in\RR^{R\times R}:n=0,\dots,N-1\}$
\STATE \textbf{Output:} updated state $\widehat{\bf Z}\approx {\bf 1}^\top.\sum_{n=1}^{N}\widehat\psi_n(\Delta x)^{\otimes n} + {\bf Z}.\sum_{n=0}^{N-1}\widehat\phi_n(\Delta x)^{\otimes n}\in(\mathfrak R_N)^{1\times R}$

\STATE \textbf{Convention:} set $\widehat\phi_N\gets 0$
\STATE $\widehat{\bf Z}^{(0)}\gets {\bf Z}^{(0)}.\widehat\phi_0$
\FOR{$n=1,\dots,N$}
    \STATE $U_n\gets \widehat\psi_n$
    \STATE $W_n\gets {\bf Z}^{(0)}.\widehat\phi_n$
    \FOR{$k=1,\dots,n-1$}
        \STATE $U_n\gets U_n\otimes \Delta x$
        \STATE $W_n\gets (W_n\otimes \Delta x)+{\bf Z}^{(k)}.\widehat\phi_{n-k}$
    \ENDFOR
    \STATE $\widehat{\bf Z}^{(n)}\gets (U_n+W_n)\otimes \Delta x$
\ENDFOR
\end{algorithmic}
\vspace{1mm}
\par\noindent\raggedright
\emph{Parallelization.} The loop over $n$ is independent and can be parallelized. For fixed $n$, the recursion over $k$ is sequential.
\end{algorithm}

\section{Numerical methods for the Volterra signature kernel}\label{sec:sig-kernel}

In this section, we turn our attention to numerical aspects of the \emph{Volterra signature kernel} associated with finite state space kernels introduced in Definition~\ref{def:finite_state_space_kernels}. More precisely, we introduce numerical schemes to approximate the inner product
\begin{equation}\label{eq:signature-kernel}
	\kappa_{s,t}(x,w):= \langle \VSig{x}{K}_{0,s},\VSig{w}{K}_{0,t} \rangle \in \mathbb{R}_+, \qquad x,w \in \cC^1([0,T];\RR^d),
\end{equation} where $(s,t)\in [0,T]^2$ and $K$ is a kernel of the form \eqref{eq:algo_exponential_kernel}. The Volterra signature kernel was already introduced in \cite{i_part} for a general class of integrable matrix kernels, and a kernel trick characterizes $\kappa$ as the unique solution to a second-order Volterra integral equation; see \cite[Theorem 3.23]{i_part}. In the special case of finite state space kernels, these integral equations are shown to be equivalent to a system of Goursat PDEs, which we shall briefly recall now.

Consider matrix-valued functions $\bK,\mathbf{\Psi},\mathbf{\Phi}\colon [0,T]^2\to \RR^{R\times R}$ satisfying the coupled PDE system on $[0,T]^2$
\begin{align}
\partial^2_{st} \bK(s,t)
&= -\eta(s,t)\gamma(s,t)+\Lambda \bK(s,t) \Lambda^\top +\partial_s \mathbf{\Psi}(s,t)+\partial_t\mathbf{\Phi}(s,t),
\label{eq:goursat-general-K}
\\
\partial_s \mathbf{\Psi}(s,t)
&=
-\Lambda \mathbf{\Psi}(s,t)+\gamma(s,t)\,\eta(s,t),
\label{eq:goursat-general-psi}
\\
\partial_t \mathbf{\Phi}(s,t)
&=
-\mathbf{\Phi}(s,t)\Lambda^\top+\gamma(s,t)\,\eta(s,t),
\label{eq:goursat-general-phi}
\end{align}
where $\eta(s,t):=1+\mathbf 1^\top \bK(s,t)\mathbf 1$ and with boundary conditions
\begin{equation}
\bK(s,0)=\bK(0,t)=
\mathbf{\Psi}(0,t)= \mathbf{\Phi}(s,0)=0, \qquad (s,t) \in [0,T]^2.
\label{eq:goursat-general-bc}
\end{equation} The matrix $\Lambda\in\RR^{R\times R}$ appears in the kernel \eqref{eq:algo_exponential_kernel}, and $\gamma(s,t)\in \mathbb{R}^{R\times R}$ is given by 
\begin{equation}\label{eq:alpha_gamma_2}
	\gamma(s,t)\in \mathbb{R}^{R \times R}, \qquad \gamma_{ij}(s,t) = \sum_{r,p=1}^q b_{r}^ib_{p}^j \la A_r \dot x_s,A_p\dot{w}_t \ra,
\end{equation}
where $A$, $b$, $R$, and $q$ are as in the finite state space kernel \eqref{eq:algo_exponential_kernel}.

\begin{thm}[Theorem 3.26 in \cite{i_part}]
Let $x,w \in \cC^{1}([0,T];\RR^d)$ and suppose that $K$ is a kernel as in Definition~\ref{def:finite_state_space_kernels}. Then $\kappa=\eta$, where $\eta$ is obtained from the unique solution $(\bK,\mathbf{\Psi},\mathbf{\Phi})$ of \eqref{eq:goursat-general-K}--\eqref{eq:goursat-general-bc}.\footnote{We note that in \cite[Theorem 3.26]{i_part}, the equation \eqref{eq:goursat-general-K} was equivalently written with $\partial_s \mathbf{\Psi}$ and $\partial_t \mathbf{\Phi}$ replaced by the right-hand sides of \eqref{eq:goursat-general-psi} and \eqref{eq:goursat-general-phi}. The representation in the present paper is motivated by the numerical schemes.}
\end{thm}

\begin{rem}
	We emphasize that for $R=q=1$, $\Lambda=0$, $b=1$, and $A=\mathrm{Id}$, the equation \eqref{eq:goursat-general-K} reduces to the Goursat PDE characterizing the classical signature kernel established in \cite{Salvi2021}. The generalization to our framework requires the auxiliary states \eqref{eq:goursat-general-psi}--\eqref{eq:goursat-general-phi}, which enter linearly into the main equation \eqref{eq:goursat-general-K}. In addition, these states are characterized by linear ODEs and therefore admit the explicit representations
	\begin{align}\label{eq:auxilliary_expl_1}
	\mathbf{\Psi}(s,t)
	& =
	\int_0^s e^{-\Lambda(s-u)}\eta(u,t)\gamma(u,t)\,\dd u,
	\\ 
	\mathbf{\Phi}(s,t)
	& =
	\int_0^t \gamma(s,u)\eta(s,u) e^{-\Lambda^\top(t-u)}\,\dd u. \label{eq:auxilliary_expl_2}
	\end{align}
	With these observations in hand, it should not come as a surprise that the numerical schemes introduced in the classical case \cite{Salvi2021} admit straightforward generalizations to our framework, which we illustrate in the next sections.
\end{rem}
\subsection{Static kernel lift}
A key advantage of the classical signature kernel trick
\cite{Kiraly2016,Salvi2021} is that, without losing tractability, one may first
lift the underlying data from the ambient space into an RKHS,
\(x_t \mapsto \phi(x_t)\in \cH\), and then apply the signature kernel transform
in this possibly infinite-dimensional space. The same idea can be incorporated
in our Volterra setting, as we briefly outline now. 

Given a kernel from the class \eqref{eq:algo_exponential_kernel} and a path
$x\in \cC^{0,1}([0,T];\RR^d)$, recall that the corresponding Volterra path reads
\[
\int_s^t K(t,u)\dd x_u
=
\sum_{r=1}^q
\int_s^t
\mathbf{1}^\top e^{-\Lambda(t-u)} b_r \, \dd (A_r x_u),
\qquad
A_r \in \cL(\RR^d;\RR^m).
\]
Thus, the linear maps \(A_1,\dots,A_q\) may be interpreted as preprocessing
channels which transform the input data before the Volterra kernel is applied. In the spirit of the static kernel lift, one may replace these finite-dimensional
preprocessing maps by feature maps taking values in a Hilbert space. More
precisely, let \(A_r:\RR^d\to \cH\) be sufficiently regular feature maps. %
The resulting Volterra signature then takes values in
\(T((\cH))\), so that even truncated signatures are infinite-dimensional whenever
\(\cH\) is infinite-dimensional. The important observation is that the PDE
characterization of the Volterra signature kernel in
\eqref{eq:signature-kernel}--\eqref{eq:goursat-general-psi} remains valid.
Only the coefficient \(\gamma\) has to be replaced by the Hilbert-space inner
product of the lifted increments, namely
\begin{equation}\label{eq:coeff_Hilbi}
    \gamma(s,t)\in \RR^{R\times R},
    \qquad
    \gamma_{ij}(s,t)
    =
    \sum_{r,p=1}^q
    b_r^i b_p^j
    \big\langle
        \dot X_s^r,
        \dot W_t^p
    \big\rangle_{\cH}, \qquad X_s^r =A_r(x_s), \quad  W_t^p = A_p(w_t).
\end{equation}

Hence, provided that the Hilbert-space inner products can be evaluated
efficiently, for instance through a reproducing kernel such as the RBF kernel,
the computation of the Volterra signature kernel remains 
tractable. 

\subsection{A predictor-corrector finite-difference scheme}
We now present our main numerical scheme for solving the PDE system
\eqref{eq:goursat-general-K}-\eqref{eq:goursat-general-psi}. The scheme applies,
in particular, to the general coefficient matrix \eqref{eq:coeff_Hilbi} arising
from a static RKHS lift.

We consider two grids on $[0,S]$ and $[0,T]$ respectively
\[
0=s_0<\cdots<s_{J_s}=S,
\qquad
0=t_0<\cdots<t_{J_t}=T,
\]
and set
\[
\Delta s_i:=s_{i+1}-s_i,
\qquad
\Delta t_j:=t_{j+1}-t_j.
\]
In the sequel, we write
\[
\bK_{i,j}\approx \bK(s_i,t_j),\qquad
\mathbf{\Psi}_{i,j}\approx \mathbf{\Psi}(s_i,t_j),\qquad
\mathbf{\Phi}_{i,j}\approx \mathbf{\Phi}(s_i,t_j),
\]
and set $\eta_{i,j}:=1+\mathbf 1^\top \bK_{i,j}\mathbf 1.$ Moreover, we denote the static kernel induced by the feature maps $A_1,\dots,A_r$ described in the last section, by $k^{r,p}_\mathrm{stat}(x,w):= \la A_r(x),A_p(w)\ra_{\cH}$. The coefficient matrix $\gamma$ is approximated by \[
\gamma_{nm}(s_i,t_j) \approx \gamma^{nm}_{ij}:= b^\top G_{i,j} b, \quad G_{i,j} = (G_{i,j}^{a,b})_{1\leq a,b \leq q}, 
\] 
where 
\begin{align*}
G_{i,j}^{a,b}=k^{a,b}_\mathrm{stat}(x_{s_{i+1}},w_{t_{j+1}})& - k^{a,b}_\mathrm{stat}(x_{s_{i}},w_{t_{j+1}}) - k^{a,b}_\mathrm{stat}(x_{s_{i+1}},w_{t_{j}})+k^{a,b}_\mathrm{stat}(x_{s_{i}},w_{t_{j}}).
\end{align*}
According to the boundary conditions \eqref{eq:goursat-general-bc}, we initialize
\[
\bK_{i,0}=0,
\qquad
\bK_{0,j}=0,
\qquad
\mathbf{\Psi}_{0,j}=0,
\qquad
\mathbf{\Phi}_{i,0}=0,
\qquad
i=0,\dots,J_s,\quad j=0,\dots,J_t,
\]
and hence, in particular, $\eta_{i,0}=\eta_{0,j}=1$. Moreover, we introduce the exponential transport operators
\[
E_i^s:=e^{-\Delta s_i\Lambda},
\qquad
P_i^s:=\Delta s_i\,\varphi_1(-\Delta s_i\Lambda),
\qquad
Q_i^s:=\Delta s_i\,\varphi_2(-\Delta s_i\Lambda),
\]
and
\[
E_j^t:=e^{-\Delta t_j\Lambda^\top},
\qquad
P_j^t:=\Delta t_j\,\varphi_1(-\Delta t_j\Lambda^\top),
\qquad
Q_j^t:=\Delta t_j\,\varphi_2(-\Delta t_j\Lambda^\top),
\]
where $\varphi_1(z)=(e^z-1)/z$ and $
\varphi_2(z)=(e^z-z-1)/z^2.$ On each cell, we first compute a predictor. To this end, we set
\[
\mathbf H^{(0)}_{i,j}
:=
\gamma_{i,j}\,\frac{\eta_{i,j}+\eta_{i,j+1}}{2},
\qquad
\mathbf J^{(0)}_{i,j}
:=
\gamma_{i,j}\,\frac{\eta_{i,j}+\eta_{i+1,j}}{2}.
\]
Using these quantities, we define the predictor values for the auxiliary variables by
\begin{align*}
\mathbf{\Psi}^{p}_{i+1,j+1}
&=
E_i^s\,\mathbf{\Psi}_{i,j+1}
+
P_i^s\,\mathbf H^{(0)}_{i,j},
\\
\mathbf{\Phi}^{p}_{i+1,j+1}
&=
\mathbf{\Phi}_{i+1,j}E_j^t
+
\mathbf J^{(0)}_{i,j}P_j^t.
\end{align*}
With the notation $\cL(M):=\Lambda M \Lambda ^\top$, the predictor for the main variable is then given by
\begin{align*}
\bK^{p}_{i+1,j+1}
&=
\bK_{i+1,j}
+
\bK_{i,j+1}
-
\bK_{i,j} 
+
\frac{\Delta s_i\Delta t_j}{2}
\Bigl(
\cL(\bK_{i+1,j})+\cL(\bK_{i,j+1})
\Bigr)
\\ & \qquad
-
\frac12\,\gamma_{i,j}\bigl(\eta_{i+1,j}+\eta_{i,j+1}\bigr)
+
\bigl(\mathbf{\Psi}^{p}_{i+1,j+1}-\mathbf{\Psi}_{i,j+1}\bigr)
+
\bigl(\mathbf{\Phi}^{p}_{i+1,j+1}-\mathbf{\Phi}_{i+1,j}\bigr).
\end{align*}
Finally, we set $\eta^{p}_{i+1,j+1}
:=
1+\mathbf 1^\top \bK^{p}_{i+1,j+1}\mathbf 1.$ Given the predictors, the corrector step we update the source terms by is
\[
\mathbf H^{(1)}_{i,j}
:=
\gamma_{i,j}\,\frac{\eta_{i+1,j}+\eta^{p}_{i+1,j+1}}{2},
\qquad
\mathbf J^{(1)}_{i,j}
:=
\gamma_{i,j}\,\frac{\eta_{i,j+1}+\eta^{p}_{i+1,j+1}}{2}.
\]
We then define
\begin{align*}
\mathbf{\Psi}_{i+1,j+1}
&=
E_i^s\mathbf{\Psi}_{i,j+1}
+
P_i^s\mathbf H^{(0)}_{i,j}
+
Q_i^s\bigl(\mathbf H^{(1)}_{i,j}-\mathbf H^{(0)}_{i,j}\bigr),
\\
\mathbf{\Phi}_{i+1,j+1}
&=
\mathbf{\Phi}_{i+1,j}E_j^t
+
\mathbf J^{(0)}_{i,j}P_j^t
+
\bigl(\mathbf J^{(1)}_{i,j}-\mathbf J^{(0)}_{i,j}\bigr)Q_j^t.
\end{align*}
Finally, the corrected update for $\bK$ is given by
\begin{align*}
\bK_{i+1,j+1}
&=
\bK_{i+1,j}
+
\bK_{i,j+1}
-
\bK_{i,j}
\\ & \quad 
+
\frac{\Delta s_i\Delta t_j}{4}
\Bigl(
\cL(\bK_{i,j})
+
\cL(\bK_{i+1,j})
+
\cL(\bK_{i,j+1})
+
\cL(\bK^{p}_{i+1,j+1})
\Bigr)
\\
&\quad
-
\frac14\,\gamma_{i,j}
\Bigl(
\eta_{i,j}
+
\eta_{i+1,j}
+
\eta_{i,j+1}
+
\eta^{p}_{i+1,j+1}
\Bigr)
\\
&\quad
+
\bigl(\mathbf{\Psi}_{i+1,j+1}-\mathbf{\Psi}_{i,j+1}\bigr)
+
\bigl(\mathbf{\Phi}_{i+1,j+1}-\mathbf{\Phi}_{i+1,j}\bigr).
\end{align*}
The scalar quantity at the new corner is then updated by $\eta_{i+1,j+1}
=
1+\mathbf 1^\top \bK_{i+1,j+1}\mathbf 1.$ This yields an explicit predictor--corrector scheme for the transformed system
\eqref{eq:goursat-general-K}--\eqref{eq:goursat-general-phi}. More precisely, on each cell
$[s_i,s_{i+1}]\times[t_j,t_{j+1}]$, the values
\[
\bK_{i,j},\quad \bK_{i+1,j},\quad \bK_{i,j+1},\quad
\mathbf{\Psi}_{i,j+1},\quad \mathbf{\Phi}_{i+1,j}
\]
are already known, and the above formulas produce
\[
\bK_{i+1,j+1},\qquad
\mathbf{\Psi}_{i+1,j+1},\qquad
\mathbf{\Phi}_{i+1,j+1},\qquad
\eta_{i+1,j+1}.
\]

\begin{algorithm}[h]
\caption{Predictor--corrector finite-difference scheme}\label{alg:fd_scheme_antidiag_pc}
\begin{algorithmic}[1]
\STATE \textbf{Input:} grids $\{s_i\}_{i=0}^{J_s}$, $\{t_j\}_{j=0}^{J_t}$, cell coefficients $\gamma_{i,j}\in\RR^{R\times R}$, matrix $\Lambda\in\RR^{R\times R}$, and precomputed matrices
\[
E_i^s=e^{-\Delta s_i\Lambda},\quad
P_i^s=\Delta s_i\varphi_1(-\Delta s_i\Lambda),\quad
Q_i^s=\Delta s_i\varphi_2(-\Delta s_i\Lambda),
\]
\[
E_j^t=e^{-\Delta t_j\Lambda^\top},\quad
P_j^t=\Delta t_j\varphi_1(-\Delta t_j\Lambda^\top),\quad
Q_j^t=\Delta t_j\varphi_2(-\Delta t_j\Lambda^\top).
\]
\STATE \textbf{Output:} arrays $\{\bK_{i,j},\mathbf{\Psi}_{i,j},\mathbf{\Phi}_{i,j},\eta_{i,j}\}$ and $\kappa_{J_s,J_t}\approx \eta_{J_s,J_t}$

\STATE \textbf{Initialization:}
\[
\bK_{i,0}=\bK_{0,j}=0,\qquad
\mathbf{\Psi}_{0,j}=0,\qquad
\mathbf{\Phi}_{i,0}=0,\qquad
\eta_{i,0}=\eta_{0,j}=1.
\]

\STATE \textbf{Set} $\cL(\bM):=\Lambda \bM \Lambda^\top$.

\FOR{$m=0,\dots,J_s+J_t-2$}
\STATE \textbf{Parallel for each cell} $(i,j)$ with $i+j=m$, $0\le i<J_s$, $0\le j<J_t$:
\begingroup
\setlength{\baselineskip}{1.3\baselineskip}

\STATE \hspace{1.5em}$\mathbf H^{(0)}_{i,j}\gets 0.5\,\gamma_{i,j}(\eta_{i,j}+\eta_{i,j+1}),\qquad
\mathbf J^{(0)}_{i,j}\gets 0.5\,\gamma_{i,j}(\eta_{i,j}+\eta_{i+1,j})$

\STATE \hspace{1.5em}$\mathbf{\Psi}^{\mathrm{pred}}_{i+1,j+1}\gets E_i^s\mathbf{\Psi}_{i,j+1}+P_i^s\mathbf H^{(0)}_{i,j}$
\STATE \hspace{1.5em}$\mathbf{\Phi}^{\mathrm{pred}}_{i+1,j+1}\gets \mathbf{\Phi}_{i+1,j}E_j^t+\mathbf J^{(0)}_{i,j}P_j^t$

\STATE \hspace{1.5em}$\bK^{\mathrm{pred}}_{i+1,j+1}\gets
\bK_{i+1,j}+\bK_{i,j+1}-\bK_{i,j}$
\STATE \hspace{7.3em}$
+0.5\Delta s_i\Delta t_j\bigl(\cL(\bK_{i+1,j})+\cL(\bK_{i,j+1})\bigr)$
\STATE \hspace{7.3em}$
-0.5\,\gamma_{i,j}(\eta_{i+1,j}+\eta_{i,j+1})$
\STATE \hspace{7.3em}$
+\bigl(\mathbf{\Psi}^{\mathrm{pred}}_{i+1,j+1}-\mathbf{\Psi}_{i,j+1}\bigr)
+\bigl(\mathbf{\Phi}^{\mathrm{pred}}_{i+1,j+1}-\mathbf{\Phi}_{i+1,j}\bigr)$

\STATE \hspace{1.5em}$\eta^{\mathrm{pred}}_{i+1,j+1}\gets 1+\mathbf 1^\top \bK^{\mathrm{pred}}_{i+1,j+1}\mathbf 1$

\STATE \hspace{1.5em}$\mathbf H^{(1)}_{i,j}\gets 0.5\,\gamma_{i,j}(\eta_{i+1,j}+\eta^{\mathrm{pred}}_{i+1,j+1}),\qquad
\mathbf J^{(1)}_{i,j}\gets 0.5\,\gamma_{i,j}(\eta_{i,j+1}+\eta^{\mathrm{pred}}_{i+1,j+1})$

\STATE \hspace{1.5em}$\mathbf{\Psi}_{i+1,j+1}\gets
E_i^s\mathbf{\Psi}_{i,j+1}+P_i^s\mathbf H^{(0)}_{i,j}
+Q_i^s\bigl(\mathbf H^{(1)}_{i,j}-\mathbf H^{(0)}_{i,j}\bigr)$

\STATE \hspace{1.5em}$\mathbf{\Phi}_{i+1,j+1}\gets
\mathbf{\Phi}_{i+1,j}E_j^t+\mathbf J^{(0)}_{i,j}P_j^t
+\bigl(\mathbf J^{(1)}_{i,j}-\mathbf J^{(0)}_{i,j}\bigr)Q_j^t$

\STATE \hspace{1.5em}$\bK_{i+1,j+1}\gets
\bK_{i+1,j}+\bK_{i,j+1}-\bK_{i,j}$
\STATE \hspace{7.3em}$
+0.25\Delta s_i\Delta t_j\bigl(\cL(\bK_{i,j})+\cL(\bK_{i+1,j})+\cL(\bK_{i,j+1})+\cL(\bK^{\mathrm{pred}}_{i+1,j+1})\bigr)$
\STATE \hspace{7.3em}$
-0.25\,\gamma_{i,j}\bigl(\eta_{i,j}+\eta_{i+1,j}+\eta_{i,j+1}+\eta^{\mathrm{pred}}_{i+1,j+1}\bigr)$
\STATE \hspace{7.3em}$
+\bigl(\mathbf{\Psi}_{i+1,j+1}-\mathbf{\Psi}_{i,j+1}\bigr)
+\bigl(\mathbf{\Phi}_{i+1,j+1}-\mathbf{\Phi}_{i+1,j}\bigr)$

\STATE \hspace{1.5em}$\eta_{i+1,j+1}\gets 1+\mathbf 1^\top \bK_{i+1,j+1}\mathbf 1$

\endgroup
\ENDFOR

\STATE \textbf{return} $\{\bK_{i,j},\mathbf{\Psi}_{i,j},\mathbf{\Phi}_{i,j},\eta_{i,j}\}$ and $\kappa_{J_s,J_t}\gets \eta_{J_s,J_t}$
\end{algorithmic}
\end{algorithm}
\begin{rem}[Dyadic refinement]

Suppose that the input paths $x,w$ are piecewise linear on two grids
\[
\cI_0=\{0=s_0<\cdots<s_{J_s}=S\},
\qquad
\cJ_0=\{0=t_0<\cdots<t_{J_t}=T\},
\]
and set $\cP_0:=\cI_0\times \cJ_0$. For any $\lambda\in\NN_0$, we define the dyadic refinement $\cP_\lambda$ by
\[
\cP_\lambda \cap \bigl([s_i,s_{i+1}] \times [t_j,t_{j+1}]\bigr)
:=
\left\{
\bigl(s_i+k2^{-\lambda}\Delta s_i,\; t_j+l2^{-\lambda}\Delta t_j\bigr)
\right\}_{0\leq k,l \leq 2^{\lambda}}.
\]
The predictor-corrector scheme described above can then be applied on $\cP_\lambda$ without changing its local update rule: one simply replaces the mesh widths $\Delta s_i,\Delta t_j$ by the refined step sizes
\[
\Delta s_i^\lambda:=2^{-\lambda}\Delta s_i,
\qquad
\Delta t_j^\lambda:=2^{-\lambda}\Delta t_j,
\]
and performs the same cellwise update for the unknowns $\bK$, $\mathbf{\Psi}$ and $\mathbf{\Phi}$ on the refined mesh.
\end{rem}

\section{Computational cost analysis}\label{sec:cost_analysis}

In this section we analyze the computational costs of the algorithms for Volterra signatures introduced in \Cref{sec:algo_quadratic} and the finite state space kernels in \Cref{sec:fssk_algorithmic}, and finally, the finite-difference scheme presented in  Section~\ref{sec:sig-kernel}. By \emph{computational cost} we mean the number of elementary arithmetic operations (additions and multiplications) required by a dense implementation of the stated recursions, up to constant factors that depend only on fixed model parameters.

To limit the scope, we focus on the asymptotic dependence of these costs on the number of time steps $J$, the truncation level $N$, and, where applicable, the state space dimension $R$. Concretely, we introduce the following definition.

\begin{defn}\label{def:cost_asymptotic_dependence}
We say that the cost of a specific (dense) implementation of an algorithm is \emph{asymptotically in $(J,R,N)$ of order $\Theta(f(J,R,N))$} for some function $f:\NN^3\to(0,\infty)$ if there exist constants $c,C>0$ and $J_0,R_0,N_0\in\NN$ such that the number of elementary arithmetic operations $\mathrm{Cost}(J,R,N)$ satisfies
$$
c\,f(J,R,N) \le \mathrm{Cost}(J,R,N)\le C\,f(J,R,N)
$$
for all $J\ge J_0$, $R\ge R_0$, $N\ge N_0$.
\end{defn}

In general, it is of course not clear a priori that the cost of a given algorithm can be described by this multivariate $\Theta$-convention.
However, in our setting the dependence on $J$ and $R$ is visible directly from explicit prefactors in the stated recursions, so the main discussion below concerns the asymptotic order in the truncation level~$N$. 

\begin{rem}
The natural benchmark for our algorithms is the cost of computing the standard (non-Volterra) truncated signature of a piecewise linear path with $J$ steps and truncation level $N$ in $\RR^{d}$, which is of order $\Theta(J\,d^{N})$. As recalled in the introduction (cf.\ \eqref{eq:sig_chen_comp} and the surrounding discussion), this cost follows from iterating Chen's identity together with a Horner-type evaluation of the tensor exponential on each linear segment. The algorithms developed in this paper for Volterra signatures necessarily carry an additional cost: the general scheme of \Cref{sec:algo_quadratic} incurs a quadratic factor $J^{2}$ instead of $J$, while the finite state space scheme of \Cref{sec:algo_multiplicative} recovers linearity in $J$, but involves a factor $R^{2}$ (where $R$ is the state space dimension).
\end{rem}

We first recap basic cost estimates for tensor operations in \Cref{sec:combinatorial_estimates}, then discuss the computational cost of the general approximation scheme from \Cref{sec:algo_quadratic} in \Cref{sec:cost_quad}. Finally, in \Cref{sec:compcostFSSK} we carry out the analogous cost analysis for the finite state space recursion.

\subsection{Basics}\label{sec:combinatorial_estimates}
Throughout this subsection we assume a dense representation of homogeneous tensors, i.e.\ an element $\bx^{(k)}\in(\RR^m)^{\otimes k}$ is stored by its $m^k$ coefficients in the standard basis.
For simplicity, we will exclude the trivial case $m=1$ throughout as this would sometimes lead to  case distinctions, i.e., we assume $m \in \NN_{\ge 2}$ fixed.
We recap the computational costs of forming (truncated) tensor products and shuffle-by-vector updates (as used in \Cref{alg:quad_evalVtE,alg:fsskFG}). Note that the following two lemmas are elementary and we include their proofs for lack of a proper reference.
For notational convenience, for two functions $f,g:\NN\to(0,\infty)$ we write $$f(N)\sim g(N), \quad \text{if and only if}\quad \lim_{N\to\infty} f(N)^{-1}g(N)\in(0,\infty).$$

\begin{lem}\label{sec:combinatorial_estimates:tensors}
Let $\bx^{(k)}\in(\RR^m)^{\otimes k}$ and $\by^{(\ell)}\in(\RR^m)^{\otimes \ell}$ with $N := k+\ell$. The cost of forming the concatenation product $\bx^{(k)}\otimes \by^{(\ell)}\in(\RR^m)^{\otimes N}$ is of order $\Theta(m^{N})$. 
Moreover, for $\bx,\by\in T^N(\RR^m)$ the costs of forming the truncated tensor product $\bx\otimes_N \by$ is of the order
$\Theta (N m^N).
$
\end{lem}
\begin{proof}
The tensor $\bx^{(k)}\otimes \by^{(\ell)}\in(\RR^m)^{\otimes(k+\ell)}=(\RR^m)^{\otimes N}$ has $m^{N}$ coefficients in the standard basis and can be formed coefficientwise, hence the cost is of order $\Theta(m^{N})$.
For $\bx,\by\in T^N(\RR^m)$, at level $n$ the truncated product satisfies
$$
\pi_n(\bx\otimes_N \by)=\sum_{r=0}^n \bx^{(r)}\otimes \by^{(n-r)},\qquad n=0,\dots,N.
$$
Each summand is a tensor in $(\RR^m)^{\otimes n}$ and can be formed coefficientwise at cost of order $\Theta(m^n)$, so computing $\pi_n(\bx\otimes_N \by)$ costs of order $\Theta((n+1)m^n)$. Summing over $n=0,\dots,N$ yields
$$
\sum_{n=0}^N (n+1)m^n
= \frac{m\big(1-(N+2)m^{N+1}+(N+1)m^{N+2}\big)}{(1-m)^2} \sim N m^N,
$$
which proves the claim.
\end{proof}

\begin{lem}\label{sec:combinatorial_estimates:shuffle}
Let $\bx^{(k)}\in(\RR^m)^{\otimes k}$ and $y\in\RR^m$ with $N:=k+1$. The cost of forming the shuffle product $\bx^{(k)}\shuffle y\in(\RR^m)^{\otimes N}$ is of order $\Theta(N m^{N})$.
\end{lem}
\begin{proof}
By \eqref{eq:shuffle_single_letter} in \Cref{def:shuffle}, the shuffle product $\bx^{(k)}\shuffle y\in(\RR^m)^{\otimes(k+1)}$ is obtained by inserting $y$ into $\bx^{(k)}$ at all possible positions. Writing $\bx^{(k)}=\sum_{i_1,\dots,i_k=1}^m \bx^{i_1\cdots i_k}\, e_{i_1}\otimes\cdots\otimes e_{i_k}$ and $y=\sum_{j=1}^m y_j e_j$, one has for $i_1,\dots,i_{k+1}\in\{1,\dots,m\}$ that
$$
(\bx^{(k)}\shuffle y)^{i_1\cdots i_{k+1}}
=
\sum_{r=1}^{k+1} \bx^{i_1\cdots i_{r-1}\, i_{r+1}\cdots i_{k+1}}\; y_{i_r}.
$$
Thus, computing each coefficient $(\bx^{(k)}\shuffle y)^{i_1\cdots i_{k+1}}$ requires $k+1=N$ multiplications and $k$ additions. Since there are $m^{k+1}=m^{N}$ coefficients, the total cost is of order $(k+1)m^{k+1}=N m^{N}$, i.e.\ $\Theta(N m^{N})$.
\end{proof}

\subsection{Costs of the general approximative algorithm}\label{sec:cost_quad}

We next summarize the asymptotic computational costs of the general approximative scheme \eqref{eq:higher_order_scheme}-\eqref{eq:coeffcicients_higher_order}, distinguishing between a direct implementation, the shuffle-recursive evaluation under \Cref{def:symmetric_cK}, and the scalar case $q=1$ with a Horner scheme.

\begin{prop}\label{prop:cost_quad}
We work in the setting of \Cref{thm:quadratic_aglo} and consider the computation
of ${\bv}_j \approx \pi_{\le N}\VSig{x}{K}^{t_j}_{0,t_j}$, $j=1,\dots,J$,
for truncation levels $N\in\NN$ via the higher-order scheme
\eqref{eq:higher_order_scheme}--\eqref{eq:coeffcicients_higher_order},
implemented as in \Cref{alg:higher_order_vsig}. Let $K$ be as in
Notation~\ref{not:comp_sec}, with scalar components
$k_1,\dots,k_q\in L^{\infty,1}(\Delta^2;\RR)$ and linear maps
$A_1,\dots,A_q\in\mathcal L(\RR^d;\RR^m)$, and let $x$ be piecewise linear on
a time grid $t_0<t_1<\cdots<t_J$.

Assume that the $\rho$-dependent local kernel coefficient tensors
$(\widehat{\cK}_{i,j}^{\rho})_{1\le i\le j,\rho\in \decoSet}$ appearing in
\Cref{alg:higher_order_vsig} are precomputed
(cf.\ \Cref{sec:kernel_computations}). If $(q,m,d,{\decoNum})$ are fixed
with $m>1$, then the computational costs have the following asymptotic
dependence with respect to $(J,N)$, in the sense of
\Cref{def:cost_asymptotic_dependence}:
\begin{enumerate}[label=(\roman*), itemsep=0.4em]
    \item If each product ${\bv}\otimes_N\mathcal E^\rho$ is evaluated by a
    direct implementation of \eqref{eq:mathcalE}, respectively its
    $\rho$-dependent analogue, then the costs are of order
    \[
        \Theta\big(J^2\,m^N\,(q^N+N)\big).
    \]

    \item\label{itm:shuffle_costs} If \Cref{def:symmetric_cK} holds and each
    product ${\bv}\otimes_N\mathcal E^\rho$ is evaluated via the
    shuffle-recursive scheme \Cref{alg:quad_evalVtE}, then the costs are of
    order
    \[
        \Theta\big(J^2\,N\,m^N\big).
    \]

    \item If $q=1$ and each product ${\bv}\otimes_N\mathcal E^\rho$ is
    evaluated via the Horner scheme \Cref{alg:quad_evalVtE_q1}, then the costs
    are of order
    \[
        \Theta\big(J^2\,m^N\big).
    \]

        \item If \Cref{hyp:fft_convolution} holds and {\bf v} is evaluated by the FFT-accelerated scheme \Cref{alg:fft}, then
    the costs are of order
    \[
        \Theta\bigl(J\log(J)\,N^q\,m^N\bigr).
    \]
\end{enumerate}
\end{prop}

\begin{rem}
In \ref{itm:shuffle_costs} it is of course not claimed that the costs do not
depend on $q$ or on the number of exponents in $\decoSet$. Rather, for fixed
$(q,{\decoNum})$ these quantities only enter through constant prefactors. The
higher-order interpolation therefore does not change the asymptotic dependence
on $(J,N)$ compared with the left-point scheme.
\end{rem}

\begin{proof}[Proof of Proposition~\ref{prop:cost_quad}]
We first consider the non-FFT implementations. As can be seen from
\Cref{alg:higher_order_vsig}, the higher-order scheme has two quadratic
summation stages. First, for each cell $i$ and interpolation node $\theta_a$,
the value $\widehat{\mathbf F}_i^{\theta_a}$ is obtained by summing all
previous-cell contributions $b<i$ and all exponents $\rho\in\decoSet$. Second,
after the interpolation coefficients $\mathbf C_{i,\rho}$ have been computed,
each diagonal readout value ${\bv}_j$ is obtained by summing the contributions
from $i=1,\dots,j$ and $\rho\in\decoSet$. Since
$|\decoSet|=\decoNum+1$ is fixed, both stages contain $\Theta(J^2)$ calls to
the same local evaluation routines
\[
    \textsc{EvalVtE}
    \bigl(
        \mathbf C_{i,\rho},
        y_i^1,\dots,y_i^q,
        N,
        \cK_{i,j}^{\rho}
    \bigr).
\]
The local interpolation systems are solved componentwise. Since their size is
fixed, this contributes only $\Theta(Jm^N)$ operations for fixed ${\decoNum}$,
and is therefore subleading compared with the quadratic summation stages. The
costs of evaluating $y_i^p=A_p(x_{t_i}-x_{t_{i-1}})$ are of order
$\Theta(J)$ for fixed $(q,m,d)$ and likewise do not affect the leading
asymptotics.

It remains to estimate the cost of one local evaluation
${\bv}\otimes_N\mathcal E^\rho_{s,t}(y_1,\dots,y_q)$.

\begin{enumerate}[label=(\roman*), itemsep=0.6em]
    \item A direct implementation of \eqref{eq:mathcalE}, respectively of the
    $\rho$-dependent local block $\mathcal E^\rho$, requires, for each
    $n=1,\dots,N$, the evaluation and summation of $q^n$ many $n$-fold tensor
    products in $(\RR^m)^{\otimes n}$. Each such tensor product has cost
    $\Theta(m^n)$ in a dense basis, giving a total contribution of order
    $\Theta(q^N m^N)$. Forming the truncated product with ${\bv}$ costs
    $\Theta(Nm^N)$ by \Cref{sec:combinatorial_estimates:tensors}. Thus one
    local evaluation costs $\Theta(m^N(q^N+N))$, and multiplication by the
    $\Theta(J^2)$ local evaluations yields the claimed bound.

    \item Assume \Cref{def:symmetric_cK}. Then the local evaluation
    ${\bv}\otimes_N\mathcal E^\rho$ is computed by the shuffle-recursive scheme
    \Cref{alg:quad_evalVtE}. The exponent $\rho$ changes only the scalar
    coefficient tensor, not the tensor operations. Hence the same cost estimate
    as for the undecorated case applies. Namely, the dominant operations are
    shuffle-by-vector updates of the form $\bx\shuffle y_r$, whose cost is
    $\Theta((k+1)m^{k+1})$ when $\bx\in(\RR^m)^{\otimes k}$; cf.\
    \Cref{sec:combinatorial_estimates:shuffle}. Referring to Step~5 in
    \Cref{alg:quad_evalVtE}, and fixing $p\in\cA_q$, at level $|\ell|=n$ the
    number of multi-indices is
    \[
        \big\vert\{\ell\in\NN^q:\ |\ell|=n\}\big\vert
        =
        {n+q-1\choose q-1}.
    \]
    For each such $\ell$, the recursion performs $q$ shuffle-by-vector updates
    $E_p(\ell+1_r)\shuffle y_r$, $r=1,\dots,q$. Moreover, for $|\ell|=n$ the
    tensors $E_p(\ell+1_r)$ have maximal degree $N-2-n$, so summing over
    degrees gives
    \[
        \sum_{k=0}^{N-2-n} (k+1)m^{k+1}.
    \]
    Hence, summing over $n=0,\dots,N-2$ and over $p\in\cA_q$, we obtain
    \[
        q\sum_{n=0}^{N-2}{n+q-1\choose q-1}\,
        q\sum_{k=0}^{N-2-n} (k+1)m^{k+1},
    \]
    which is of order $\Theta(Nm^N)$ for fixed $(q,m)$ by
    \Cref{sec:combinatorial_estimates:asymptotics}. Forming the final
    truncated tensor product costs another $\Theta(Nm^N)$ by
    \Cref{sec:combinatorial_estimates:tensors}. Therefore one local evaluation
    costs $\Theta(Nm^N)$, and multiplication by $\Theta(J^2)$ yields the stated
    bound.

    \item If $q=1$, the local evaluation is computed via the Horner scheme
    \Cref{alg:quad_evalVtE_q1}. The exponent $\rho$ again changes only the
    scalar coefficients. The Horner recursion computes each tensor level by
    successive concatenations with $y$. A concatenation at tensor degree $n$ has
    cost $\Theta(m^n)$ by \Cref{sec:combinatorial_estimates:tensors}, and the
    levelwise recursion performs $\Theta(n)$ such steps to produce
    $\pi_n({\bv}\otimes_N\mathcal E^\rho)$. Summing over $n=1,\dots,N$ gives
    cost $\Theta(m^N)$ for one local evaluation. Multiplying by $\Theta(J^2)$
    gives the claimed bound.

    \item Under \Cref{hyp:fft_convolution}, the sums over previous cells in
    \Cref{alg:higher_order_vsig} become causal convolutions as described in
    \Cref{sec:fft}. In \Cref{alg:fft}, each such convolution is evaluated by a
    zero-padded FFT and inverse FFT of length $L=\Theta(J)$. By the standard
    FFT complexity \cite{cooley1965algorithm}, this costs
    $\Theta(J\log J)$ operations per scalar coordinate.

    At tensor level $n$, in the scalar case $q=1$, the convolved sequences are
    \[
        G_i^{n,r,\rho}
        =
        \pi_{n-r}\mathbf C_{i,\rho}\otimes y_i^{\otimes r},
        \qquad r=1,\dots,n.
    \]
    For $q>1$, the scalar source $y_i^{\otimes r}$ is replaced by the
    shuffle-polynomial terms $M_\ell(y_i)\otimes y_i^p$ from
    \Cref{sec:quad_conv}, where $p=1,\dots,q$ and
    $\ell\in\NN^q$ satisfies $|\ell|=r-1$. Hence, for fixed local order $r$,
    the number of such terms is
    \[
        q\bigl|\{\ell\in\NN^q:\ |\ell|=r-1\}\bigr|
        =
        q{r+q-2\choose q-1}.
    \]
    Summing over $r=1,\dots,n$ gives
    \[
        q\sum_{r=1}^{n}{r+q-2\choose q-1}
        =
        q{n+q-1\choose q}
        =
        \Theta(n^q),
    \]
    where $q$ is fixed. Each convolved sequence takes values in
    $(\RR^m)^{\otimes n}$ and therefore has $m^n$ dense coordinates. Thus the
    cost at tensor level $n$ is
    $
        \Theta\bigl(J\log(J)\,n^q\,m^n\bigr).
    $
    Summing over $n=1,\dots,N$ and using $m>1$ yields
    \[
        \Theta\bigl(J\log(J)\,N^q\,m^N\bigr).
    \]
    The precomputation of the FFTs of the lag-weight sequences has the same
    polynomial factor in $N$ but no dense tensor-coordinate factor $m^N$, and
    is therefore subleading. The componentwise interpolation solves contribute
    only $\Theta(Jm^N)$ for fixed $\decoNumPlus$ and are also subleading. This
    proves the FFT cost estimate.
\end{enumerate}
The proof is complete.
\end{proof}
The cost estimate in case~(ii) of the proof above relies on the following combinatorial lemma, which establishes the asymptotic order of the double sum arising from the shuffle-recursive scheme.

\begin{lem}\label{sec:combinatorial_estimates:asymptotics}
Fix $q,m\in\NN$ with $m>1$. Then, as $N\to\infty$,
$$
\sum_{n=0}^{N-2}{n+q-1\choose q-1}\sum_{k=0}^{N-2-n} (k+1)m^{k+1}
\;\sim\; N m^{N}.
$$
\end{lem}
\begin{proof}
We first verify (by elementary entire series arguments) that for the inner sum we have the two-sided estimate
$$
(N-1-n)m^{N-1-n} \le \sum_{k=0}^{N-2-n} (k+1)m^{k+1}
\;\le\;
\frac{m}{m-1}(N-1-n)m^{N-1-n},
$$
for all $n\in\{0,\dots,N-2\}$. Hence,
\begin{equation}\label{eq:comb_lem_interm}
    \sum_{n=0}^{N-2}{n+q-1\choose q-1}\sum_{k=0}^{N-2-n} (k+1)m^{k+1}
\;\sim\; \sum_{n=0}^{N-2}{n+q-1\choose q-1}(N-1-n)m^{N-1-n}.
\end{equation}
Using $m^{N-1-n}=m^{N-1}m^{-n}$ and the binomial series (\cite[\S 3.6.8]{abramowitz1974handbook})
$$
\sum_{n=0}^\infty {n+q-1\choose q-1} m^{-n}
=
\left(1-\frac{1}{m}\right)^{-q}
<\infty,
$$
we observe that the right-hand side in \eqref{eq:comb_lem_interm} is dominated by the $n=0$ term and hence of the order $(N-1)m^{N-1}\sim N m^{N}$, which proves the claim.
\end{proof}

\subsection{Costs of the finite state space algorithm}\label{sec:compcostFSSK}
We next summarize the asymptotic computational costs of the finite state space recursion from \Cref{thm:algo_multiplicative}, implemented as in \Cref{alg:fssk_state_recursion} together with the readout \Cref{alg:fssk_vsig_readout}. In addition to $(J,N)$ we track the dependence on the state dimension $R$, and we distinguish between the general case and the scalar-kernel case $q=1$ with a Horner-type state update.

To start the computational discussion, we briefly record the costs of computing the coefficient families $(\mathbf{1}^{\top}\Psi,\Phi)$. This step is carried out by the quadrature scheme \Cref{alg:fssk_weight_quad_eval}. 

\begin{prop}\label{prop:cost_fssk_weights}
Fix $\delta>0$ and apply \Cref{alg:fssk_weight_quad_eval} with truncation level $N$ (and a fixed quadrature parameter) to compute the normalized coefficient families $\{\widehat\psi_\ell:|\ell|\le N\}$ and $\{\widehat\Phi_{p,\ell}:p=1,\dots,q,\ |\ell|\le N\}$.
Treating $q$ (and the quadrature parameter) as fixed, the computational costs have the following asymptotic dependence (in the sense of \Cref{def:cost_asymptotic_dependence}) with respect to $(R,N)$:
\begin{enumerate}[label=(\roman*), itemsep=0.4em]
    \item For a generic dense matrix $\Lambda$, if the shifted linear systems in lines~\ref{line:evalphipsi_quad_start}--\ref{line:evalphipsi_quad_end} of \Cref{alg:fssk_weight_quad_eval} are solved using dense linear algebra (e.g.\ LU factorization), then the costs are of order
    $$
        \Theta\big(R^3+R^2N^{q}\big).
    $$
    \item If $\Lambda$ is in real Jordan normal form (cf.\ \Cref{lem:prony_implies_real_jordan_scalar}), so that these shifted solves can be carried out in $\Theta(R^2)$ operations per quadrature node, then the costs are of order
    $$
        \Theta\big(R^2N^{q}\big).
    $$
\end{enumerate}
\end{prop}

\begin{rem}
Similarly, one may precompute the matrix exponentials $E(\delta)=e^{-\Lambda\delta}$ and the vectors $E(\delta)b_p$.
For a generic dense $\Lambda$, standard dense methods for evaluating $E(\delta)$ have cost $\Theta(R^3)$ per distinct $\delta$ (see, e.g., \cite{moler03}),
whereas under the real Jordan normal form assumption from \Cref{lem:prony_implies_real_jordan_scalar} these evaluations reduce to blockwise scalar formulas and do not affect the leading-order dependence on $(J,R,N)$ in the subsequent recursions.
\end{rem}
\begin{proof}[Proof of Proposition~\ref{prop:cost_fssk_weights}]
Fix $\delta>0$ and a quadrature parameter $m$.
In lines~\ref{line:evalphipsi_quad_start}--\ref{line:evalphipsi_quad_end} of \Cref{alg:fssk_weight_quad_eval}, for each $j=1,\dots,m$ we solve one linear system with
$(\zeta_j I+\delta\Lambda)^\top$ and $q$ further systems with $(\zeta_j I+\delta\Lambda)$.
In a dense implementation with a generic dense $\Lambda$, each such solve costs $\Theta(R^3)$ (e.g.\ via LU), hence (with $q$ and $m$ fixed) this stage contributes $\Theta(R^3)$.
If $\Lambda$ is in real Jordan normal form, the shifted solves can be carried out in $\Theta(R^2)$ per right-hand side, so this contribution reduces accordingly.

Next, for each multi-index $\ell\in\NN^q$ with $|\ell|\le N$, lines~\ref{line:evalphipsi_accumul_start}--\ref{line:evalphipsi_accumul_end} accumulate $\widehat\psi_\ell\in\RR^{1\times R}$ and
$\widehat\Phi_{p,\ell}\in\RR^{R\times R}$ by adding real parts of rank-one updates $u_{p,j}r_j^\top$.
For fixed $(q,m)$, this costs $\Theta(R^2)$ per $\ell$, and thus
$\Theta\!\big(R^2\,|\{\ell\in\NN^q:\ |\ell|\le N\}|\big)=\Theta(R^2N^q)$ in total.

Finally, lines~\ref{line:evalphipsi_accumul_start}--\ref{line:evalphipsi_accumul_end} perform $\Theta(R^2)$ arithmetic operations per multi-index $\ell$ in a dense representation,
hence $\Theta(R^2N^q)$ in total. Together with the $\Theta(R^3)$ (resp.\ $\Theta(R^2)$) shifted solves in lines~\ref{line:evalphipsi_quad_start}--\ref{line:evalphipsi_quad_end}
this yields $\Theta(R^3+R^2N^q)$ (resp.\ $\Theta(R^2N^q)$).
\end{proof}

Given these coefficients, we can now trace the costs of the state recursion and readout.

\begin{prop}\label{prop:cost_fssk}
We work in the setting of \Cref{thm:algo_multiplicative} and consider the computation
${\bf v}_j=\pi_{\le N}\VSig{x}{K_{A,b}^{\Lambda}}^{t_j}_{t_0,t_j}$ ($j=0,\dots,J$)
for a kernel $K_{A,b}^{\Lambda}$ as in \Cref{def:finite_state_space_kernels} with kernel data $(\Lambda,\{A_p,b_p\}_{p=1}^q)$ and a piecewise linear path $x$ on a grid $t_0<\cdots<t_J$.
Assume that for each interval length $\delta_j=t_j-t_{j-1}$ the normalized coefficient families $(E_j,\widehat\psi_j,\widehat\Phi_j)$ required in \Cref{alg:fssk_state_recursion} are precomputed (cf.~\Cref{alg:fssk_weight_quad_eval}, \Cref{prop:cost_fssk_weights}), and that $(q,m,d)$ are fixed with $m>1$.
Then the costs of \Cref{alg:fssk_state_recursion} together with applying the readout \Cref{alg:fssk_vsig_readout} for all $j=0,\dots,J$ have the following asymptotic dependence with respect to $(J,R,N)$ in the sense of \Cref{def:cost_asymptotic_dependence}:
\begin{enumerate}[label=(\roman*), itemsep=0.4em]
    \item If $q>1$ and $(\widehat f_j,\widehat{\mathcal G}_j^1,\dots,\widehat{\mathcal G}_j^q)$ is evaluated via \Cref{alg:fsskFG}, then the costs are of order
    $$
        \Theta\big(J\,R^2\,N\,m^N\big).
    $$
    \item If $q=1$ and the state update is evaluated via the Horner scheme \Cref{alg:fsskq1_horner_update}, then the costs are of order
    $$
        \Theta\big(J\,R^2\,m^N\big).
    $$
\end{enumerate}
\end{prop}

\begin{proof}
The recursion in \Cref{alg:fssk_state_recursion} consists of a single outer loop over $j=1,\dots,J$, so the overall cost is the per-step cost multiplied by $J$.
\begin{enumerate}[label=(\roman*), itemsep=0.4em]
\item Assume $q>1$ and $(\widehat f_j,\widehat{\mathcal G}_j^1,\dots,\widehat{\mathcal G}_j^q)$ is evaluated via \Cref{alg:fsskFG}.
The recursion in \Cref{alg:fsskFG} is the same shuffle-recursive evaluation as in the symmetric-weights case of the general approximative algorithm (cf.\ \Cref{alg:quad_evalVtE} and \Cref{prop:cost_quad}\ref{itm:shuffle_costs}), except that the propagated quantities now take values in $(\mathfrak R_N)^{1\times R}$ and $(\mathfrak R_N)^{R\times R}$, and the coefficients $\widehat\psi_\ell$ and $\widehat\Phi_{p,\ell}$ are matrix-valued.
Hence, each shuffle-by-vector update is performed entrywise and incurs an additional factor $R^2$ compared to the scalar case.
Using the shuffle-by-vector cost from \Cref{sec:combinatorial_estimates:shuffle} and the corresponding asymptotic summation from \Cref{sec:combinatorial_estimates:asymptotics}, we obtain a per-step cost of order $\Theta\big(R^2\,N\,m^N\big)$.

Given $(\widehat f_j,\widehat{\mathcal G}_j^1,\dots,\widehat{\mathcal G}_j^q)$, forming
${\bf B}_j=\widehat f_j+\sum_{l=1}^q {\bf Z}_{j-1}^l.\widehat{\mathcal G}_j^l$
involves $q$ matrix products of a $(\mathfrak R_N)^{1\times R}$ row with a $(\mathfrak R_N)^{R\times R}$ matrix.
By definition of the matrix product over $\mathfrak R_N$, each such product consists of $\Theta(R^2)$ scalar operations in $\mathfrak R_N$, and each scalar operation has cost $\Theta(N m^N)$ (cf.\ \Cref{sec:combinatorial_estimates:tensors}).
Thus, forming ${\bf B}_j$ is of order $\Theta\big(R^2\,N\,m^N\big)$.

Finally, each update
${\bf Z}_j^p={\bf Z}_{j-1}^p.E_j+{\bf B}_j\otimes y_j^p$
has the same dominant cost $\Theta(R^2 N m^N)$ from the vector--matrix product ${\bf Z}_{j-1}^p.E_j$ (the concatenation ${\bf B}_j\otimes y_j^p$ contributes only $\Theta(R N m^N)$ and is lower order in $R$).
Since $p=1,\dots,q$ with $q$ fixed, one outer step $j$ is therefore of order $\Theta(R^2 N m^N)$.
Multiplying by the $J$ steps yields overall costs of order $\Theta\big(J R^2 N m^N\big)$.
Applying the readout \Cref{alg:fssk_vsig_readout} for all $j$ contributes only lower-order terms in $(R,N)$ and therefore does not change the asymptotic order.

\item In the scalar-kernel case $q=1$, the state recursion is given by \Cref{alg:fsskq1_state_recursion} with the Horner-type update \Cref{alg:fsskq1_horner_update}.
As in the scalar Horner scheme for the general approximative algorithm (cf.\ \Cref{alg:quad_evalVtE_q1}), the key point is that each level is built by successive concatenations with $\Delta x$ without forming a full tensor polynomial first.
For fixed level $n$, the update performs $\Theta(n)$ concatenations and $\Theta(n)$ multiplications with $R\times R$ coefficient matrices; in a dense tensor basis this costs $\Theta(R^2 m^n)$.
Summing over $n=1,\dots,N$ yields $\Theta(R^2 m^N)$, since $m^n$ is dominated by the top level.
Hence each outer step $j$ is of order $\Theta(R^2 m^N)$ and, over $J$ steps, the overall costs are of order $\Theta\big(J\,R^2\,m^N\big)$.
As above, applying the readout \Cref{alg:fssk_vsig_readout} for all $j$ does not change the asymptotic order.
\end{enumerate}
This concludes the proof of our proposition.
\end{proof}

\subsection{Cost of the Volterra signature kernel algorithm}\label{sec:cost_sig_kernel}

For completeness we also discuss the computational costs of the predictor--corrector scheme
\Cref{alg:fd_scheme_antidiag_pc} for the computation of the Volterra signature kernel
associated with finite state space kernels.

\begin{prop}\label{prop:cost_sig_kernel_fd}
We work in the setting of \Cref{sec:sig-kernel} and consider the computation of
$\kappa_{J_s,J_t}\approx \eta_{J_s,J_t}$ by the finite-difference scheme
\Cref{alg:fd_scheme_antidiag_pc} for piecewise linear paths $x$ and $w$ on grids
$0=s_0<\cdots<s_{J_s}=S$ and $0=t_0<\cdots<t_{J_t}=T$ respectively.
Assume that the cell coefficients $\gamma_{i,j}$ and the transport matrices
\(
(E_i^s,\ P_i^s,\ Q_i^s, E_j^t,\ P_j^t,\ Q_j^t)
\)
are precomputed, and that all matrix-valued quantities are stored in a dense
representation in $\RR^{R\times R}$. Then the computational costs have
asymptotic dependence
\[
    \Theta\bigl(J_sJ_tR^3\bigr)
\]
with respect to $(J_s,J_t,R)$.
\end{prop}
\begin{proof}
The scheme in \Cref{alg:fd_scheme_antidiag_pc} updates one cell
$(i,j)$ for each $0\le i<J_s$ and $0\le j<J_t$. Hence the total number of
cell updates is $J_sJ_t$.

It remains to estimate the cost of one such update. The variables
$\bK_{i,j}$, $\mathbf{\Psi}_{i,j}$, $\mathbf{\Phi}_{i,j}$, the coefficients
$\gamma_{i,j}$, and the transport matrices are all dense $R\times R$ matrices.
Thus, matrix additions, scalar multiplications, and the evaluations of
$\eta=1+\mathbf 1^\top \bK\mathbf 1$ cost at most $\Theta(R^2)$ operations.
The dominant operations are the dense matrix--matrix products appearing in the
updates of $\mathbf{\Psi}$ and $\mathbf{\Phi}$, as well as in the applications
of
\(
    \cL(\bM)=\Lambda \bM\Lambda^\top .
\)
Each of these operations costs $\Theta(R^3)$ in a dense implementation.
Since the predictor and corrector stages contain only a fixed number of such
matrix products per cell, independently of $i,j,J_s,J_t$, one cell update costs
$\Theta(R^3)$. Multiplying by the $J_sJ_t$ cells gives the claimed total cost.
\end{proof}
\begin{rem}\label{rem:cost_sig_kernel_jordan}
As in \Cref{lem:prony_implies_real_jordan_scalar}, $-\Lambda$ may be parametrized in
real Jordan normal form, and our implementation applies the corresponding
transport matrices blockwise.
Denoting by $\rho_\Lambda$ the sum of squared Jordan block sizes, one cell update costs $\Theta(R\rho_\Lambda)$ instead of $\Theta(R^3)$.
Consequently, the total cost is
\[
    \Theta\bigl(J_sJ_tR\rho_\Lambda\bigr),
\]
which reduces to $\Theta(J_sJ_tR^2)$ when the Jordan block sizes remain
bounded as $R$ increases, for instance in the diagonal $\Lambda$ case.
\end{rem}

\bibliographystyle{plain}

\appendix
\addtocontents{toc}{\protect\setcounter{tocdepth}{1}}

\section{Numerical Validation}\label{sec:numerical_validation}

The purpose of this section is to provide numerical validation of the algorithms derived in this paper, both in terms of accuracy and computational cost.
We structure the discussion according to \Cref{sec:algo_quadratic,sec:algo_multiplicative,sec:sig-kernel}. Namely, we treat the general approximative algorithm, the exact finite state space algorithm, and the finite state space signature kernel algorithm in individual sections.
Below we detail first two basic components of the  common experimental setup.

\textbf{Sample paths.} In all numerical validations we require sample paths.
Their particular choice is not of importance for the discussion below; the only requirements are that they have non-trivial signature components and are normalized to allow for a consistent error analysis.
For reproducibility, we fix the definition here.

Specifically, we parametrize the derivative of $X:[0,1]\to\mathbb{R}^3$ in spherical coordinates by
\[
\dot X_t
=
\bigl(
\cos\theta(t)\sin\phi(t),
\sin\theta(t)\sin\phi(t),
\cos\phi(t)
\bigr),
\]
where
\begin{equation*}
    \begin{aligned}
        \theta(t)
        &=
        2\pi \nu t
        +
        a_\theta\sin(2\pi t+\varphi_\theta)
        +
        0.35a_\theta\sin(6\pi t+\varphi_{\theta,2}),\\
        \phi(t)
        &=
        \frac{\pi}{2}
        +
        a_\phi\sin(4\pi t+\varphi_\phi)
        +
        0.25a_\phi\sin(8\pi t+\varphi_{\phi,2}).
    \end{aligned}
\end{equation*}
The parameters
\(
(\nu, a_\theta, a_\phi, \varphi_\theta,\varphi_\phi,\varphi_{\theta,2},\varphi_{\phi,2})
\)
are sampled uniformly from
\[
[1.25,3.75]\times[0.2,1.0]\times[0.15,0.65]\times[0,2\pi]^4.
\]
We then construct numerical sample paths by cumulatively summing the sampled derivatives evaluated on a fine grid. This yields paths that are approximately parametrized at \emph{unit speed}. For the evaluation, these paths are subsequently sampled on coarser grids with mesh size $\Delta t$, and the corresponding piecewise linear interpolations are used as reference paths $\{x^{(i)}\}_{i=0}^M$.
Denote also $J = \Delta t^{-1}$.

To illustrate that these sample paths have non-trivial signatures, we record in \Cref{fig:signature-std} the level-wise mean norm and sample standard deviation of the factorially adjusted signature levels.

\begin{figure}[h]
    \centering
    \includegraphics[width=\textwidth]{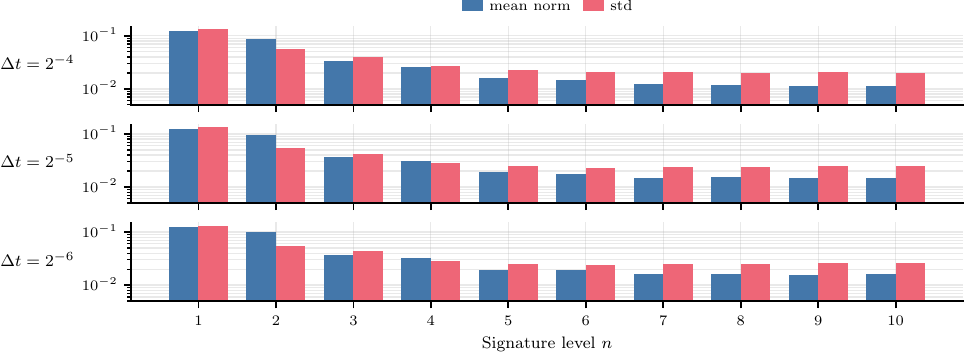}
    \caption{Level-wise sample standard deviation of the factorially adjusted signature levels $n!\,\pi_n\mathrm{Sig}(x^{(i)})$ for the generated sample paths.}
    \label{fig:signature-std}
\end{figure}

\textbf{Computational costs.} To allow for a hardware-independent validation of the computational costs of the proposed algorithms, we use the number of floating-point operations (FLOPs) %
as the main reference quantity.
Since our implementation uses the JAX backend, FLOP counts are obtained from JAX's compiler-level cost analysis. For each benchmarked function, we compile the corresponding JIT-compiled function for the relevant input shapes and query
\[
\texttt{jax.jit(f).lower(*args).compile().cost\_analysis()}.
\]
The resulting FLOP count is a compiler estimate for the optimized computation and excludes Python overhead, compilation time, and data-transfer costs.\footnote{For computations implemented using \texttt{jax.lax.scan}, the reported cost corresponds to the compiled loop body. We therefore manually multiply this count by the number of scan iterations.}
As a secondary, hardware-dependent test, we compare CPU with wall-clock times. This provides some of evaluation of the efficiency of the implementation and of the degree of parallelism achieved in practice on a CPU. All timing experiments were run on an Apple M3 CPU with total of 8 cores and 24GB of memory, using JAX 0.10.0 with the CPU backend.
While the JAX implementation should in principle compile efficiently on GPUs,
we have not tested this setting yet.

The code used for the numerical validations is collected in the
\texttt{notebooks} directory of the accompanying \texttt{tensordev} repository:
\url{https://github.com/hagerpa/tensordev/tree/main/notebooks}.
\subsection{Validation of the general approximative algorithm}

Since the algorithms from \Cref{sec:algo_quadratic} are approximative, we first
compare the error convergence of the different orders of the scheme with a
direct discretization of the fundamental Volterra equation
\[
    \VSig{x}{K}^{t}_{0,t}
    =
    1+
    \int_0^t
    \VSig{x}{K}^{s}_{0,s}
    \otimes K(t,s)\dd x_s .
\]
Treating this as a high-dimensional linear Volterra equation, a standard
reference method is the Adams-type predictor--corrector, or fractional
Adams--Bashforth--Moulton, product-integration scheme
\cite{diethelm2004detailed}. In our notation, write $\Delta x_i:=x_{t_{i+1}}-x_{t_i}$ and $h_i:=t_{i+1}-t_i$
For a scalar kernel $k$, the Adams predictor--corrector scheme is
\[
    \widehat{\bv}_j
    =
    1+
    \sum_{i=0}^{j-1}
    \omega^{\mathrm{E}}_{i,j}
    \bv_i
    \otimes_N
    \Delta x_i,
\]
and
\[
    \bv_j
    =
    1+
    \sum_{i=0}^{j-1}
    \left(
        \omega^{\mathrm{L}}_{i,j}\bv_i
        +
        \omega^{\mathrm{R}}_{i,j}\bv^{[j]}_{i+1}
    \right)
    \otimes_N
    \Delta x_i,
    \qquad
    \bv^{[j]}_{i+1}
    =
    \begin{cases}
        \bv_{i+1}, & i\leq j-2,\\
        \widehat{\bv}_j, & i=j-1.
    \end{cases}
\]
Here
\(    \omega^{\mathrm{E}}_{i,j} =
    \omega^{\mathrm{L}}_{i,j} + \omega^{\mathrm{R}}_{i,j}
\)
and
\[
    \omega^{\mathrm{L}}_{i,j}
    :=
    \frac{1}{h_i}
    \int_{t_i}^{t_{i+1}}
    \frac{t_{i+1}-s}{h_i}k(t_j,s)\dd s,
    \qquad
    \omega^{\mathrm{R}}_{i,j}
    :=
    \frac{1}{h_i}
    \int_{t_i}^{t_{i+1}}
    \frac{s-t_i}{h_i}k(t_j,s)\dd s .
\]
For uniform grids and convolution kernels, the weights depend only on the lag
$j-i$, so that the sums can be accelerated by FFT; see the implementation for
details.

For the convergence experiments we consider the scalar fractional kernel $k_\beta(t,s)
    =
    \frac{(t-s)^{\beta-1}}{\Gamma(\beta)}$, $
    \beta\in(0,1)$.
In view of \Cref{thm:frac_scheme}, and the higher-order scheme in
\Cref{alg:higher_order_vsig}, we use the exponent sets
\[
    \text{order }0:\; \decoSet=\{0\},
    \qquad
    \text{order }1:\; \decoSet=\{0,\beta,1\},\qquad
    \text{order }2:\; \decoSet=\{0,\beta,1,1+\beta,2\}.
\]
Since the computational task is to evaluate the Volterra signature of the
piecewise linearly interpolated input paths accurately, we compare the schemes
under dyadic refinement of these paths. For a dyadic refinement level
$\lambda$, let $x^{(i),\lambda}$ denote the dyadic refinement of the sample path
$x^{(i)}$, and write
\[
    \widehat{\bv}^{\mathrm{scheme},\lambda}(x^{(i)})
    \approx
    \pi_{\leq N}\VSig{x^{(i),\lambda}}{k_\beta}^{1}_{0,1}
\]
for the corresponding numerical approximation. We record the factorially
adjusted level errors
\[
    \delta V^{\mathrm{scheme}}_{\lambda}
    :=
    \max_{i=1, ..., N}\max_{i = 1, \dots, M}
    \left\|
    n!\,\pi_n
    \left(
        \widehat{\bv}^{\mathrm{scheme},\lambda}(x^{(i)})
        -
        \bv^{\mathrm{ref}}(x^{(i)})
    \right)
    \right\|_\infty,
    \qquad
    n=1,\ldots,N,
\]
where, as reference value $\bv^{\mathrm{ref}}(x^{(i)})$, we use a Richardson
extrapolation of the order-two scheme between dyadic refinement levels
$\lambda=4$ and $\lambda=5$, with truncation level $N=6$, applied to $M=16$ paths
discretized with $J=32$ steps.

\Cref{fig:general-vsig-convergence} shows the resulting convergence lines.
The left and right panels correspond to $\beta=0.6$ and $\beta=0.1$,
respectively. We observe that the order-two algorithm strongly outperforms the
predictor--corrector scheme in terms of error convergence.

Since the predictor--corrector scheme has lower cost for a fixed time
discretization\footnote{After all, the predictor corrector scheme does not built tensor powers but only multiplies first level increments.}, we also compare the error--runtime tradeoff in
\Cref{fig:general-vsig-runtime}. This gives the more decisive comparison: to
obtain comparable accuracy, the predictor--corrector scheme requires a much
larger dyadic refinement, and by the time it reaches the accuracy of the
proposed order-two scheme it is substantially more costly. The comparison of
computational times is meaningful in this setting because both algorithms were
implemented using the same JAX backend and tensor operations. Moreover, the
predictor--corrector scheme was evaluated with the FFT acceleration available
on uniform grids.

\begin{figure}[h]
    \centering
    \includegraphics[width=0.49\textwidth]{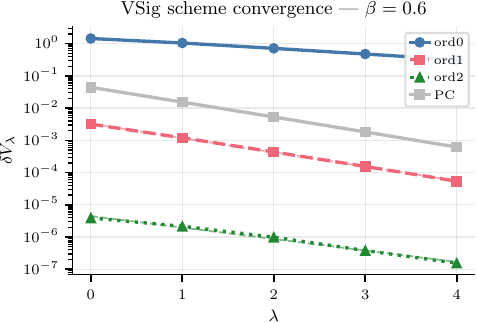}
    \includegraphics[width=0.49\textwidth]{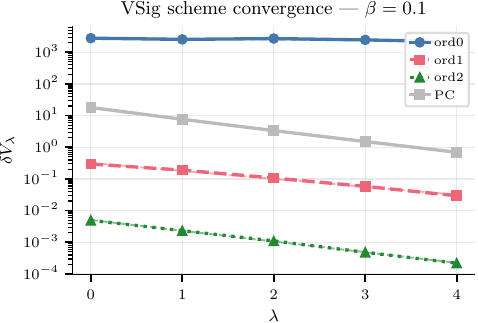}
\caption{Convergence of the general approximative Volterra signature
schemes under dyadic refinement. The plotted quantities are the
factorially adjusted level errors $\delta^{\mathrm{scheme}}_{n,\lambda}$.
Values in parentheses denote the fitted log--log slope of the error against
the dyadic grid size. Left: $\beta=0.6$. Right: $\beta=0.1$.}
    \label{fig:general-vsig-convergence}
\end{figure}

\begin{figure}[h]
    \centering
    \includegraphics[width=0.49\textwidth]{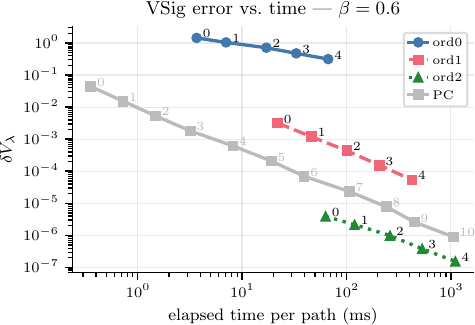}
    \includegraphics[width=0.49\textwidth]{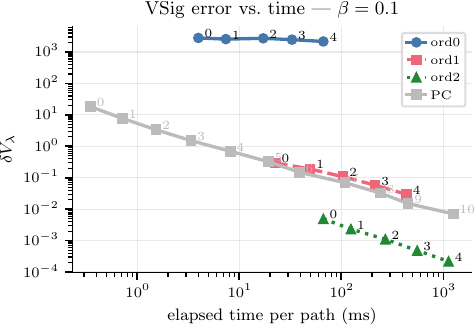}
    \caption{Error--runtime tradeoff for the predictor--corrector reference
    scheme and the proposed higher-order Volterra signature schemes.}
    \label{fig:general-vsig-runtime}
\end{figure}

Finally, we validate the computational scaling predicted by
\Cref{prop:cost_quad} by considering compiler-reported FLOP counts. Recall, that for the
quadratic triangular recursion from \Cref{alg:higher_order_vsig}, the predicted
asymptotic work is
\[
W_{\mathrm{quad}}(J,N)
=
\begin{cases}
J^2m^N, & q=1 \text{ with the Horner scheme of }
\Cref{alg:quad_evalVtE_q1},\\
J^2Nm^N, & q>1 \text{ with the shuffle-recursive scheme of }
\Cref{alg:quad_evalVtE}.
\end{cases}
\]
For the FFT-accelerated implementation from \Cref{alg:fft}, in the
uniform-grid convolutional setting of \Cref{hyp:fft_convolution}, the predicted
asymptotic work is
\[
    W_{\mathrm{FFT}}(J,N,q)
    =
    J\log(J)N^qm^N .
\]
We compile the corresponding implementations using various configurations up
to the benchmark workload filter
$W_{\mathrm{quad}}(J,N,4)\leq W_{\mathrm{quad}}(512,9,4)$, with fixed
$m=d=3$, $\beta=0.6$, and order-two exponent set
$\decoSet=\{0,\beta,1,1+\beta,2\}$. We sample $200$ admissible pairs $(J,N)$
satisfying this constraint and then cross them with all four values
$q\in\{1,2,3,4\}$, yielding $800$ benchmark configurations in total.

\Cref{fig:quad-fft-flop-scaling} shows the resulting FLOP counts. The left
panel reports the quadratic triangular recursion, including both the scalar
Horner branch and the multi-component shuffle-recursive branch. The right panel
reports the FFT-accelerated convolutional implementation. In both cases, the
observed direct proportionality of the FLOP counts with the predicted work supports the
asymptotic cost estimates in \Cref{prop:cost_quad}.
Although the total FLOP counts of the quadratic and FFT algorithms may appear
comparable in this benchmark regime, their ratio depends strongly on $J$ and
$N$. In agreement with the theoretical complexity analysis, numerical
experiments show that the FFT algorithm becomes preferable for large $J$,
whereas the quadratic algorithm is preferable for large $N$.

\begin{figure}[h]
    \centering
    \includegraphics[width=0.48\textwidth]{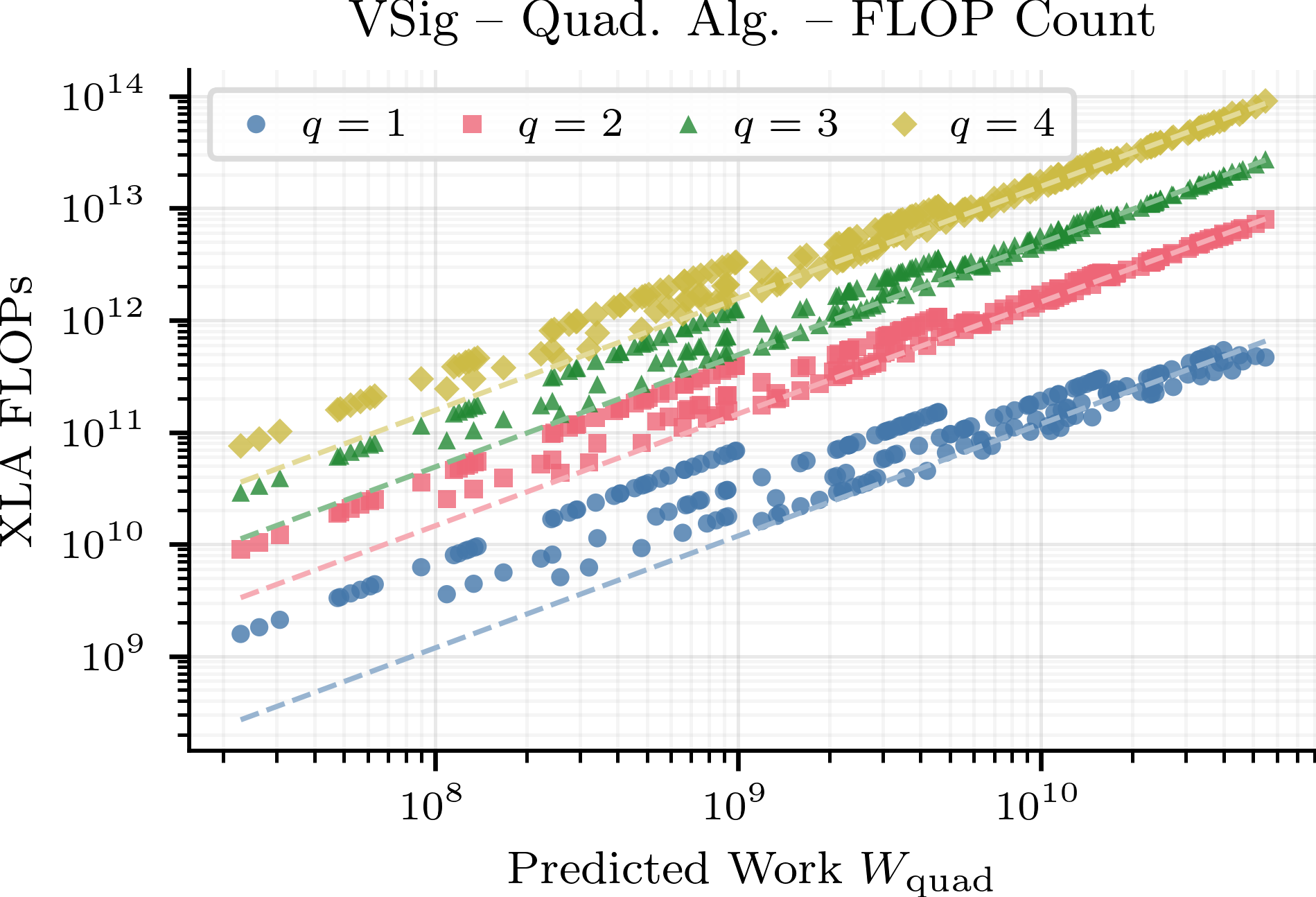}
    \hfill
    \includegraphics[width=0.48\textwidth]{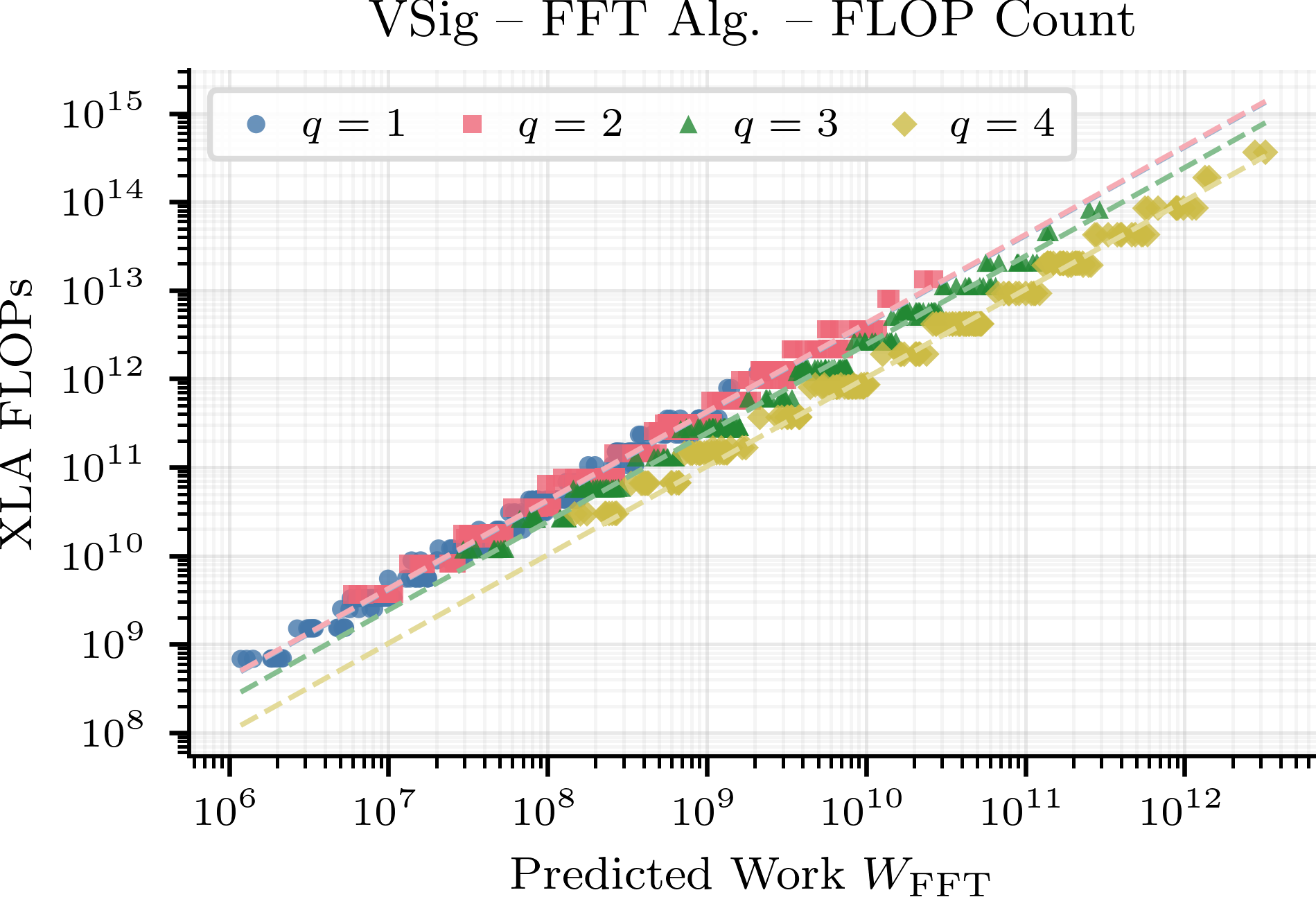}
    \caption{Computational scaling of the general approximative Volterra signature
algorithms. Left: compiler-reported FLOP counts for the quadratic triangular
recursion plotted against $W_{\mathrm{quad}}(J,N,q)$. Right:
compiler-reported FLOP counts for the FFT-accelerated implementation plotted
against $W_{\mathrm{FFT}}(J,N,q)$. Dashed lines indicate per-$q$ unit-slope
intercepts fits against largest $40$ workloads.}
    \label{fig:quad-fft-flop-scaling}
\end{figure}

\subsection{Validation of the finite-state-space algorithm}

Our first sanity check of the implementation was to compare the Volterra signature of a finite-state-space kernel $K^\Lambda_{A,b}$, see \Cref{def:finite_state_space_kernels}, with $\Lambda=0$ to the classical signature of the transformed path
\[
y_{t_j} = \sum_{p=1}^q \left(\sum_{r=1}^R b_{p,r}\right) A_p x_{t_j}, \qquad j=1, \dots, J.
\]
Comparing $\pi_{\leq N}\VSig{K^{\Lambda}_{A,b}}{x}$ with $\pi_{\leq N}\mathrm{Sig}(y)$ resulted in equality up to machine precision in all tested examples, with $N=10$ and various shapes of $A$ and $b$.

Secondly, we verified our implementation of \Cref{alg:fssk_state_recursion,alg:fssk_vsig_readout} against a dyadically refined explicit Euler scheme for the tensor-algebra-valued ODE satisfied by the state lift $\bZ$ of the Volterra signature, see \cite[Proposition 2.42]{i_part}. Each interval $[t_j,t_{j+1}]$ is split into $2^p$ equal subintervals, with both $\Delta t_j$ and $\Delta x_j^{(i)}$ divided by $2^p$. Initializing $\widehat{Z}_{0,0}^{\ell,i}=0\in\mathfrak{R}_N$, we iterate over the refined grid using
\[
\widehat{Z}_{j,k+1}^{\ell,i}
=
\widehat{Z}_{j,k}^{\ell,i}
-
\frac{t_{j+1}-t_j}{2^p}
\sum_{r=1}^R \Lambda_{\ell r}\widehat{Z}_{j,k}^{r,i}
+
\left(1+\sum_{r=1}^R \widehat{Z}_{j,k}^{r,i}\right)
\otimes_N
\left(
\sum_{\alpha=1}^q b_{\alpha,\ell}A_\alpha
\frac{\Delta x_j^{(i)}}{2^p}
\right),
\]
for $i=1,\dots,M$, $j=0,\ldots,J-1$, $k=0,\ldots,2^p-1$, and $\ell=1,\ldots,R$, with $\widehat{Z}^{\ell,i}_{j+1,0}=\widehat{Z}^{\ell,i}_{j,2^p}$. The terminal Volterra signature is read out as
\[
\widehat{V}_{i,p}
=
1+\sum_{\ell=1}^R \widehat{Z}_{J,0}^{\ell,i}.
\]
We compare this with the output of \Cref{alg:fssk_state_recursion,alg:fssk_vsig_readout}, which is exact up to quadrature errors in the Laplace approximation of the coefficients, and record the level-wise factorially rescaled Euler error
\[
\delta V_{n,p}
=
\max_{i=1, \dots, M}
\left\|
n!\,\pi_n
\left(
\widehat{V}_{i,p}
-
\VSig{K^\Lambda_{A,b}}{x^{(i)}}_{0,1}^1
\right)
\right\|_\infty,
\qquad n=1,\ldots,N.
\]
We tested this for several choices of $q$ and $R$ and observed the expected first-order convergence of the Euler scheme under dyadic refinement towards the exact benchmark, uniformly across tensor levels. 
In \Cref{fig:euler-conv-demo} we illustrate this convergence for the two representative setups $(q,R)=(1,1)$ and $(q,R)=(4,3)$. In particular, the results indicate that the Euler scheme requires a rather high degree of dyadic refinement in order to achieve reasonable accuracy in the Volterra signature components.

\begin{figure}[h]
    \centering
    \begin{subfigure}[t]{0.49\textwidth}
        \centering
        \includegraphics[width=\textwidth]{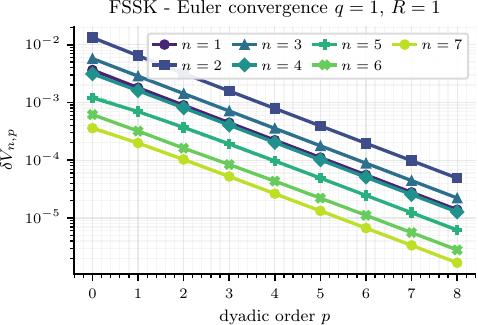}
    \end{subfigure}
    \hfill
    \begin{subfigure}[t]{0.49\textwidth}
        \centering
        \includegraphics[width=\textwidth]{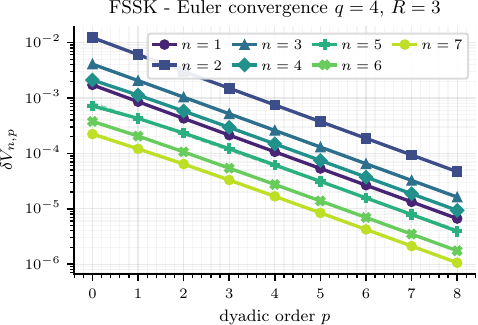}
    \end{subfigure}
    \caption{Convergence of the dyadically refined Euler scheme towards the exact benchmark Volterra signature for two representative setups.}
    \label{fig:euler-conv-demo}
\end{figure}

Finally, we analyze the computational cost of the Volterra signature computation
according to \Cref{alg:fssk_state_recursion,alg:fssk_vsig_readout}. We first
consider the state recursion and readout alone, excluding the precomputation of
the coefficients appearing in these algorithms; we comment on the cost of this coefficient
precomputation separately below. From \Cref{prop:cost_fssk}, we
expect the number of FLOPs to scale as
\[
W_q(J,R,N)
=
\begin{cases}
J\,R^2\,m^N, & q=1,\\
J\,R^2\,N\,m^N, & q>1.
\end{cases}
\]
We compile the implementation using various configurations up to the benchmark workload filter $W_q(J,R,N)\leq W_1(512,3;11)$, with fixed $m=d=3$. We sample $300$ admissible triples $(J,N,R)$ satisfying this constraint and then cross them with all four values of $q$. This yields $1200$ benchmark configurations in total. We then compare the compiler-reported FLOP counts with the predicted work $W_q(J,R,N)$.

As the left panel in \Cref{fig:fssk-flop-scaling} shows, the evaluated FLOP counts plotted against the predicted work lie approximately on straight lines. The dependence on $q$ appears mainly through the intercept, in agreement with the expected scaling. The right panel shows the measured elapsed time for each configuration $(J,R,N,q)$, averaged per path, plotted against the predicted work. As expected, this hardware-dependent quantity is less deterministic, but overall it exhibits a roughly proportional scaling with $W_q(J,R,N)$. The ratio of wall-clock time to CPU time ranges between $0.17$ and $1.00$, with a median of $0.37$, indicating an expected moderate degree of parallelism on the CPU setup used.

The coefficient precomputation was benchmarked separately, both for the dense
and the Jordan-structured implementations, and showed analogous agreement with
the theoretical scaling in \Cref{prop:cost_fssk_weights}. In the present
benchmark regime, it accounted for only $0.01\%$--$1.34\%$ of the total
Volterra signature workload, with the relative contribution increasing with
$q$.

\begin{figure}[h]
    \centering
    \includegraphics[width=0.48\textwidth]{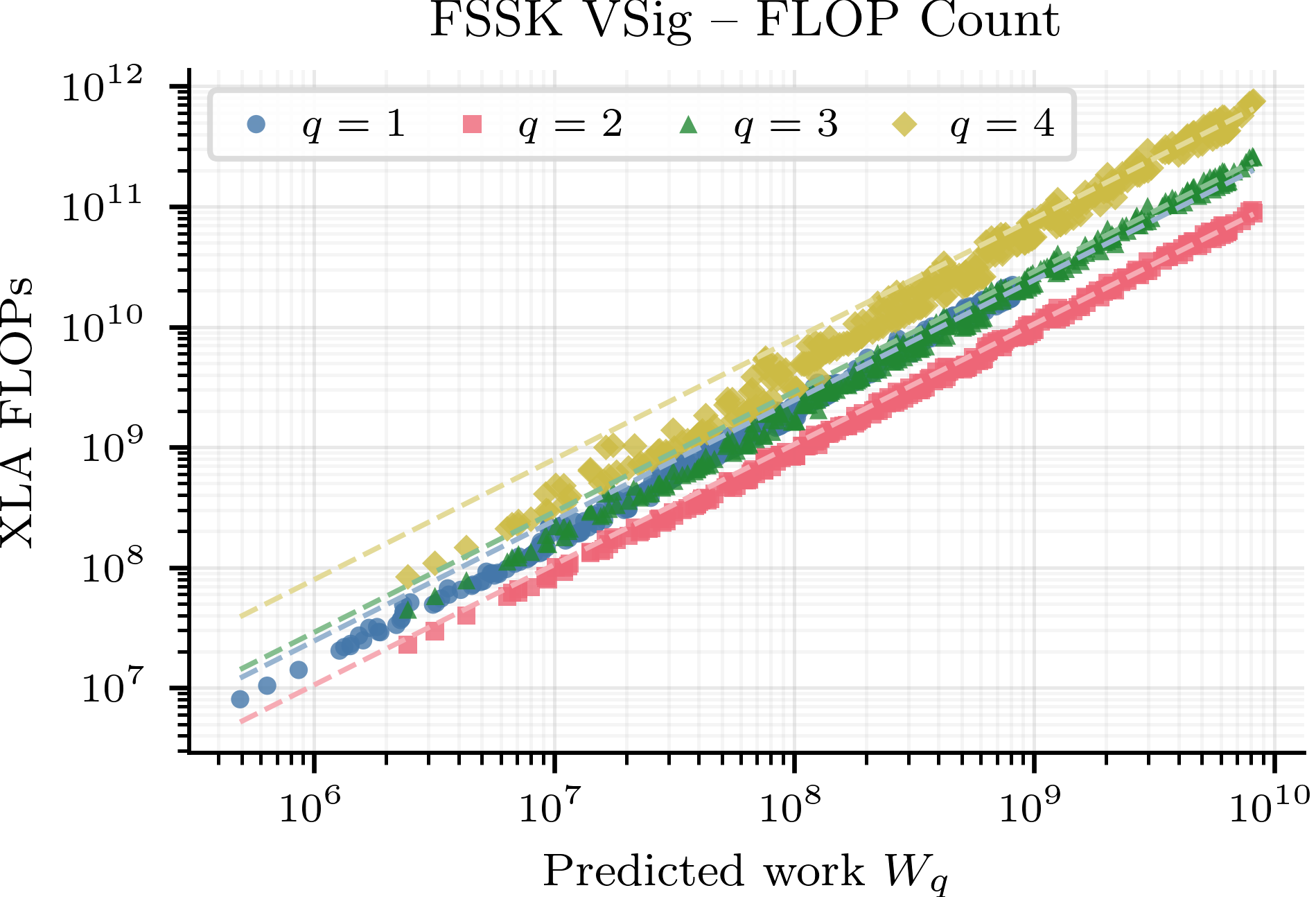}
    \hfill
    \includegraphics[width=0.48\textwidth]{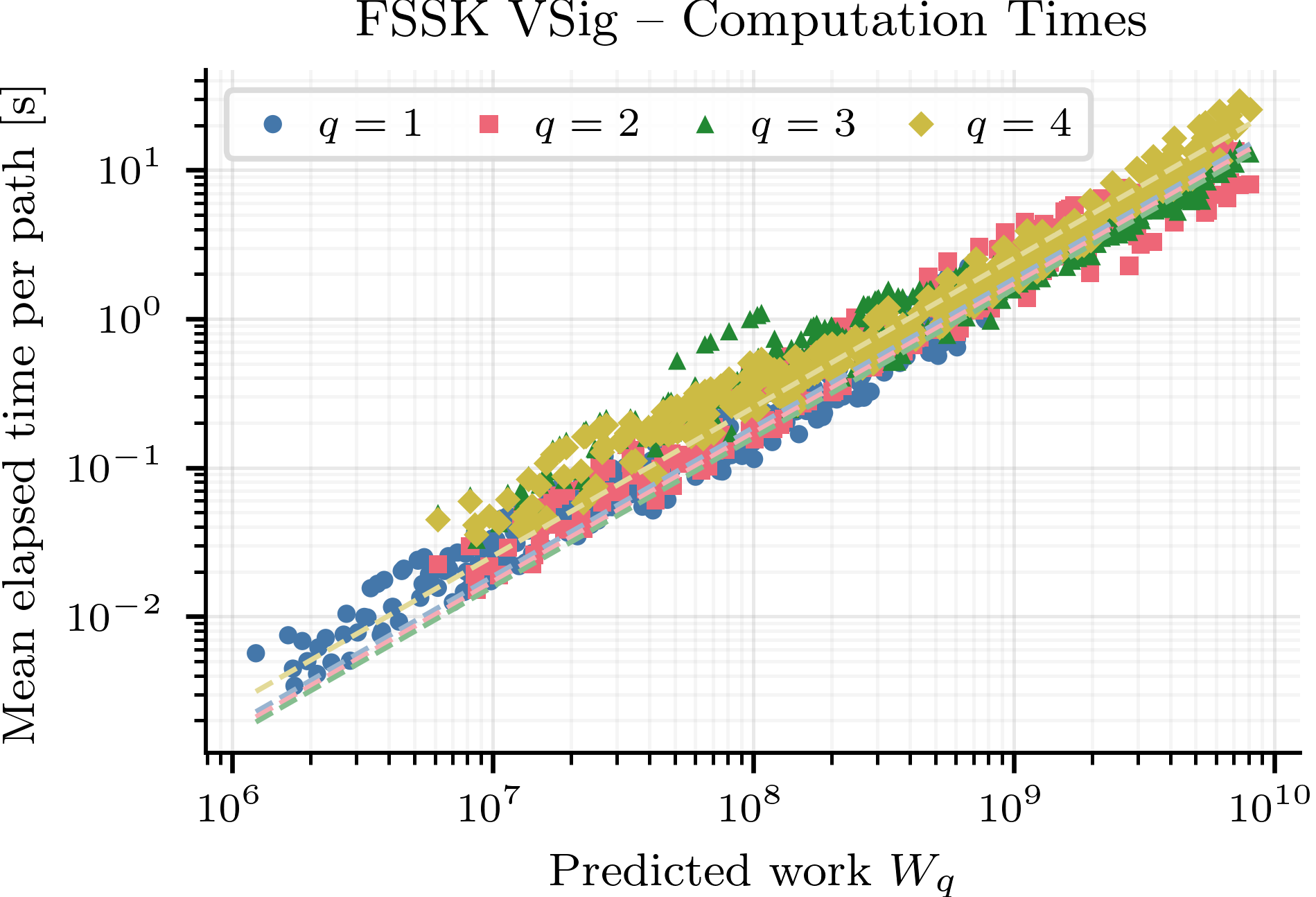}
    \caption{Computational scaling of the finite-state-space Volterra signature computation. Left: compiler-reported FLOP counts against the predicted work $W_q(J,R,N)$. Right: measured wall-clock time per path against $W_q(J,R,N)$. Dashed lines indicate per-$q$ unit-slope
intercepts fits against largest $40$ workloads.}
    \label{fig:fssk-flop-scaling}
\end{figure}

\subsection{Validation of the signature kernel algorithm}

We finally validate the finite-difference scheme from \Cref{alg:fd_scheme_antidiag_pc} for the computation of the Volterra signature kernel associated with finite-state-space kernels. 
We first compare its accuracy against a truncated inner-product reference obtained from the finite-state-space Volterra signature recursion of \Cref{alg:fssk_state_recursion,alg:fssk_vsig_readout}. More precisely, for a truncation level $N$ we set
\[
\kappa^{\mathrm{ref},N}(x,w)
:=
\left\langle
\pi_{\leq N}\VSig{x}{K^\Lambda_{A,b}}_{0,1}^{1},
\pi_{\leq N}\VSig{w}{K^\Lambda_{A,b}}_{0,1}^{1}
\right\rangle.
\]
In the experiment below we take $N=10$. 

We compare the predictor--corrector scheme of \Cref{alg:fd_scheme_antidiag_pc} with two simpler finite-difference schemes. The first is a direct ``\emph{naive}'' discretization of \eqref{eq:goursat-general-K}--\eqref{eq:goursat-general-phi}. With the notation of \Cref{sec:sig-kernel}, it uses
\begin{align*}
\mathbf{\Psi}_{i+1,j+1}
&=
\mathbf{\Psi}_{i,j+1}
-
\Delta s_i\Lambda\mathbf{\Psi}_{i,j+1}
+
\gamma_{i,j}\eta_{i,j+1},
\\
\mathbf{\Phi}_{i+1,j+1}
&=
\mathbf{\Phi}_{i+1,j}
-
\Delta t_j\mathbf{\Phi}_{i+1,j}\Lambda^\top
+
\gamma_{i,j}\eta_{i+1,j},
\\
\bK_{i+1,j+1}
&=
\bK_{i+1,j}
+
\bK_{i,j+1}
-
\bK_{i,j}
+
\frac{\Delta s_i\Delta t_j}{2}
\bigl(
\cL(\bK_{i+1,j})
+
\cL(\bK_{i,j+1})
\bigr)
\\ &\qquad
-
\frac12\gamma_{i,j}
\bigl(
\eta_{i+1,j}
+
\eta_{i,j+1}
\bigr)
+
\bigl(\mathbf{\Psi}_{i+1,j+1}-\mathbf{\Psi}_{i,j+1}\bigr)
+
\bigl(\mathbf{\Phi}_{i+1,j+1}-\mathbf{\Phi}_{i+1,j}\bigr),
\end{align*}
where $\cL(\bM)=\Lambda\bM\Lambda^\top$ and $\eta_{i,j}=1+\mathbf 1^\top\bK_{i,j}\mathbf 1$. 
The second alternative scheme uses \emph{exponential integration} for the auxiliary equations, but keeps the same first-order cell stencil for the main equation. Namely,
\begin{align*}
\mathbf{\Psi}_{i+1,j+1}
&=
E_i^s\mathbf{\Psi}_{i,j+1}
+
P_i^s\gamma_{i,j}\eta_{i,j+1},
\\
\mathbf{\Phi}_{i+1,j+1}
&=
\mathbf{\Phi}_{i+1,j}E_j^t
+
\gamma_{i,j}\eta_{i+1,j}P_j^t,
\\
\bK_{i+1,j+1}
&=
\bK_{i+1,j}
+
\bK_{i,j+1}
-
\bK_{i,j}
+
\frac{\Delta s_i\Delta t_j}{2}
\bigl(\cL(\bK_{i+1,j})+\cL(\bK_{i,j+1})\bigr)
\\ &\qquad
-
\frac12\gamma_{i,j}\bigl(\eta_{i+1,j}+\eta_{i,j+1}\bigr)
+
\bigl(\mathbf{\Psi}_{i+1,j+1}-\mathbf{\Psi}_{i,j+1}\bigr)
+
\bigl(\mathbf{\Phi}_{i+1,j+1}-\mathbf{\Phi}_{i+1,j}\bigr),
\end{align*}
where $(E_i^s,E_j^t,P_i^s,P_j^t)$ are defined as in
\Cref{alg:fd_scheme_antidiag_pc}.

For dyadic refinement level $\lambda$, we denote the resulting terminal kernel value by $\kappa^{\mathrm{scheme}}_\lambda(x,w)$ and record the error
\[
\delta\kappa_\lambda
:=
\max_{i=1, \dots, M}
\left|
\kappa^{\mathrm{scheme}}_\lambda(x^{(i)},w^{(i)})
-
\kappa^{\mathrm{ref},N}(x^{(i)},w^{(i)})
\right|.
\]
The left panel of \Cref{fig:sigkernel-validation} shows the resulting convergence curves. The predictor--corrector scheme converges substantially faster than the two first-order alternatives, which is consistent with the second-order correction built into \Cref{alg:fd_scheme_antidiag_pc}.

We also verify the computational scaling predicted by
\Cref{prop:cost_sig_kernel_fd}. In our implementation,
$-\Lambda$ is parametrized in real Jordan normal form and the maximal Jordan
block size is fixed at two. Hence \Cref{rem:cost_sig_kernel_jordan} predicts a
leading solver cost of order $J^2R^2$ when $J_s=J_t=J$, assuming that the cell
coefficients $\gamma_{i,j}$ are already available.

To demonstrate additionally that the precomputation contributes only marginally to overall costs, the FLOP count reported in
the right panel of \Cref{fig:sigkernel-validation} includes the computation of $\gamma$.
The experiment ranges over $m=d=3$, $q\in\{1,\dots,5\}$,
$R\in\{6,\dots,13\}$, and grid sizes
$J\in\{2^8,2^8+1,\dots,2^{12}\}$, subject to the cutoff
$J^2R^2 \leq 512^2\cdot 10^2.$
The right panel of \Cref{fig:sigkernel-validation} plots
the resulting FLOP counts against the predicted asymptotic work $J^2R^2$. The guide line has
the theoretical unit slope in log--log coordinates, with only the intercept
fitted by least squares. The observed alignment supports the asymptotic
solver-cost estimate from
\Cref{prop:cost_sig_kernel_fd,rem:cost_sig_kernel_jordan}.
The CPU-to-wall-time ratio was consistently around $0.5$, indicating limited
parallel utilization on the tested CPU setup. We expect this ratio to improve
substantially in a GPU setting.

The small regular offsets in the plot are caused by the coefficient
precomputation. In the implementation, forming $\gamma_{i,j}$ has leading
dependence
$
    R^2q + q^2R
$
per grid cell, corresponding to the two matrix products involved in the
coefficient construction.

\begin{figure}[h]
    \centering
    \includegraphics[width=0.48\textwidth]{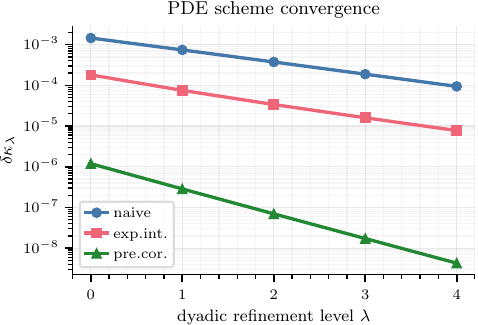}
    \hfill
    \includegraphics[width=0.48\textwidth]{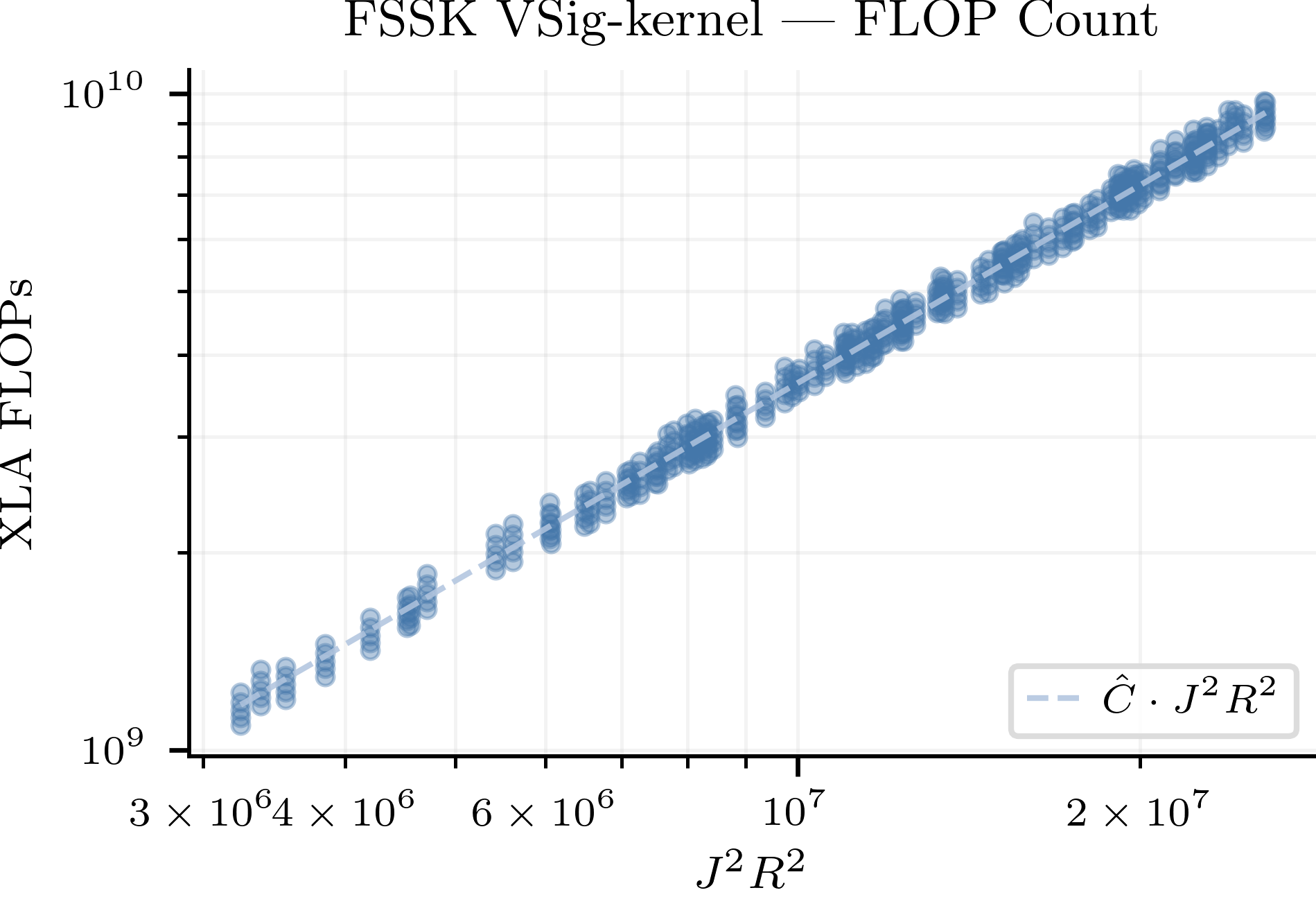}
    \caption{Validation of the Volterra signature kernel algorithm. Left:
    convergence of the naive, exponential integration, and
    predictor--corrector schemes against the truncated inner-product reference
    $\kappa^{\mathrm{ref},N}$. Right: compiler-reported total FLOP counts
    plotted against the leading asymptotic work
    $J^2R^2$.}
    \label{fig:sigkernel-validation}
\end{figure}

\end{document}